\documentclass[11pt]{amsart}

\usepackage{amssymb,mathrsfs,color,colortbl,booktabs,graphicx,comment,ytableau,tikz,enumerate,boldline}
\usepackage[all,cmtip]{xy}
\usepackage{diagxy}
\usepackage[labelsep=space]{caption}

\newtheorem{thrm}{Theorem}[section]
\newtheorem{prop}[thrm]{Proposition}
\newtheorem{lem}[thrm]{Lemma}
\newtheorem{cor}[thrm]{Corollary}
\newtheorem{conj}[thrm]{Conjecture}

\numberwithin{equation}{section}
\theoremstyle{remark}
\newtheorem{remark}[thrm]{Remark}

\theoremstyle{definition}
\newtheorem{definition}[thrm]{Definition}
\newtheorem{example}[thrm]{Example}

\newcommand{\g}{\mathcal{G}}
\newcommand{\s}{\mathcal{S}}

\newcommand{\h}{\mathcal{H}}
\newcommand{\col}{\mathcal{C}}
\newcommand{\row}{\mathcal{R}}

\newcommand{\gs}{\g(\s_{n})}
\newcommand{\gsz}{\g(\s^{0}_{n})}

\newcommand{\Ds}{\mathcal{D}_{\s}}         
\newcommand{\Dh}{\mathcal{D}_{\h}}
\newcommand{\Dsz}{\mathcal{D}^{0}_{\s}}            
\newcommand{\Dhz}{\mathcal{D}^{0}_{\h}}
\newcommand{\As}{\mathcal{A}_{\s}}
\newcommand{\Ah}{\mathcal{A}_{\h}}

\newcommand{\F}{\mathbb{F}}
\newcommand{\C}{\mathbb{C}}
\newcommand{\p}{\mathcal{P}}
\newcommand{\pr}{\mathcal{P}^{\textnormal{$e$-reg}}}
\newcommand{\ps}{\mathcal{P}^{\textnormal{$e$-sing}}}

\newcommand{\la}{\lambda}

\newcommand{\domr}{\vartriangleright}

\newcommand{\adj}{\textnormal{adj}}

\newcommand{\uar}{{\uparrow}}
\newcommand{\delf}{W_{\mathbb{F}}}
\newcommand{\LF}{L_{\mathbb{F}}}

\newdimen\abax \abax=0cm
\newdimen\abay \abay=0cm
\newdimen\abah \abah=1cm
\newdimen\abav \abav=.85cm
\newdimen\abahd \abahd=1.6cm
\newdimen\abavd \abavd=1.6cm
\newdimen\abadts \abadts=.9cm
\newdimen\abab \abab=.8cm
\newdimen\aban \aban=.3cm
\newdimen\abat \abat=.3cm

\newdimen\abau \abau=.5cm

\def\abas{.37}
\def\drawbead{\draw(\abax,\abay-.1cm)--++(0,\abav+.2cm);\shade[shading=ball,ball color=black](\abax,\abay+.5*\abav)circle(.5*\abab);\advance\abax\abah}
\def\drawnobead{\draw(\abax,\abay-.1cm)--++(0,\abav+.2cm);\draw(\abax-.5*\aban,\abay+.5*\abav)--++(\aban,0);\advance\abax\abah}

\def\drawminus{\draw(\abax-.5*\abau,\abay+.5*\abav)--++(\abau	,0);\advance\abax\abah}
\def\drawplus{\draw(\abax,\abay+.1cm)--++(0,\aban+.3cm);\draw(\abax-.5*\abau,\abay+.5*\abav)--++(\abau,0);\advance\abax\abah}
\def\drawtopleft{\draw(\abax,\abay-.1cm)--++(0,\abat+.1cm)--++(.5*\abah+.1cm,0);\advance\abax\abah}
\def\drawnormal{\draw[red](\abax-.5*\abau,\abay+.5*\abav)--++(\abau	,0);\advance\abax\abah}
\def\drawconormal{\draw[red](\abax,\abay+.1cm)--++(0,\aban+.3cm);\draw[red](\abax-.5*\abau,\abay+.5*\abav)--++(\abau,0);\advance\abax\abah}
\def\drawtilde{$\sim$;\advance\abax\abah}

\def\drawvdots{\draw[dotted,thick](\abax,\abay+.5*\abavd-.5*\abadts)--++(0,\abadts);\advance\abax\abah}
\def\drawhdots{\draw[dotted,thick](\abax+.5*\abahd-.5*\abah-.5*\abadts,\abay+.5*\abav)--++(\abadts,0);\advance\abax\abahd}
\def\drawempty{\advance\abax\abah}
\def\drawtopright{\draw(\abax,\abay-.1cm)--++(0,\abat+.1cm)--++(-.5*\abah-.1cm,0);\advance\abax\abah}
\def\drawtopmiddle{\draw(\abax,\abay-.1cm)--++(0,\abat+.1cm)++(-.5*\abah-.1cm,0)--++(\abah+.2cm,0);\advance\abax\abah}
\def\drawtopdots{\draw[dotted,thick](\abax+.5*\abahd-.5*\abah-.5*\abadts,\abay+\abat)--++(\abadts,0);\advance\abax\abahd}
\def\@bacus#1#2.{%
\if b#1\drawbead\fi%
\if n#1\drawnobead\fi%
\if -#1\drawminus\fi%
\if +#1\drawplus\fi%
\if E#1\drawempty\fi%
\if C#1\drawconormal\fi%
\if N#1\drawnormal\fi%
\if s#1\drawtilde\fi%

\if ,#1\abax=0cm\advance\abay-\abav\fi%
\if ;#1\abax=0cm\advance\abay-\abavd\fi%
\if h#1\drawhdots\fi%
\if v#1\drawvdots\fi%
\if l#1\drawtopleft\fi%
\if m#1\drawtopmiddle\fi%
\if r#1\drawtopright\fi%
\if t#1\drawtopdots\fi%
\if_#1\advance\abax\abahd\fi%
\if\space #2\else\@bacus#2 .\fi%
}
\def\sabacus(#1,#2){\abax=0cm\abay=0cm\begin{tikzpicture}[scale=\abas*#1]\@bacus#2 .\end{tikzpicture}}
\def\abacus(#1){\sabacus(1,#1)}

\usepackage{geometry}
\usepackage{amssymb,mathrsfs,color,colortbl,booktabs,graphicx,comment,hyperref,enumerate,boldline,tikz}

\begin{document}
\title{James's Conjecture holds for blocks of $q$-Schur algebras of weights 3 and 4}
\author{Aaron Yi Rui Low}
\address{Department of Mathematics, National University of Singapore}
\email{aaronlyr94@gmail.com}
\date{28 January 2021}
%\urladdr{www.math.sc.edu/$\sim$howard} % Delete if not wanted.
\begin{abstract}
When the characteristic of the underlying field is at least 5, we prove that the adjustment matrix for blocks of $q$-Schur algebras of weights 3 and 4 is the identity matrix. Moreover, we show that the decomposition numbers for weight 3 blocks of $q$-Schur algebras are bounded above by one.
\end{abstract}

\maketitle
\pagestyle{plain}
%% LaTeX can automatically make a table of contents.  This is done by
%% uncommenting the following:
%\tableofcontents

\section{Introduction}
Given a positive integer $m$, let $\mathbb{N}^{m}_{>0}:=\{(a_{1},a_{2},\dots,a_{m}) \mid a_{i}\in\mathbb{N}_{>0} \forall 1\le i\le m  \}$ be the set of $m$-tuples of positive integers. We now fix an $n\in\mathbb{N}_{>0}$.
\begin{definition}
An $m$-tuple $\lambda=(\lambda_{1},\lambda_{2},\dots,\lambda_{m})\in\mathbb{N}^{m}_{>0}$ is a $composition$ of $n$ if $\sum\limits_{i=1}^{m}\lambda_{i}=n$; if moreover, $\lambda_{i}\ge\lambda_{i+1}$ $ \forall 1\le i \le m-1$, we say that $\lambda$ is a $partition$ of $n$. We refer to $\lambda_{i}$ as the \textit{parts} of $\lambda$ and denote the \textit{number of parts of} $\lambda$ by $l(\lambda):=m$. 
\end{definition}
We denote the set of all partitions of $n$ by $\mathcal{P}(n)$. 
\begin{definition}
Let $e \ge 2$ be an integer. We call a partition $\lambda$ \textit{e-singular} if there is an integer $1\le i \le l(\lambda)-e+1$ such that $\lambda_{i}=\lambda_{i+1}=\dots=\lambda_{i+e-1}\neq0$; we call $\lambda$ \textit{e-regular} otherwise. We denote the set of all $e$-singular partitions of $n$ by $\ps(n)$ and the set of all $e$-regular partitions of $n$ by $\pr(n)$.
\end{definition}

Suppose that $q$ is a non-zero element of a field $\mathbb{F}$. Unless mentioned explicitly, we assume that $\mathbb{F}$ can have any characteristic (including zero). Let $e$ be the least positive integer such that $1+q+\dots+q^{e-1}=0$, assuming throughout the paper that it exists. Let $n$ be a positive integer. Denote by $\s_{n}=\s_{q}(n,n)$ the \textit{q-Schur algebra} (over $\mathbb{F}$) defined in~\cite{dipper 2}. When $q=1$, this is just the classical Schur algebra over $\mathbb{F}$. To each partition $\la\in\p(n)$, we associate a \textit{Weyl module} $W^{\la}$. Each $W^{\la}$ has a simple head $L^{\la}$ which is self-dual with respect to the contravariant duality induced by the anti-automorphism of $\s_{n}$. Moreover, the set $\{L^{\la} \mid \la\in \p(n)\}$ is a complete set of mutually non-isomorphic simple modules of $\s_{n}$. Given any two partitions $\la, \mu \in \p(n)$, the composition multiplicities $[W^{\lambda}:L^{\mu}]$ are called the \textit{decomposition numbers} of $\s_{n}$. We record these in a \textit{decomposition matrix} with rows and columns indexed by $\p(n)$ and whose $(\lambda,\mu)$-entry is $[W^{\lambda}:L^{\mu}]$.

Let $\mathcal{H}_{n}=\mathcal{H}_{\mathbb{F},q}(\mathfrak{S}_{n})$ denote the \textit{Iwahori-Hecke algebra} of the symmetric group $\mathfrak{S}_{n}$. This is a `deformation' of the group algebra $\mathbb{F}\mathfrak{S}_{n}$; we refer the reader to~\cite{mathas book} for its definition. When $q=1$, this is simply the group algebra $\mathbb{F}\mathfrak{S}_{n}$. The representation theory of $\h_{n}$ is very similar to that of $\s_{n}$. To each partition $\lambda\in\p(n)$, we associate a \textit{Specht Module} $S^{\lambda}$ for $\mathcal{H}_{n}$. If $\lambda\in \pr(n)$, then $S^{\lambda}$ has a simple head $D^{\lambda}$. The set $\{D^{\lambda} \mid \la\in\pr(n)   \}$ is a complete set of mutually non-isomorphic simple $\mathcal{H}_{n}$-modules. The decomposition numbers of $\mathcal{H}_{n}$ are the composition multiplicities $[S^{\lambda}:D^{\mu}]$; the decomposition matrix for $\mathcal{H}_{n}$ has rows indexed by $\p(n)$ and columns indexed by $\pr(n)$, with $(\lambda,\mu)$-entry $[S^{\lambda}:D^{\mu}]$. 

\textbf{Warning:}
The Specht and Weyl modules, $S^{\la}$ and $W^{\la}$ defined in~\cite{mathas book} are isomorphic to the dual of the Specht and Weyl modules defined by Dipper and James in~\cite{dipper 1} indexed by $\la'$. We adopt the Specht and Weyl modules defined by Dipper and James in~\cite{dipper 1} rather than those in~\cite{mathas book}.

A central problem in the study of the representation theory of $\s_{n}$ and $\mathcal{H}_{n}$ is to determine their decomposition matrices. These two problems are closely related. In fact, there is an exact functor called the \textit{Schur functor} from the category of $\s_{n}$-modules to the category of $\mathcal{H}_{n}$-modules. Given $\mu\in\p(n)$, the Schur functor sends the Weyl module $W^{\mu}$ to the Specht module $S^{\mu}$. If $\mu\in\pr(n)$, the simple module $L^{\mu}$ is sent to the simple module $D^{\mu}$; otherwise if $\mu\in\ps(n)$, $L^{\mu}$ is sent to zero. Using the Schur functor, one can deduce the following:
 \begin{thrm}~\cite[Theorem 4.18]	{mathas book}\label{schur functor d}
Suppose that $\lambda$ and $\mu$ are partitions of $n$ and that $\mu$ is $e$-regular. Then, $$[W^{\lambda}:L^{\mu}]=[S^{\lambda}:D^{\mu}].$$
\end{thrm}
In other words, the decomposition matrix of $\mathcal{H}_{n}$ is a submatrix of the decomposition matrix of $\s_{n}$; obtained by deleting the columns indexed by $\ps(n)$. Lascoux, Leclerc and Thibon conjectured that their recursive LLT algorithm~\cite{llt} calculates the decomposition matrices of $\h_{n}$ over $\mathbb{C}$. This was later proved by Ariki in~\cite{Ariki}. It is known that decomposition matrices over fields of prime characteristic may be obtained from decomposition matrices over $\C$ by post-multiplying by an `adjustment matrix'. In view of this, we often study the adjustment matrices instead of the decomposition matrices directly when working over fields of prime characteristic. The complexity of the representation theory of a block of $\h_{n}$ or $\s_{n}$ is roughly captured by a measure called the weight of that block. James conjectured~\cite{glnq} that when the weight of a block is less than the characteristic $p>0$ of the underlying field, the adjustment matrix is the identity matrix; therefore the decomposition matrix is computable using the LLT algorithm in principle. In the case of $\h_{n}$, the conjecture has been proved for weights up to $4$ by the works of Richards~\cite{Richards} and Fayers~\cite{wt 3,wt 4}. The author~\cite{wt 5} has also proved that the adjustment matrix for the principal block of $\mathcal{H}_{5e}$ of weight 5 is the identity matrix when $\textnormal{char}(\mathbb{F})\ge5$ and $e\neq4$. For the $q$-Schur algebras, the weight 2 case was proved by Schroll and Tan~\cite{Ac duality}. However, Williamson found a counter-example~\cite{Williamson} to James's Conjecture. Nevertheless, the smallest counter-example produced in his paper occurs in the symmetric group $\mathfrak{S}_{n}$ where $n=1744860$. There is considerable interest in finding smaller counter-examples. 

This paper is structured in the following way. In section \ref{background}, we mention some fundamental results in the modular representation theory of $\h_{n}$ and $\s_{n}$. Our main object of interest, the adjustment matrices are introduced in section \ref{section adj}. We give an overview of the existing results on adjustment matrices and provide some new tools for studying them. In sections \ref{section wt 3} and \ref{section wt 4}, we apply the results from sections \ref{background} and \ref{section adj} to prove James's Conjecture for blocks of $\s_{n}$ of weights $3$ and $4$. We also show that the decomposition numbers for weight $3$ blocks of $\s_{n}$ are bounded above by one.

\section{Background}\label{background}
\subsection{Partitions and abacus displays}
Take an abacus with $e$ vertical runners, numbered $0,\dots,e-1$ from left to right, marking positions $0,1,\dots$ on the runners increasing from left to right along successive `rows'. Given $\la\in \p(n)$, take an integer $r\ge l(\lambda)$. Define 
\begin{equation*}
\beta_{i}=\begin{cases}
\lambda_{i}+r-i, & \textnormal{if } 1\le i \le l(\lambda),\\
r-i, &  \textnormal{if } r>l(\lambda).
\end{cases}
\end{equation*}
Now, place a bead at position $\beta_{i}$ for each $i$. The resulting configuration is an abacus display for $\lambda$ with $r$ beads. If a bead has been placed at position $j$, we say that the position $j$ is $occupied$. Otherwise, we say that the position $j$ is $unoccupied$ or $vacant$. We remark that moving a bead from position $a$ to an unoccupied position $b<a$ is akin to removing a rim hook of length $h=b-a$ from the Young diagram of $\lambda$. The \textit{leg-length} of the hook is given by the number of occupied positions between $a$ and $b$. By moving all the beads as high as possible on their runners, the resulting configuration is an abacus display for the \textit{e-core} of $\lambda$. The $relative$ $e$-$sign$ of $\lambda$, denoted by $\sigma_{e}(\lambda)$ is $(-1)^{t}$, where $t$ is the total leg lengths of \textit{e}-hooks removed to obtain the \textit{e}-core (See~\cite[\S2]{esign}). Note that although $t$ depends on the choice of $e$-hooks removed, its parity is constant. We define the \textit{e-weight} of a bead to be the number of vacant positions above it on the same runner. The $e$-weight of $\lambda$ is the sum of $e$-weights of all the beads in an abacus display for $\lambda$. Thus, if $\lambda$ has $e$-weight $w$, then its $e$-core is a partition of $n-ew$. If there is no ambiguity, we often just refer to $e$-weight as \textit{weight} and $e$-core as \textit{core}.

It is easy to read addable and removable nodes from an abacus display. Display $\lambda$ on an abacus with $e$ runners and $r$ beads, where $r\ge l(\lambda)$. Let $i$ be the residue class of $(j+r)$ modulo $e$. Then, the removable nodes of $\lambda$ with $e$-residue $j$ correspond to the beads on runner $i$ with unoccupied preceding positions, while the addable nodes of $e$-residue $j$ correspond to the vacant positions on runner $i$ with occupied preceding positions. We call the removable (resp. addable) nodes with $e$-residue $j$ \textit{j-removable} (resp. \textit{j-addable}). Removing a removable node corresponds to moving the corresponding bead to its preceding unoccupied position, while adding an addable node corresponds to moving the corresponding bead to its succeeding unoccupied position. We adopt the convention that position $0$ has an occupied preceding position.

\subsection{Blocks of $\s_{n}$ and $\mathcal{H}_{n}$}
\begin{thrm}~\cite[Theorem 5.37]{mathas book}\label{nakayama}
Let $\lambda$ and $\mu$ be partitions of $n$. Then, $W^{\lambda}$ and $W^{\mu}$ lie in the same block of $\s_{n}$ if and only if $\lambda$ and $\mu$ have the same $e$-core.
\end{thrm}
Applying the Schur functor to the theorem above, we get the following corollary.
\begin{cor}(Nakayama Conjecture)~\cite[Corollary 5.38]{mathas book}\label{nakayama hecke}
Let $\lambda$ and $\mu$ be partitions of $n$. Then, $S^{\lambda}$ and $S^{\mu}$ lie in the same block of $\mathcal{H}_{n}$ if and only if $\lambda$ and $\mu$ have the same $e$-core.
\end{cor}
Given a block $B$ of $\s_{n}$ or $\mathcal{H}_{n}$, we say that a partition $\lambda \in \p(n)$ lies in $B$ if $W^{\lambda}$ or $S^{\lambda}$ lies in $B$. Given the Nakayama Conjecture, we may define the $e$-weight and $e$-core of a block $B$ of $\s_{n}$ or $\mathcal{H}_{n}$ simply to be the $e$-weight and $e$-core of a partition lying in $B$. 

Let $\lambda(i)$ be the partition corresponding to the abacus display containing only a single runner, the $i^{th}$ runner. Denote the number of beads on the $i^{th}$ runner as $b_{i}$. Then, we may write $\lambda$ as $$\langle 0_{\lambda(0)},\dots,(e-1)_{\lambda(e-1)} \mid b_{0},\dots,b_{e-1}\rangle;$$
we omit $i_{\lambda(i)}$ if $\lambda(i)=\varnothing$ and write $i_{\la(i)}$ simply as $i$ if $\lambda(i)=(1)$. Additionally, we may omit $b_{0},\dots,b_{e-1}$ if there is no ambiguity. If $\lambda$ has $e$-weight $w$ and lies in a block $B$ of $\s_{n}$ or $\mathcal{H}_{n}$, we say that $B$ is the block of $e$-weight $w$ with the $\langle b_{0},\dots,b_{e-1}\rangle$ notation.

\subsection{Jantzen-Schaper formula and the product order}\label{section jantzen}
Let $\trianglerighteq$ denote the usual dominance order for partitions. Due to the fact that $\s_{n}$ is a \textit{cellular algebra} ~\cite{graham}, we have the following result.
\begin{thrm}\label{theorem cellular}\cite[Corollary 4.13]{dipper 2}
Suppose that $\la$ and $\mu$ are partitions of $n$. We have
\begin{itemize}
\item $[W^{\mu}:L^{\mu}]=1$,
\item $[W^{\la}:L^{\mu}]=0$ unless $\mu\trianglerighteq\lambda$.
\end{itemize}
\end{thrm}
Let $\lambda$ be a partition and consider its abacus display, say with $r$ beads. Suppose that after moving a bead at position $a$ up its runner to a vacant position $a-ie$, we obtain the partition $\mu$. Denote $l_{\lambda\mu}$ for the number of occupied positions between $a$ and $a-ie$, and let $h_{\lambda\mu}=i$. \\
Further, write $\lambda \xrightarrow{\text{$\mu$}} \tau$ if the abacus display of $\tau$ with $r$ beads is obtained from that of $\mu$ by moving a bead at position $b-ie$ to a vacant position $b$, and $a<b$.
\begin{definition}\label{defn js} \textit{Jantzen-Schaper bound}\\
Let $p=\textnormal{char}(\mathbb{F})$. For any ordered pair $(\lambda,\tau)$, we define the \textit{Jantzen-Schaper coefficient} to be the integer $$J_{\mathbb{F}}(\lambda,\tau):=\sum\limits_{\lambda \xrightarrow{\text{$\sigma$}} \tau}(-1)^{l_{\lambda \sigma}+l_{\tau \sigma} +1}(1+v_{p}(h_{\lambda \sigma})),$$
where $v_{p}$ denotes the standard $p$-valuation if $p>0$ and $v_{0}(x)=0$  $\forall x$.
The Jantzen-Schaper bound is defined as the integer
$$B_{\mathbb{F}}(\lambda,\mu)=\sum\limits_{\tau}J_{\mathbb{F}}(\lambda,\tau)[W^{\tau}_{\mathbb{F}}:L^{\mu}_{\mathbb{F}}].$$
\end{definition}

\begin{remark}
If $\la$ has $e$-weight $w$ and $p>w$, then $v_{p}(h_{\lambda \sigma})$ is always zero. In this case, $J_{\mathbb{F}}(\lambda,\tau)$ is independent of $\textnormal{char}(\mathbb{F})$ and we may just refer to it as $J(\lambda,\tau)$. Similarly, if $B_{\mathbb{F}}(\lambda,\tau)$ turns out to be independent of $\textnormal{char}(\mathbb{F})$, we just refer to it as $B(\lambda,\tau)$.
\end{remark}

\begin{thrm} Jantzen-Schaper formula(~\cite[Theorem 4.7]{Jantzen Schaper})\label{jantzen schaper}
$$[W^{\lambda}_{\mathbb{F}}:L^{\mu}_{\mathbb{F}}] \le B_{\mathbb{F}}(\lambda,\mu).$$
Moreover, the left-hand side is zero if and only if the right-hand side is zero.
\end{thrm}

\begin{cor}\label{cor jantzen}
If $B_{\mathbb{F}}(\lambda,\mu) \le 1$, then 
$$[W^{\lambda}_{\mathbb{F}}:L^{\mu}_{\mathbb{F}}] = B_{\mathbb{F}}(\lambda,\mu).$$
\end{cor}

We write $\lambda \rightarrow \tau$ if there exists some $\mu$ such that $\lambda \xrightarrow{\text{$\mu$}} \tau$. Further, write $\lambda <_{J}\sigma$ if there exist partitions $\tau_{0},\tau_{1},\dots,\tau_{r}$ such that $\tau_{0}=\lambda$, $\tau_{r}=\sigma$ and $\tau_{i-1}\rightarrow \tau_{i}$ $\forall i \in\{1,2,\dots,r\}$. We call $\le_{J}$ the \textit{Jantzen order} and it is clear that this defines a partial order on the set of all partitions. Only partitions in the same block are comparable under this partial order. Moreover, the dominance order extends the Jantzen order in the following sense: $\mu >_{J} \lambda$ implies $\mu \domr \lambda$. Theorem \ref{jantzen schaper} can be used to refine Theorem \ref{theorem cellular} the following way:
\begin{thrm}\label{jantzen order d} Suppose that $\lambda$ and $\mu$ are partitions of n. Then,
\begin{itemize}
\item $[W^{\mu}:L^{\mu}]=1;$
\item $[W^{\lambda}:L^{\mu}]>0 \Rightarrow \mu \ge_{J} \lambda .$
\end{itemize}
\end{thrm}

It is difficult to check that $\mu\ngtr_{J}\lambda$ by inspection. To this end, we introduce the \textit{product order} on partitions which was first defined by Tan in~\cite{beyond}. Let $\lambda$ be a partition, displayed on an abacus with $e$ runners and $r$ beads. Suppose that the beads having positive \textit{e-weights} are at positions $a_{1},a_{2},\dots,a_{s}$ with weights $w_{1},w_{2},\dots,w_{s}$ respectively. The \textit{induced e-sequence} of $\lambda$, denoted $s(\lambda)_{r}$, is defined as 
$$\bigsqcup\limits_{i=1}^{s}(a_{i},a_{i}-e,\dots,a_{i}-(w_{i}-1)e),$$ where $(b_{1},b_{2},\dots,b_{t})\sqcup(c_{1},c_{2},\dots,c_{u})$ denotes the weakly decreasing sequence obtained by rearranging terms in the sequence $(b_{1},\dots,b_{t},c_{1},\dots,c_{u})$. Note that $s(\lambda)_{r}\in\mathbb{N}_{>0}^{w}$, where $w$ is the $e$-weight of $\lambda$. 

We define a partial order $\ge_{P}$ on the set of partitions by: $\mu\ge_{P}\lambda$ if and only if $\mu$ and $\lambda$ have the same $e$-core and $e$-weight, and $s(\mu)_{r}\ge s(\lambda)_{r}$ (for sufficiently large $r$) in the standard product order of $\mathbb{N}_{>0}^{w}$.

\begin{lem}(~\cite[Lemma 2.9]{beyond})\label{product order}
$$\lambda\le_{J}\mu\Rightarrow\lambda\le_{P}\mu.$$
\end{lem}
Therefore, $\mu\ngeq_{P}\lambda \Rightarrow [W^{\lambda}:L^{\mu}]=0$.

%%%%%%%%%%%    Should say $\mu >_{P} \la$ implies $\mu \domr \la$. However, think of a proof first.        %%%%%%%%%%%%%

\begin{example}
Suppose that $e=5$, $r=10$, $\la=\langle 2,4_{2^{2}} \mid 2^{5} \rangle=(10,6,5,2,1^{2})$ and $\mu= \langle 0,1,4_{3} \mid 2^{5} \rangle=(15,3^{2},2^{2})$.
\[
\begin{array}{c@{\qquad}c}
\la&\mu\\
\abacus(bbbbn,bbnbn,nnbnb,nnnnb,nnnnn)&
\abacus(bbbbb,nnbbn,bbnnn,nnnnn,nnnnb)
\end{array}
\]
Then, $s(\la)_{10}=(19,14,14,12,9)$, $s(\mu)_{10}=(24,19,14,11,10)$. Hence, $\mu \domr \lambda$ but $\mu\ngtr_{P} \la$.

\end{example}

\subsection{$v$-decomposition numbers}\label{section v decomp}
For a brief introduction to \textit{v-decomposition numbers}, the reader may refer to~\cite[\S2.7]{Ac duality}. For our purposes, all we need to know is that given 2 partitions $\lambda$ and $\mu$ of $n$, we may define a polynomial $d^{e}_{\lambda\mu}(v)\in\mathbb{N}[v]$ with the following crucial property (which explains its name):

\begin{thrm}\label{v-decomposition numbers}
Let $\la,\mu\in\p(n)$. Then,
\begin{enumerate}
\item$d^{e}_{\lambda\mu}(1)=[W^{\lambda}_{\mathbb{C}}:L^{\mu}_{\mathbb{C}}]$,
\item$\frac{d}{dv}(d^{e}_{\lambda\mu}(v))|_{v=1}=B_{\mathbb{C}}(\lambda,\mu)$.
\end{enumerate}
\end{thrm}
\begin{proof}
(1) has been proven by Varagnolo and Vasserot in~\cite{varagnolo}. (2) was proved by Schroll and Tan in~\cite[Theorem 2.13]{Ac duality}.
\end{proof}
The following result tells us that $d^{e}_{\lambda\mu}(v)$ is either an even or an odd polynomial, depending on the relative $e$-signs of $\lambda$ and $\mu$.
\begin{thrm}~\cite[Theorem 2.4]{parities}\label{v-decomposition numbers are odd or even}
If $d^{e}_{\lambda\mu}(v)\neq0$, then
\begin{equation*}
d^{e}_{\lambda\mu}(v)\in
\begin{cases}
\mathbb{N}[v^{2}] & \text{if          } \sigma	_{e}(\lambda)=\sigma_{e}(\mu), \\
v\mathbb{N}[v^{2}] & \text{otherwise}.
\end{cases}
\end{equation*}
\end{thrm}

\begin{remark}\label{remark after}In section \ref{section wt 4}, we will sometimes calculate that $B_{\mathbb{C}}(\lambda,\mu)=2$ for some pair of partitions $(\lambda,\mu)$. In order to verify whether $[W^{\lambda}_{\mathbb{C}}:L^{\mu}_{\mathbb{C}}]$ is 1 or 2, we would calculate $\sigma_{e}(\lambda)$ and $\sigma_{e}(\mu)$. If they turn out to be the same, then $d^{e}_{\lambda\mu}(v)=v^{2}$ and $[W^{\lambda}_{\mathbb{C}}:L^{\mu}_{\mathbb{C}}]=1$ by Theorem \ref{v-decomposition numbers} and Theorem \ref{v-decomposition numbers are odd or even}. Otherwise, $d^{e}_{\lambda\mu}(v)=2v$ and $[W^{\lambda}_{\mathbb{C}}:L^{\mu}_{\mathbb{C}}]=2$.
\end{remark}

The following two theorems commonly known as the Runner Removal Theorems allow us to relate $v$-decomposition numbers for different values of $e$.

\begin{thrm}\label{Runner Removal Theorem Empty}~\cite[Theorem 4.5]{runner removal}
Suppose that $e\ge3$. Let $\lambda$ and $\mu$ be partitions lying in the same block, and display them on an abacus with $e$ runners and $r$ beads, for some large enough $r$. Suppose that there exists some $i$ such that in both abacus displays, the last bead on runner $i$ occurs before every unoccupied space on the abacus. Define two abacus displays with $e-1$ runners by deleting runner $i$ from each display, and let $\lambda^{-}$ and $\mu^{-}$ be the partitions defined by these displays. Then,
$$d^{e}_{\lambda\mu}(v)=d^{e-1}_{\lambda^{-}\mu^{-}}(v).$$
\end{thrm}

\begin{thrm}\label{Runner Removal Theorem Full}~\cite{another runner removal theorem}
Suppose that $e\ge3$. Let $\lambda$ and $\mu$ be partitions lying in the same block, and display them on an abacus with $e$ runners and $r$ beads, for some large enough $r$. Suppose that there exists some $i$ such that in both abacus displays, the first unoccupied space on runner $i$ occurs after every bead on the abacus. Define two abacus displays with $e-1$ runners by deleting runner $i$ from each display, and let $\hat{\lambda}$ and $\hat{\mu}$ be the partitions defined by these displays. Then,
$$d^{e}_{\lambda\mu}(v)=d^{e-1}_{\hat{\lambda}\hat{\mu}}(v).$$
\end{thrm}

%%%%%%%%%%%%%                REMARK AND EXAMPLE IN THESIS             %%%%%%%%%%%%%%%%%%%%%%%%%%%%

\begin{remark}\label{remark empty runner} \textbf{Calculating $v$-decomposition numbers in practice.}
When $\mu\in \pr(n)$, we have a relatively fast recursive algorithm for calculating $d^{e}_{\lambda\mu}(v)$ called the LLT algorithm~\cite{llt} (developed by Lascoux, Leclerc and Thibon). The author uses the GAP package hecke (https://www.gap-system.org/Packages/hecke.html) which was first written by Mathas for running the LLT algorithm. 

If $\mu\in \ps(n)$, we may calculate $d^{e}_{\lambda\mu}(v)$ by adding an empty runner to $\la$ and $\mu$ and using Theorem \ref{Runner Removal Theorem Empty}. By Theorem \ref{v-decomposition numbers}, the decomposition matrix for $\s_{n}$ can be calculated in principle when the underlying field is $\mathbb{C}$.
\end{remark}

\begin{example}\label{example empty runner}
Suppose that $e\ge4$, $\mu=\langle 0,1,2,3 \mid 2^{e} \rangle$, $\la=\langle 0_{1^{2}},2_{1^{2}} \mid 2^{e} \rangle$, $\hat{\mu}=\langle 0,1,2,3 \mid 2^{4} \rangle$, $\hat{\la}=\langle 0_{1^{2}},2_{1^{2}} \mid 2^{4} \rangle$, $\mu^{+}=\langle 0,1,2,3 \mid 2^{4},0 \rangle$ and $\la^{+}=\langle 0_{1^{2}},2_{1^{2}} \mid 2^{4}, 0 \rangle$.
\[
\begin{array}{c@{\qquad}c@{\qquad}c}
\mu&\hat{\mu}&\mu^{+}\\
\abacus(bbbbbhb,nnnnbhb,bbbbnhn)&
\abacus(bbbb,nnnn,bbbb)&
\abacus(bbbbn,nnnnn,bbbbn)

\end{array}
\]
\[
\begin{array}{c@{\qquad}c@{\qquad}c}
\la&\hat{\la}&\la^{+}\\
\abacus(nbnbbhb,bbbbbhb,bnbnnhn)&
\abacus(nbnb,bbbb,bnbn)&
\abacus(nbnbn,bbbbn,bnbnn)
\end{array}
\]
By Theorem \ref{Runner Removal Theorem Full}, $d^{e}_{\la\mu}(v)=d^{4}_{\hat{\la}\hat{\mu}}(v)$. By Theorem \ref{Runner Removal Theorem Empty} and the LLT algorithm, $d^{4}_{\hat{\la}\hat{\mu}}(v)=d^{5}_{\la^{+}\mu^{+}}(v)=v$.
\end{example}

\subsection{The modular branching rules}\label{section branching}
We use some notational conventions for modules. We write
$$M\sim M_{1}^{a_{1}}+M_{2}^{a_{2}}+\dots+M_{t}^{a_{t}}$$
to indicate that $M$ has a filtration in which the factors are $M_{1},\dots,M_{t}$ appearing $a_{1},\dots,a_{t}$ times respectively. Additionally, we write $M^{\oplus a}$ to indicate the direct sum of $a$ isomorphic copies of $M$.

There are restriction and induction functors which are exact functors between $\s_{n-1}$ and $\s_{n}$. If $M$ is a module for $\s_{n}$, the restriction of $M$ to $\s_{n-t}$ is denoted by $M{\downarrow}_{\s_{n-t}}$. Similarly, the induction of $M$ to $\s_{n+t}$ is denoted by $M{\uparrow}^{\s_{n+t}}$. If $B$ is a block of $\s_{n-t}$, we write $M{\downarrow}_{B}$ to indicate the projection of $M{\downarrow}_{\s_{n-t}}$ onto $B$. Similarly, if $C$ is a block of $\s_{n+t}$, we write $M{\uparrow}^{B}$ to indicate the projection of $M{\uparrow}^{\s_{n+t}}$ onto $C$. In this section, we describe the restriction and induction of Weyl modules and simple modules.

%Let $A$, $B$ and $C$ be blocks of $\s_{n-k}$, $\s_{n}$ and $\s_{n+k}$ with cores $\kappa_{A}$, $\kappa_{B}$ and $\kappa_{C}$ respectively. Suppose that there is an integer $j$ such that $\kappa_{A}$ is obtained from $\kappa_{B}$ by removing $k$ $j$-removable nodes, whereas $\kappa_{C}$ is obtained from $\kappa_{B}$ by adding $k$ $j$-addable nodes.

Suppose that $A$, $B$ and $C$ are blocks of $\s_{n-k}$, $\s_{n}$ and $\s_{n+k}$ respectively, and that there is an $e$-residue $j$ such that a partition lying in $A$ may be obtained from a partition lying in $B$ by removing exactly $k$ $j$-removable nodes, while a partition lying in $C$ may be obtained from a partition lying in $B$ by adding exactly $k$ $j$-addable nodes.

Suppose that $\lambda$ is a partition in $B$, and that $\lambda^{-1},\lambda^{-2},\dots,\lambda^{-t}$ are the partitions in $A$ that may be obtained from $\lambda$ by removing $k$ $j$-removable nodes. Similarly, let $\lambda^{+1},\lambda^{+2},\dots,\lambda^{+s}$ be the partitions in $C$ that may be obtained from $\lambda$ by adding $k$ $j$-addable nodes. We have the following result.

\begin{thrm}(The Branching Rule~\cite{brundan})\label{branching rule}
Suppose that $A$, $B$, $C$ and $\lambda$ are as above. Then,
$$W^{\lambda}{\downarrow}_{A} \sim (W^{\lambda^{-1}})^{k!}+ (W^{\lambda^{-2}})^{k!}+\dots+ (W^{\lambda^{-t}})^{k!}$$ and
$$W^{\lambda}{\uparrow}^{C} \sim (W^{\lambda^{+1}})^{k!}+ (W^{\lambda^{+2}})^{k!}+\dots+ (W^{\lambda^{+s}})^{k!}.$$
\end{thrm}

We now discuss the restriction and induction of simple modules. Suppose that the nodes with residue $j$ are on runner $i$. The \textit{j-signature} of $\lambda$ is the sequence of signs defined as follows. Starting from the top row of the abacus display for $\lambda$ and working downwards, write a $-$ if there is a bead on runner $i$ but no bead on runner $i-1$; write a $+$ if there is a bead on runner $i-1$ but no bead on runner $i$; write nothing for that row otherwise. Given the $j$-signature of $\lambda$, successively delete all neighbouring pairs of the form $-+$ to obtain the \textit{reduced j-signature} of $\lambda$. If there are any $-$ (resp. $+$) signs in the reduced $j$-signature of $\lambda$, we call the corresponding nodes on runner $i$ \textit{normal} (resp. \textit{conormal}). Normal (resp. conormal) nodes with residue $j$ are also called \textit{j-normal} (resp. \textit{j-conormal}).   

\begin{definition}
Let $\la$ be a partition. We denote the number of $j$-normal nodes of $\la$ by $\epsilon_{j}(\lambda)$ and the number of $j$-conormal nodes of $\la$ by $\varphi_{j}(\lambda)$. For $t\le \epsilon_{j}(\lambda)$, we define $\tilde{E}_{j}^{t}\lambda$ to be the partition obtained from $\lambda$ by removing the $t$ highest (in an abacus display for $\la$) $j$-normal nodes. For $t \le \varphi_{j}(\lambda)$, we define $\tilde{F}_{j}^{t}\lambda$ to be the partition obtained from $\lambda$ by adding the $t$ lowest (in an abacus display for $\la$) $j$-conormal nodes.
\end{definition}

\begin{thrm}~\cite{brundan}\label{modular branching rule}
Suppose that $A$, $B$, $C$ and $\lambda$ are as above.
\begin{itemize}
\item If $\epsilon_{j}(\lambda)<k$, then $L^{\lambda}{\downarrow}_{A}=0$.
\item If $\epsilon_{j}(\lambda)>k$, then $soc(L^{\lambda}{\downarrow}_{A})\cong(L^{\tilde{E}_{j}^{k}\lambda})^{\oplus k!}$.
\item If $\epsilon_{j}(\lambda)=k$, then $L^{\lambda}{\downarrow}_{A}\cong(L^{\tilde{E}_{j}^{k}\lambda})^{\oplus k!}$.
\item If $\varphi_{j}(\lambda)<k$, then $L^{\lambda}{\uparrow}^{C}=0$.
\item If $\varphi_{j}(\lambda)>k$, then $soc(L^{\lambda}{\uparrow}^{C})\cong(L^{\tilde{F}_{j}^{k}\lambda})^{\oplus k!}$.
\item If $\varphi_{j}(\lambda)=k$, then $L^{\lambda}{\uparrow}^{C}\cong(L^{\tilde{F}_{j}^{k}\lambda})^{\oplus k!}$.
\end{itemize}
\end{thrm}

The following lemma guarantees that the weight of a partition will not increase if we remove all of its $j$-normal nodes (or add all of its $j$-conormal nodes) for some $e$-residue $j$.
\begin{lem}\label{lemma w-kl}
Suppose that $\mu$ lies in a block of $\s_{n}$ of weight $w$. Let $k=\epsilon_{j}(\mu)$ and $l=\varphi_{j}(\mu)$. Then, $\tilde{E}_{j}^{k}\mu$ and $\tilde{F}_{j}^{l}\mu$ have weight $w-kl$.
\end{lem}
Before we proceed with the proof of this lemma, it may be helpful to first look at an example.
\begin{example}
In the diagrams below, we only display two runners of the abacus displays for the partitions; the runner on the right corresponds to the nodes with $e$-residue equal to $j$. We highlight the reduced $j$-signature in red. In this example, $k=3$ and $l=2$.
\[
\begin{array}{c@{\qquad}c@{\qquad}c}
\mu&\tilde{E}_{j}^{k}\mu&\tilde{F}_{j}^{l}\mu \\
\abacus(bnC,bnC,nbN,nb-,bb,nb-,bn+,bn+,nbN,nb-,bb,bn+,nbN)&
\abacus(bnC,bnC,bnC,nb-,bb,nb-,bn+,bn+,bnC,nb-,bb,bn+,bnC)&
\abacus(nbN,nbN,nbN,nb-,bb,nb-,bn+,bn+,nbN,nb-,bb,bn+,nbN)
\end{array}
\]
Assuming for simplicity that the other runners which are not displayed in the diagram have no weight, we may count that $\mu$ has weight $43$ while $\tilde{E}_{j}^{k}\mu$ and $\tilde{F}_{j}^{l}\mu$ have weight $37$.
\end{example}
\begin{proof}
We only show the proof for $\tilde{E}_{j}^{k}\mu$ here as the other case is similar.
Let $\mu$ be a partition, $k=\epsilon_{j}(\mu)$ and $l=\varphi_{j}(\mu)$. We focus our attention on the two adjacent runners, with the runner on the right corresponding to the nodes with $e$-residue equal to $j$. Our task is to keep track of the change of weight of each bead in these two runners when $\tilde{E}_{j}^{k}$ is applied to $\mu$. We may categorize the beads in these two runners into four categories:
\begin{enumerate}
\item Normal; in which case there are no conormal beads in the rows below it.
\item Conormal; in which case there are no normal beads in the rows above it.
\item Two beads in the same row.
\item Two beads in two distinct rows forming a $(-+)$ pair that was deleted from the $j$-signature to form the reduced $j$-signature; by definition, there are no normal or conormal beads in between these two rows.
\end{enumerate}
The change of weight of each of these types of beads when $\tilde{E}_{j}^{k}$ is applied to $\mu$ is summarised below:
\begin{enumerate}
\item Each normal bead loses weight $l$ when moved to the left.
\item The conormal beads do not experience any change in weight since there are no normal beads above them.
\item If two beads are in the same row and that there are $\alpha$ normal beads above this row, then the bead on the right gains weight $\alpha$ while the bead on the left loses weight $\alpha$. Therefore, there is no net change in weight contributed by beads of this type.
\item Suppose that we have two beads in two distinct rows forming a $(-+)$ pair that was deleted from the $j$-signature to form the reduced $j$-signature. If there are $\beta$ normal beads in the rows above this $(-+)$ pair (this is well defined), then the bead on the right corresponding to the $-$ gains weight $\beta$, while the bead on the left corresponding to the $+$ gains weight $\beta$. Therefore, there is no net change in weight contributed by beads of this type.
\end{enumerate}
Since there are $k$ normal beads, $\tilde{E}_{j}^{k}\mu$ must have weight $kl$ less than $\mu$.

\end{proof}
\section{Adjustment Matrices}\label{section adj}
Let the $q$-Schur algebra over an arbitrary field $\mathbb{F}$ be denoted by $\s_{n}$, and denote the $q$-Schur algebra over $\mathbb{C}$ by $\s^{0}_{n}$. Let $\zeta$ be a primitive $e^{th}$-root of unity in $\mathbb{C}$. We write $\mathcal{H}^{0}_{n}$ for $\mathcal{H}_{\mathbb{C},\zeta}(\mathfrak{S}_{n})$ and $\mathcal{H}_{n}$ for $\mathcal{H}_{\mathbb{F},q}(\mathfrak{S}_{n})$. By Theorem \ref{nakayama}, the Weyl modules corresponding to two partitions lie in the same block of $\s_{n}$ if and only if they lie in the same block of $\s^{0}_{n}$. Similarly, the Specht modules corresponding to two partitions lie in the same block of $\mathcal{H}_{n}$ if and only if they lie in the same block of $\mathcal{H}^{0}_{n}$ by Corollary \ref{nakayama hecke}. Therefore, given a block $B$ of $\s_{n}$ or $\mathcal{H}_{n}$, we may denote $B^{0}$ to be its corresponding block in $\s_{n}^{0}$ or $\mathcal{H}_{n}^{0}$ respectively.

The \textit{Grothendieck group} $\g(\s_{n})$ of $\s_{n}$ is the additive abelian group (with complex coefficients) generated by the symbols $[E]$, where $E$ runs over the isomorphism classes of finite dimensional $\s_{n}$-modules, together with the relations $[F]=[E]+[G]$ whenever there exists a short exact sequence $0\rightarrow E \rightarrow F \rightarrow G \rightarrow 0$. Thus, as a complex vector space, $\g(\s_{n})$ has a basis given by $\{[L^{\la}_{\F}]\mid \la \in \p(n)\}$. We denote the decomposition matrix for $\s_{n}$ by $\Ds$. Since $\Ds$ is unitriangular, $\{[W^{\la}_{\F} \mid \la \in \p(n) \}$ must be another basis for $\gs$ with $\Ds$ being the transition matrix between these two bases. In other words, given any $\la\in\p(n)$,
$$[W^{\la}_{\F}]=\sum\limits_{\mu\in\p(n)}[W^{\la}_{\F}:L^{\mu}_{\F}][L^{\mu}_{\F}].$$
There is a well-defined homomorphism $d_{\s}:\gsz\rightarrow \gs$ which fixes the Weyl modules, $d_{\s}([W^{\la}_{\C}])=[W^{\la}_{\F}]$ $\forall \la \in \p(n)$. In the literature~\cite{mathas book}, this homomorphism is known as the \textit{decomposition map}.

\begin{thrm}~\cite[Theorem 6.35]{mathas book}\label{theorem adjustment}
Let $\Ds$ and $\Dsz$ be the decomposition matrices for $\s_{n}$ and $\s_{n}^{0}$ respectively. Let $\mathcal{A}_{\mathcal{S}}$ be the matrix $(a_{\mu\nu}^{\s})_{\mu,\nu \in \p(n)}$, where $a^{\s}_{\mu\nu}$'s satisfy $d_{\mathcal{S}}([L^{\mu}_{\mathbb{C}}]) = \sum\limits_{\nu \in \p(n)} a_{\mu\nu}^{\s} [L^{\nu}_{\mathbb{F}}]$. Then, $a_{\mu\nu}^\mathcal{S} \in \mathbb{N}$ for all $\mu,\nu \in \mathcal{P}(n)$ and $$\Ds=\Dsz\As.$$
\end{thrm}
We call the matrix $\As$ in Theorem \ref{theorem adjustment} the \textit{adjustment matrix} for $\s_{n}$. Replacing $\s_{n}$ by $\h_{n}$ and $\s^{0}_{n}$ by $\h^{0}_{n}$ in Theorem \ref{theorem adjustment} yields the following theorem.

\begin{thrm}\label{theorem adjustment h}
Let $\Dh$ and $\Dhz$ be the decomposition matrices for $\h_{n}$ and $\h_{n}^{0}$ respectively. Let $\mathcal{A}_{\mathcal{H}}$ be the matrix $(a_{\mu\nu}^{\h})_{\mu,\nu \in \pr(n)}$, where $a^{\h}_{\mu\nu}$'s satisfy $d_{\mathcal{H}}([D^{\mu}_{\mathbb{C}}]) = \sum\limits_{\nu \in \pr(n)} a_{\mu\nu}^{\h} [D^{\nu}_{\mathbb{F}}]$. Then, $a_{\mu\nu}^\mathcal{H} \in \mathbb{N}$ for all $\mu,\nu \in \pr(n)$ and $$\Dh=\Dhz\Ah.$$
\end{thrm} 
We call the matrix $\Ah$ in Theorem \ref{theorem adjustment h} the \textit{adjustment matrix} for $\h_{n}$. The matrix $\As$ has rows and columns indexed by $\p(n)$, whereas $\Ah$ has rows and columns indexed by $\pr(n)$. One may argue using the Schur functor that $\Ah$ is the submatrix of $\As$ obtained by removing the rows and columns of $\As$ indexed by $\ps(n)$. Therefore, we refer to adjustment matrices as $\mathcal{A}$ when it is clear whether we are dealing with $\s_{n}$ or $\h_{n}$. Moreover, given any two partitions $\la,\mu\in\p(n)$, we may refer to the $(\la,\mu)$-entry of $\mathcal{A}$ as $\adj_{\la\mu}$ without any ambiguity. 

We highlight the unitriangular property of adjustment matrices inherited from the unitriangularity of the decomposition matrices in the following corollary.
\begin{cor}\label{jantzen order adj}
Suppose that $\lambda$ and $\mu$ are partitions lying in a block B of $\s_{n}$. Then,
\begin{itemize}
\item $\textnormal{adj}_{\mu\mu}=1$,
\item $\textnormal{adj}_{\lambda\mu}=0$ unless $\mu\ge_{J}\lambda$.
\end{itemize}
\end{cor}
It follows from Lemma \ref{product order} and Corollary \ref{jantzen order adj} that $\mu\ngeq_{P}\lambda\Rightarrow\textnormal{adj}_{\lambda\mu}=0$. As mentioned before, it is difficult to check that $\mu\ngeq_{J}\lambda$, whereas $\mu\ngeq_{P}\lambda$ can be verified by inspection. 
In terms of adjustment matrices, we have the following corollary of Theorem \ref{v-decomposition numbers}.
\begin{cor}\label{decomposition number 0 or v}
Suppose that $\lambda$ and $\mu$ are partitions lying in a block $B$ of $\s_{n}$ of weight $w<\textnormal{char}(\mathbb{F})$. Additionally, suppose that $\textnormal{adj}_{\nu\mu}=0$ for all partitions $\nu$ such that $\lambda <_{J} \nu <_{J} \mu$, and that $d_{\lambda \mu}^{e}(v) \in \{0,v\}$. Then, $\textnormal{adj}_{\lambda\mu}=0.$
\end{cor}
\begin{proof}
The proof of this corollary is essentially the same as the proof of~\cite[Corollary 2.12]{wt 4} by replacing $S^{\nu}$ with $W^{\nu}$ and $D^{\mu}$ with $L^{\mu}$.
\end{proof}

In section \ref{section wt 4}, we will use the following easy fact several times.
\begin{lem}\label{lemma easy fact}
Let $\la$ and $\mu$ be two distinct partitions lying in some block $B$ of $\s_{n}$. If $[W^{\la}_{\F}:L^{\mu}_{\F}]=[W^{\la}_{\C}:L^{\mu}_{\C}]$, then $\adj_{\la\mu}=0$.
\end{lem}
\begin{proof}
By Theorem \ref{theorem adjustment}, Theorem \ref{jantzen order d} and Corollary \ref{jantzen order adj},
$$[W^{\la}_{\F}:L^{\mu}_{\F}]=[W^{\la}_{\C}:L^{\mu}_{\C}]+\adj_{\la\mu}+\sum_{\la<_{J}\nu<_{J}\mu}[W^{\la}_{\C}:L^{\nu}_{\C}]\adj_{\nu\mu}.$$
The terms in the sum are all non-negative, so $\adj_{\la\mu}=0$.
\end{proof}

\subsection{James's Conjecture}
Throughout the rest of this paper, we shall adopt the Kronecker delta. In view of the LLT algorithm (see Remark \ref{remark empty runner}), the decomposition matrix $\mathcal{D}^{0}$ of $\s^{0}_{n}$ or $\h^{0}_{n}$ can be calculated in principle. Therefore, we focus our attention on studying the adjustment matrices $\mathcal{A}$. The following is the famous James's Conjecture for adjustment matrices.

\begin{conj}(James's Conjecture~\cite[\S4]{glnq})
Suppose that $\lambda$ and $\mu$ are partitions lying in a block B of $\s_{n}$ or $\h_{n}$ with $e$-weight $w$. If $w<\textnormal{char}(\mathbb{F})$, then $\textnormal{adj}_{\lambda\mu}=\delta_{\lambda\mu}$.
\end{conj}

James's Conjecture is easy to verify for blocks of weight $0$ or $1$. A block of weight $0$ contains its core $\kappa$ as the only partition, so $W^{\kappa}=L^{\kappa}$ regardless of the underlying field. Blocks of weight 1 each contain $e$ partitions which can be totally ordered by the dominance order, $\lambda^{1}\domr\cdots\domr\lambda^{e}$. The decomposition numbers are independent of the underlying field; $[W^{\lambda^{i}}:L^{\lambda^{j}}]$ is equal to $1$ if $i \in \{j,j+1\}$, and is equal to $0$ otherwise. We summarize the progress on James's Conjecture made so far by the works of Fayers, Richards, Schroll and Tan in the following two theorems.

\begin{thrm} (James's Conjecture for blocks of Iwahori-Hecke algebras of weight at most 4)~\cite[Theorem 4.1]{wt 3}~\cite[Theorem 2.6]{wt 4}\label{James Conjecture hecke 4}~\cite{Richards}

Suppose that $\textnormal{char}(\mathbb{F})\ge5$. Let $\lambda$ and $\mu$ be $e$-regular partitions lying in $B$, a block of $\mathcal{H}_{n}$ of weight at most 4. Then, $\textnormal{adj}_{\lambda\mu}=\delta_{\lambda\mu}$.
\end{thrm}

\begin{thrm} (James's Conjecture for blocks of q-Schur algebras of weight 2)~\cite[Corollary 3.6]{Ac duality}\label{James Conjecture schur 2}

Suppose that $\textnormal{char}(\mathbb{F})\ge3$. Let $\lambda$ and $\mu$ be partitions lying in $B$, a block of $\s_{n}$ of weight 2. Then, $\textnormal{adj}_{\lambda\mu}=\delta_{\lambda\mu}$.
\end{thrm}

\begin{comment}
The author has also proved James's Conjecture when $e\neq4$ for the principal block of $\mathcal{H}_{5e}$ which has \textit{e-weight} equal to 5. 
\begin{thrm}~\cite[Theorem 2.9]{wt 5}\label{James Conjecture weight 5}
Suppose that $e\neq4$ and $\textnormal{char}(\mathbb{F})$$\ge5$. Let $\lambda$ and $\mu$ be partitions lying in the principal block of $\h_{5e}$. Then, $\adj_{\la\mu}=\delta_{\la\mu}$.
\end{thrm}
\end{comment}

In section \ref{section wt 3}, we prove James's Conjecture for blocks of $\s_{n}$ of weight 3.
\begin{thrm} (James's Conjecture for blocks of q-Schur algebras of weight 3)\label{James Conjecture schur 3}
Suppose that $\textnormal{char}(\mathbb{F})\ge5$. Let $\lambda$ and $\mu$ be partitions lying in $B$, a block of $\s_{n}$ of weight 3. Then, $\textnormal{adj}_{\lambda\mu}=\delta_{\lambda\mu}$.
\end{thrm}

In section \ref{section wt 4}, we prove James's Conjecture for blocks of $\s_{n}$ of weight 4.
\begin{thrm} (James's Conjecture for blocks of q-Schur algebras of weight 4)\label{James Conjecture schur 4}
Suppose that $\textnormal{char}(\mathbb{F})\ge5$. Let $\lambda$ and $\mu$ be partitions lying in $B$, a block of $\s_{n}$ of weight 4. Then, $\textnormal{adj}_{\lambda\mu}=\delta_{\lambda\mu}$.
\end{thrm}

In the context of Theorem \ref{James Conjecture schur 3} and Theorem \ref{James Conjecture schur 4}, we already know that $\textnormal{adj}_{\lambda\mu}=\delta_{\lambda\mu}$ when $\lambda$ and $\mu$ are both $e$-regular due to Theorem \ref{James Conjecture hecke 4}. In fact, this knowledge can be strengthened for $q$-Schur algebras:

\begin{prop} \label{e-regular mu}
Let $B$ be a block of $\s_{n}$ and let $\bar{B}$ be its corresponding block of $\mathcal{H}_{n}$ with the same $e$-core and $e$-weight. Let $\lambda$ and $\mu$ be partitions lying in $B$ with $\mu$ being $e$-regular. If the adjustment matrix for $\bar{B}$ is the identity matrix, then $\textnormal{adj}_{\lambda\mu}=\delta_{\lambda\mu}$.
\end{prop}
\begin{proof}
Suppose that $\la\neq\mu$. Then,
$$[W^{\la}_{\mathbb{F}}:L^{\mu}_{\mathbb{F}}]=[W^{\la}_{\mathbb{C}}:L^{\mu}_{\mathbb{C}}]+\textnormal{adj}_{\la\mu}+\sum\limits_{\la<_{J}\sigma<_{J}\mu}[W^{\la}_{\mathbb{C}}:L^{\sigma}_{\mathbb{C}}]\textnormal{adj}_{\sigma\mu}.$$
Since $\mu$ is $e$-regular, we can apply Theorem \ref{schur functor d} to get $[W^{\la}_{\mathbb{F}}:L^{\mu}_{\mathbb{F}}]=[S^{\la}_{\mathbb{F}}:D^{\mu}_{\mathbb{F}}]$ and $[W^{\la}_{\mathbb{C}}:L^{\mu}_{\mathbb{C}}]=[S^{\la}_{\mathbb{C}}:D^{\mu}_{\mathbb{C}}]$. On the other hand, $[S^{\la}_{\mathbb{F}}:D^{\mu}_{\mathbb{F}}]=[S^{\la}_{\mathbb{C}}:D^{\mu}_{\mathbb{C}}]$ since the adjustment matrix for $\bar{B}$ is the identity matrix by assumption. Moreover, the terms in the sum are non-negative, so $\adj_{\la\mu}$ must be zero.
\end{proof}

\subsection{The row and column removal theorems}\label{section row rem}
\begin{definition}\label{defn row rem}
We define the \textit{row removal function} $\row$ as
\begin{align*}
  \row \colon \bigcup_{n>0}\p(n) &\to \bigcup_{n>0}\p(n)\\
  \nu=(\nu_{1},\nu_{2},\dots,\nu_{l(\nu)}) &\mapsto (\nu_{2},\nu_{3},\dots,\nu_{l(\nu)}).
\end{align*}
We define the \textit{column removal function} $\col$ as
\begin{align*}
  \col \colon \bigcup_{n>0}\p(n) &\to \bigcup_{n>0}\p(n)\\
  \nu=(\nu_{1},\nu_{2},\dots,\nu_{l(\nu)}) &\mapsto (\nu_{1}-1,\nu_{2}-1,\dots,\nu_{l(\nu)}-1).
\end{align*}
%Note that $\row(\nu)$ is the partition corresponding to the Young diagram obtained by removing the first row from the Young diagram of $\nu$; similarly, $\col(\nu)$ is the partition corresponding to the Young diagram obtained by removing the first column from the Young diagram of $\nu$.

\end{definition}

\begin{thrm}\label{row removal d}
(~\cite[\S4.2 (9)]{donkin}) Suppose that $\lambda$ and $\mu$ are partitions of $n$ with $\lambda_{1}=\mu_{1}$.
Then,
$$[W^{\lambda}:L^{\mu}]=[W^{\row(\la)}:L^{\row(\mu)}].$$
\end{thrm}

\begin{cor}\label{row removal adj}~\cite[Corollary 2.19]{wt 5}
Suppose that $\lambda$ and $\mu$ are partitions of $n$ with $\lambda_{1}=\mu_{1}$. Then, $$\textnormal{adj}_{\lambda \mu}=\textnormal{adj}_{\row(\lambda) \row(\mu)}.$$
\end{cor}

\begin{thrm}\label{column removal d}
(~\cite[\S4.2 (15)]{donkin}) Suppose that $\lambda$ and $\mu$ are partitions of $n$ with $l(\lambda)=l(\mu)$.
Then,
$$[W^{\lambda}:L^{\mu}]=[W^{\col(\lambda)}:L^{\col(\mu)}].$$
\end{thrm}
\begin{cor}\label{column removal adj}
Suppose that $\lambda$ and $\mu$ are partitions of $n$ with $l(\lambda)=l(\mu)$. Then, $$\textnormal{adj}_{\lambda \mu}=\textnormal{adj}_{\col(\lambda) \col(\mu)}.$$
\end{cor}
\begin{proof}
This is similar to the proof of Corollary \ref{row removal adj} in~\cite[Corollary 2.19]{wt 5}; we use Theorem \ref{column removal d} instead of Theorem \ref{row removal d}.
\end{proof}

%\begin{remark} \textbf{Comparing 2 partitions lying in the same block of $\s_{n}$ using abacus displays}
%\end{remark}

Suppose that $\nu$ is a partition of weight $w$. Let us examine the weights of $\row(\nu)$ and $\col(\nu)$. An abacus display for $\row(\nu)$ is obtained from that of $\nu$ by replacing the bead corresponding to $\nu_{1}$ (the maximal occupied position) with an empty space. Therefore, if this bead has weight $s\ge 0$, then $\row(\nu)$ would have weight $w-s\le w$. On the other hand, an abacus display for $\col(\nu)$ is obtained from that of $\nu$ by replacing the first unoccupied space with a bead. If there are $r\ge 0$ beads in the same runner under this position, then $\col(\nu)$ would have weight $w-r\le w$. For our purposes, the crucial point is that the weights of $\row(\nu)$ and $\col(\nu)$ are at most $w$.

%%%%%%%%%%%%%%%               EXAMPLE IN THESIS %%%%%%%%%%%%%%%%%%%%%%%%%
\begin{comment}
\begin{example}
Suppose that $e=5$. Let $\la=\langle 0_{1^{2}},2_{1^{2}} \mid 3^{2},4^{2},3 \rangle$, $\mu=\langle 0,1,2_{1^{2}} \mid 3^{2},4^{2},3 \rangle$ and $\nu=\langle 0_{1^{2}},2_{2} \mid 3^{2},4^{2},3 \rangle$ be the partitions lying in the weight 4 block of $\s_{n}$ with the $\langle 3^{2},4^{2},3 \rangle$ notation.
\[
\begin{array}{c@{\qquad}c@{\qquad}c}
\la&\mu&\nu \\
\abacus(bbbbb,nbbbb,bbnbb,bnbbn,nnbnn,nnnnn,nnnnn)&
\abacus(bbbbb,bbbbb,nnnbb,bbbbn,nnbnn,nnnnn,nnnnn)&
\abacus(bbbbb,nbbbb,bbbbb,bnnbn,nnnnn,nnbnn,nnnnn)
\end{array}
\]

\[
\begin{array}{c@{\qquad}c}
\row(\la)&\row(\mu) \\
\abacus(bbbbb,nbbbb,bbnbb,bnbbn,nnnnn,nnnnn,nnnnn)&
\abacus(bbbbb,bbbbb,nnnbb,bbbbn,nnnnn,nnnnn,nnnnn)
\end{array}
\]

\[
\begin{array}{c@{\qquad}c}
\col(\la)&\col(\nu) \\
\abacus(bbbbb,bbbbb,bbnbb,bnbbn,nnbnn,nnnnn,nnnnn)&
\abacus(bbbbb,bbbbb,bbbbb,bnnbn,nnnnn,nnbnn,nnnnn)
\end{array}
\]

Observe that $\la_{1}=\mu_{1}$, so $\adj_{\la\mu}=\adj_{\row(\la) \row(\mu)}$ by Corollary \ref{row removal adj}. 
On the other hand, $l(\la)=l(\nu)$, so $\adj_{\la\nu}=\adj_{\col(\la) \col(\nu)}$ by Corollary \ref{column removal adj}.
Also observe that $\row(\la)$ and $\row(\mu)$ have weight 3, while $\col(\la)$ and $\col(\nu)$ have weight 2.
\end{example}
\end{comment}

\begin{remark}\label{remark row removal}
In sections \ref{section wt 3} and \ref{section wt 4}, we will want to prove that $\textnormal{adj}_{\lambda\mu}=\delta_{\lambda\mu}$ for all pairs of partitions $(\lambda,\mu)$ lying in a block of $\s_{n}$ of weight $w$, assuming that James's conjecture holds for all blocks of $\s_{m}$ of weight at most $w$, where $m<n$. Note that $\mu>_{P}\lambda$ implies that $\mu_{1}\ge\lambda_{1}$ and $l(\mu)\le l(\lambda)$. When $\mu_{1}=\lambda_{1}$ or $l(\mu)=l(\lambda)$, we may apply Corollary \ref{row removal adj} or Corollary \ref{column removal adj} respectively to conclude that $\textnormal{adj}_{\lambda\mu}=\delta_{\lambda\mu}$. Thus, in this setting, we have $\adj_{\la\mu}=0$ unless $\mu>_{P}\lambda$, $\mu_{1}>\lambda_{1}$ and $l(\mu)<l(\lambda)$. We write $\mu\gg \lambda$ when $\mu>_{P}\lambda$, $\mu_{1}>\lambda_{1}$ and $l(\mu)<l(\lambda)$.
\end{remark}

%%%%%%%%%%%%%%%%%                 EXAMPLE IN THESIS         %%%%%%%%%%%%%%%%%%%%%%%%%%%%%%%%%%%%%

\begin{comment}
\begin{example}
Suppose that $e=5$. Let $\la=\langle 0_{1^{2}},2_{1^{2}} \mid 3^{2},4^{2},3 \rangle$ and $\mu=\langle 0,1,2_{2} \mid 3^{2},4^{2},3 \rangle$ be the partitions lying in the weight 4 block of $\s_{n}$ with the $\langle 3^{2},4^{2},3 \rangle$ notation.
\[
\begin{array}{c@{\qquad}c}
\la&\mu \\
\abacus(bbbbb,nbbbb,bbnbb,bnbbn,nnbnn,nnnnn,nnnnn)&
\abacus(bbbbb,bbbbb,nnbbb,bbnbn,nnnnn,nnbnn,nnnnn)
\end{array}
\]

Observe that $\mu \gg \la$.
\end{example}
\end{comment}

\subsection{Lowerable partitions}\label{lowerable notation}
Recall the decomposition map $d_{\s}$ between the Grothendieck groups $\gsz$ and $\gs$. Suppose that $A$, $B$ and $C$ are blocks of $\s_{n-1}$, $\s_{n}$ and $\s_{n+1}$ respectively, and that there is an $e$-residue $j$ such that a partition lying in $A$ may be obtained from a partition lying in $B$ by removing exactly one $j$-removable node, while a partition lying in $C$ may be obtained from a partition lying in $B$ by adding exactly one $j$-addable node. Let $\mu$ be an arbitrary partition in $B$. We define $E_{j}$ to be the \textit{j-restriction} functor from $\gs$ to $\g(\s_{n-1})$ and $F_{j}$ to be the \textit{j-induction} functor from $\gs$ to $\g(\s_{n+1})$ in the following way:
$$ E_{j}([M]):= [M{\downarrow}_{A}],$$
$$ F_{j}([M]):= [M{\uparrow}^{C}].$$

Similarly, we define $\bar{E}_{j}$ to be the \textit{j-restriction} functor from $\gsz$ to $\g(\s^{0}_{n-1})$ and $\bar{F}_{j}$ to be the \textit{j-induction} functor from $\gsz$ to $\g(\s^{0}_{n+1})$ in the following way:
$$ \bar{E}_{j}( [M]):= [M{\downarrow}_{A}],$$
$$ \bar{F}_{j}( [M]):= [M{\uparrow}^{C}].$$
It is easy to check using Theorem \ref{branching rule} that $$d_{\s}\bar{E}_{j}([W^{\mu}_{\C}])=E_{j}d_{\s}([W^{\mu}_{\C}]),$$
$$d_{\s}\bar{F}_{j}([W^{\mu}_{\C}])=F_{j}d_{\s}([W^{\mu}_{\C}]).$$ 
Since $\{[W^{\mu}_{\C}] \mid \mu\in\p(n)\}$ is a basis for $\gsz$, the following diagrams commute.
$$\bfig
\square(0,0)/>`>`>`>/[\gsz`\gs`\g(\s^{0}_{n-1})`\g(\s_{n-1});d_{\s}`\bar{E}_{j}`E_{j}`d_{\s}]
\square(1200,0)/>`>`>`>/[\gsz`\gs`\g(\s^{0}_{n+1})`\g(\s_{n+1});d_{\s}`\bar{F}_{j}`F_{j}`d_{\s}]
\efig$$
Since the adjustment matrices are unitriangular, $d_{\s}$ must be bijective; we shall identify $\gsz$ with its image under $d_{\s}$. Under this identification, we get $\bar{E}_j = E_j$ and $\bar{F}_j = F_j$ from the commutative diagrams. In particular,
$$E_{j}([L_{\F}^{\mu}])=[L_{\F}^{\mu}{\downarrow}_{A}],$$ $$F_{j}([L_{\F}^{\mu}])=[L_{\F}^{\mu}{\uparrow}^{C}], $$ 
$$E_{j}([L_{\C}^{\mu}])=[L_{\C}^{\mu}{\downarrow}_{A}],$$  $$F_{j}([L_{\C}^{\mu}])=[L_{\C}^{\mu}{\uparrow}^{C}].$$
Let $t$ be a positive integer. We denote the \textit{divided power j-restriction functor} and \textit{divided power j-induction functor} as $E_{j}^{(t)}:=\frac{1}{t!} E_{j}^{t}$ and $F_{j}^{(t)}:=\frac{1}{t!} F_{j}^{t}$ respectively.

\begin{prop}\label{lowerable}
Let $\lambda$ and $\mu$ be two distinct partitions lying in some block $B$ of $\s_{n}$. If $\epsilon_{j}(\lambda)< \epsilon_{j}(\mu)$ (resp. $\varphi_{j}(\lambda)< \varphi_{j}(\mu)$) for some $e$-residue $j$, then $\textnormal{adj}_{\lambda\mu} = 0$. 
\\
If $\epsilon_{j}(\lambda)=\epsilon_{j}(\mu)$ (resp. $\varphi_{j}(\lambda)= \varphi_{j}(\mu)$), for some $e$-residue $j$, then
$\textnormal{adj}_{\lambda\mu} = \textnormal{adj}_{\tilde{E}_{j}^{k}\lambda\tilde{E}_{j}^{k}\mu}$ (resp. $\adj_{\la\mu}=\textnormal{adj}_{\tilde{F}_{j}^{k}\lambda\tilde{F}_{j}^{k}\mu}$), where $k:=\epsilon_{j}(\lambda)$ (resp. $k:=\varphi_{j}(\lambda)$).
\end{prop}
\begin{proof}
Let $l:=\epsilon_{j}(\mu)$ and $k:=\epsilon_{j}(\lambda)$. If $l>k$, then
$$E_{j}^{(l)}([L_{\C}^{\la}])=\sum\limits_{\nu \in \mathcal{P}(n)}\textnormal{adj}_{\lambda\nu}E_{j}^{(l)}[L_{\F}^{\nu}].$$
By Theorem \ref{modular branching rule}, $E_{j}^{(l)}([L_{\C}^{\la}])=0$ and $E_{j}^{(l)}[L_{\F}^{\mu}]\neq0$, so $\textnormal{adj}_{\lambda\mu}$ must be zero.

If $l=k$, then
\begin{equation}\label{equation 1}
[L_{\mathbb{C}}^{\tilde{E}_{j}^{k}\lambda}]=E_{j}^{(k)}([L_{\C}^{\la}])=\sum\limits_{\nu\in\p(n)}\textnormal{adj}_{\lambda\nu}E_{j}^{(k)}[L_{\F}^{\nu}]=\sum_{\substack{\nu \in\p(n), \\ \epsilon_{j}(\nu) \ge k}} \adj_{\lambda\nu} E_{j}^{(k)}[L_{\F}^{\nu}]=\sum_{\substack{\nu \in\p(n), \\ \epsilon_{j}(\nu) = k}} \adj_{\lambda\nu}[L_{\mathbb{F}}^{\tilde{E}_{j}^{k}\nu}],
\end{equation}
where the third equality is due to Theorem \ref{modular branching rule} and the last equality is due to the case $l>k$ that we just proved above.
On the other hand,
\begin{equation}\label{equation 2}
[L_{\mathbb{C}}^{\tilde{E}_{j}^{k}\lambda}]=\sum\limits_{\sigma\in\p(n-k)}\textnormal{adj}_{\tilde{E}_{j}^{k}\lambda,\sigma}[L_{\mathbb{F}}^{\sigma}].
\end{equation}
Comparing the coefficients of $[L_{\mathbb{F}}^{\tilde{E}_{j}^{k}\mu}]$ in equations \ref{equation 1} and \ref{equation 2}, we conclude that $\textnormal{adj}_{\lambda\mu} = \textnormal{adj}_{\tilde{E}_{j}^{k}\lambda\tilde{E}_{j}^{k}\mu}$.
The proof of the other case considering conormal nodes is similar.
\end{proof}

\begin{remark}\label{remark lowerable}
In sections \ref{section wt 3} and \ref{section wt 4}, we will want to prove that $\textnormal{adj}_{\lambda\mu}=\delta_{\lambda\mu}$ for all pairs of partitions $(\lambda,\mu)$ lying in a block of $\s_{n}$ of weight $w$, assuming that James's conjecture holds for all blocks of $\s_{m}$ of weight at most $w$, where $m<n$.  If $\epsilon_{j}(\mu)>0$ and $\epsilon_{j}(\lambda)\le \epsilon_{j}(\mu)$ for some $e$-residue $j$, then we may apply Proposition \ref{lowerable} to conclude that $\textnormal{adj}_{\lambda\mu}=\delta_{\lambda\mu}$.
\end{remark}

\begin{definition}\label{definition lowerable}
Let $\lambda$ and $\mu$ be two distinct partitions lying in some weight $w$ block $B$ of $\s_{n}$. We say that the pair $(\lambda, \mu)$ is $lowerable$ if there is some $e$-residue $j$ such that $\epsilon_{j}(\mu)>0$, $\varphi_{j}(\mu)>0$ and $\epsilon_{j}(\lambda)\le \epsilon_{j}(\mu)$.
\end{definition}

\begin{cor}\label{corollary lowerable}
Suppose that $\la$ and $\mu$ are two distinct partitions lying in some weight $w$ block $B$ of $\s_{n}$ and that $(\la,\mu)$ is lowerable. Moreover, suppose that the adjustment matrix is the identity matrix for blocks of weight less than $w$. Then $\textnormal{adj}_{\lambda\mu} = 0$.
\end{cor}
\begin{proof}
Let $k:=\epsilon_{j}(\mu)$ and $l:=\varphi_{j}(\mu)$. If $\epsilon_{j}(\lambda)<k$, this follows directly from Proposition \ref{lowerable}.
If $\epsilon_{j}(\lambda)=k$, then we have $\textnormal{adj}_{\lambda\mu} = \textnormal{adj}_{\tilde{E}_{j}^{k}\lambda\tilde{E}_{j}^{k}\mu}$ by Proposition \ref{lowerable}. By Lemma \ref{lemma w-kl}, $\tilde{E}_{j}^{k}\lambda$ has weight $w-kl<w$, so the result follows.
\end{proof}

Note that our notion of lowerable partitions here is inspired by and generalises Fayers's definition of lowerable partitions in~\cite[Proposition 2.17]{wt 4}.

\begin{example}
Suppose that $e=8$ and $w=5$. Let $\la=\langle 1,2,3,4,5\rangle$ and $\mu= \langle 2,3,4,5,6 \rangle$ be the partitions lying in the block $B$ of $\h_{n}$ with the $\langle 2^{e}\rangle$ notation. 
\[
\begin{array}{c@{\qquad}c}
\la&\mu \\
\abacus(bbbbbbbb,bnnnnnbb,nbbbbbnn,nnnnnnnn)&
\abacus(bbbbbbbb,bbnnnnnb,nnbbbbbn,nnnnnnnn)
\end{array}
\]
Observe that $\epsilon_{2}(\mu)=1$, $\varphi_{2}(\mu)=1$ and $\epsilon_{2}(\la)=0$. Hence, $(\la,\mu)$ is lowerable by Theorem \ref{James Conjecture hecke 4}.
\end{example}

\subsection{[\textit{w:k}]-\emph{pairs}}\label{section wk}
Let $B$ be a weight $w$ block of $\s_{n}$ whose core $\kappa_{B}$ has exactly $k$ removable nodes on a given runner $i$ with residue $j$ (if $0<i<e$, runner $i$ has exactly $k$ more beads than runner $i-1$. If $i=0$, runner $0$ has exactly $k+1$ more beads than runner $e-1$). Suppose that $A$ is the weight $w$ block of $\s_{n-k}$ whose core $\kappa_{A}$ is obtained from $\kappa_{B}$ by removing all of the $k$ $j$-removable nodes. We say that the blocks $A$ and $B$ form a $[w:k]$-pair. We note that for each partition $\mu$ in $A$ (resp. $B$), we have $\varphi_{j}(\mu)-\epsilon_{j}(\mu)=k$ (resp. $\epsilon_{j}(\mu)-\varphi_{j}(\mu)=k$).

Given a partition $\mu$ in $A$, recall that $\tilde{F}_{j}^{k}\mu$ is the partition in $B$ obtained from $\mu$ by adding the $k$ lowest $j$-conormal nodes. Then, $\tilde{F}_{j}^{k}$ is a bijection from the set of partitions in $A$ to the set of partitions in $B$. Moreover, we have the following:
\begin{thrm}~\cite{brundan}
Let $\mu$ be a partition in $A$. Then, $\varphi_{j}(\mu)=k$ (equivalently $\epsilon_{j}(\mu)=0$) if and only if $\epsilon_{j}(\tilde{F}_{j}^{k}\mu)=k$, in which case, $L^{\mu}{\uparrow}^{B}\cong(L^{\tilde{F}_{j}^{k}\mu})^{\oplus k!}$ and $L^{\tilde{F}_{j}^{k}\mu}{\downarrow}_{A}\cong(L^{\mu})^{\oplus k!}$. If this happens, we say that $\mu$ and $\tilde{F}_{j}^{k}\mu$ are non-exceptional for the $[w:k]$-pair $(A,B)$. We say that $\mu$ and $\tilde{F}_{j}^{k}\mu$ are exceptional otherwise.
\end{thrm}
When $w\le k$, every partition is non-exceptional for the $[w:k]$-pair $(A,B)$, and we say that $A$ and $B$ are \textit{Scopes equivalent}; they are in fact Morita equivalent~\cite{Scopes}.

\begin{remark}\label{remark excep}
Let $\la$ be a partition in $A$. 
\begin{itemize}
\item If $\la$ were exceptional, then $\varphi_{j}(\la)>k$ (equivalently $\epsilon_{j}(\la)>0$).
\item If $\la$ were non-exceptional, then $\varphi_{j}(\la)=k$ (equivalently $\epsilon_{j}(\la)=0$).
\end{itemize}
Let $\sigma$ be a partition in $B$. 
\begin{itemize}
\item If $\sigma$ were exceptional, then $\epsilon_{j}(\sigma)>k$ (equivalently $\varphi_{j}(\sigma)>0$).
\item If $\sigma$ were non-exceptional, then $\epsilon_{j}(\sigma)=k$ (equivalently $\varphi_{j}(\sigma)=0$).
\end{itemize}
\end{remark}

\begin{definition}\label{new notation}
Let $\mu\in\p(n)$ and let $j$ be some $e$-residue. If $\varphi_{j}(\mu)-\epsilon_{j}(\mu)=k>0$, we define $\dot{F}_{j}\mu$ to be $\tilde{F}_{j}^{k}\mu$ ($\dot{F}_{j}\mu$ is defined only when $\varphi_{j}(\mu)-\epsilon_{j}(\mu)>0$). %; we say that $\dot{F}_{j}\mu$ is defined if and only if $\varphi_{j}(\mu)-\epsilon_{j}(\mu)>0$. 
Given a positive integer $m$, we define $\dot{F}_{j}{\searrow}_{m}\mu$ and $\dot{F}_{j}{\nearrow}^{m}\mu$ recursively in the following way:
\begin{itemize}
\item If $\varphi_{j}(\mu)-\epsilon_{j}(\mu)>0$, $\dot{F}_{j}{\searrow}_{1}\mu:=\dot{F}_{j}\mu$ ($\dot{F}_{j}{\searrow}_{1}\mu$ is not defined when $\varphi_{j}(\mu)-\epsilon_{j}(\mu)\le0$).
\item If $m>2$ and $\varphi_{j-m+1}(\dot{F}_{j}{\searrow}_{m-1}\mu)-\epsilon_{j-m+1}(\dot{F}_{j}{\searrow}_{m-1}\mu)>0$%(assuming that $\dot{F}_{j}{\searrow}_{m-1}\mu$ was defined)
, $\dot{F}_{j}{\searrow}_{m}\mu:=\dot{F}_{j-m+1}(\dot{F}_{j}{\searrow}_{m-1}\mu)$.
\item If $\varphi_{j}(\mu)-\epsilon_{j}(\mu)>0$, $\dot{F}_{j}{\nearrow}^{1}\mu:=\dot{F}_{j}\mu$ ($\dot{F}_{j}{\nearrow}^{1}\mu$ is not defined when $\varphi_{j}(\mu)-\epsilon_{j}(\mu)\le0$).
\item If $m>2$ and $\varphi_{j+m-1}(\dot{F}_{j}{\nearrow}^{m-1}\mu)-\epsilon_{j+m-1}(\dot{F}_{j}{\nearrow}^{m-1}\mu)>0$%(assuming that $\dot{F}_{j}{\searrow}_{m-1}\mu$ was defined)
, $\dot{F}_{j}{\nearrow}^{m}\mu:=\dot{F}_{j+m-1}(\dot{F}_{j}{\nearrow}^{m-1}\mu)$.
\end{itemize}
\end{definition}

\begin{prop}\label{proposition exceptional lambda}
Suppose that $A$ and $B$ are blocks forming a $[w:k]$-pair as above with $k<w$. Let $\lambda$ and $\mu$ be two distinct partitions in $A$. 
\begin{enumerate}
\item If $\lambda$ and $\mu$ are both non-exceptional, then $\textnormal{adj}_{\lambda\mu}=\textnormal{adj}_{\dot{F}_{j}\lambda,\dot{F}_{j}\mu}$.
\item If $\la$ is non-exceptional but $\mu$ is exceptional, then $\adj_{\la\mu}=0$.
\end{enumerate}
\end{prop}
\begin{proof}
\begin{enumerate}
\item By Remark \ref{remark excep}, $\varphi_{j}(\la)=\varphi_{j}(\mu)=k$, so $\textnormal{adj}_{\lambda\mu}=\textnormal{adj}_{\dot{F}_{j}\lambda,\dot{F}_{j}\mu}$ by Proposition \ref{lowerable}.
\item By Remark \ref{remark excep}, $\varphi_{j}(\la)=k$ and $\varphi_{j}(\mu)>k$, so $\textnormal{adj}_{\lambda\mu}=0$ by Proposition \ref{lowerable}. 
\end{enumerate}
\end{proof}

Suppose that we have a series of $t+1$ blocks $B_{m}$ of weight $w$ with cores $\kappa_{m}$, where $0\le m\le t$. Moreover, for $1\le m \le t$, $B_{m}$ and $B_{m-1}$ forms a $[w:k_{m}]$-pair with $\kappa_{m-1}$ being obtained from $\kappa_{m}$ by adding $k_{m}>0$ addable nodes of residue $j_{m}$. Given any partition $\lambda$ in $B_{t}$, $f(\la):=\dot{F}_{j_{1}}\dot{F}_{j_{2}}\cdots\dot{F}_{j_{t-1}}\dot{F}_{j_{t}}\la$ is well-defined. We say that $f(\lambda)$ is \textit{semisimply induced} if $\lambda$, $\dot{F}_{j_{t}}\lambda$, $\dot{F}_{j_{t-1}}\dot{F}_{j_{t}}\lambda$, $\dots$, $ \dot{F}_{j_{2}}\cdots\dot{F}_{j_{t-1}}\dot{F}_{j_{t}}\lambda$ and $f(\lambda)$ are all non-exceptional. We say that $f(\lambda)$ is \textit{not semisimply induced} otherwise. Using this notation, we may restate our working version of Proposition \ref{proposition exceptional lambda} in the following corollary.
\begin{cor}\label{working exceptional lambda}
We adopt the notation above. Let $\lambda$ and $\mu$ be two distinct partitions in $B_{t}$. 
\begin{enumerate}
\item If $f(\lambda)$ and $f(\mu)$ are both semisimply induced, then $\textnormal{adj}_{\lambda\mu}=\textnormal{adj}_{f(\lambda)f(\mu)}$. 
\item If $f(\lambda)$ is semisimply induced but $f(\mu)$ is not semisimply induced, then $\textnormal{adj}_{\lambda\mu}=0$.
\end{enumerate}
\end{cor}
\begin{proof}
We just apply Proposition \ref{proposition exceptional lambda} repeatedly.
\end{proof}
When we want to show that $\textnormal{adj}_{\lambda\mu}=0$ for some pair of partitions $(\lambda,\mu)$ in section \ref{section wt 4}, we sometimes do this by finding a sequence of $e$-residues $j_{1},\dots,j_{t}$ such that $f(\la):=\dot{F}_{j_{1}}\cdots\dot{F}_{j_{t}}\la$ is semisimply induced and $(f(\lambda),f(\mu))$ is lowerable. We may also use this formalism without making $f$ explicit by writing $\lambda\sim\nu$ (and say that $\lambda$ \textit{induces semisimply} to $\nu$) to indicate that there exists a sequence of $e$-residues $j_{1},\dots,j_{t}$ such that $f(\la):=\dot{F}_{j_{1}}\cdots\dot{F}_{j_{t}}\la$ is semisimply induced and equals $\nu$. When we do so, it is hoped that it will not be too difficult for the reader to construct an appropriate sequence $j_{1},\dots,j_{t}$.

\begin{example}\label{example induce to lowerable}
Let $\la=\langle0_{1^{2}},3_{1^{2}} \rangle$, $\mu=\langle 0,3_{2},4\rangle$ and $\nu=\langle 0_{2},3_{1^{2}}\rangle$ be the partitions lying in the weight 4 block $B$ of $\s_{n}$ with the $\langle 4^{3},5^{2},4^{2} \rangle$ notation (the beads on runner $0$ have $e$-residue $5$).
For any partition $\sigma$ in $B$, define $\mathfrak{a}(\sigma):=\dot{F}_{5}{\nearrow}^{3}\sigma=\dot{F}_{0}\dot{F}_{6}\dot{F}_{5}\sigma$.

\[
\begin{array}{c@{\qquad}c@{\qquad}c}
\la&\mu&\nu \\
\abacus(bbbbbbb,bbbbbbb,nbbbbbb,bbbnbbb,bnnbbnn,nnnbnnn,nnnnnnn)&
\abacus(bbbbbbb,bbbbbbb,bbbbbbb,nbbbbbb,bnnnnnn,nnnnbnn,nnnbnnn)&
\abacus(bbbbbbb,bbbbbbb,bbbbbbb,nbbnbbb,nnnbbnn,bnnbnnn,nnnnnnn)
\end{array}
\]
\[
\begin{array}{c@{\qquad}c}
\mathfrak{a}(\la)&\mathfrak{a}(\mu) \\
\abacus(bbbbbbb,bbbbbbn,bbbbbbb,bbbnbbb,nnbbbnn,nnnbnnn,nnnnnnn)&
\abacus(bbbbbbb,bbbbbbb,bbbbbbn,bbbbbbb,nnbnnnn,nnnnbnn,nnnbnnn)
\end{array}
\]
Observe that $\mathfrak{a}(\la)$ and $\mathfrak{a}(\mu)$ are semisimply induced and that $(\mathfrak{a}(\la),\mathfrak{a}(\mu))$ is lowerable. However, $\mathfrak{a}(\nu)$ is not semisimply induced.
\end{example}

We state the following observation from the list of exceptional partitions of weights 3 and 4 in the appendix (Figure \ref{figure exceptional 31}, Figure \ref{figure exceptional 32} and Figure \ref{figure exceptional 41}).
\begin{remark}\label{remark excep same}
Adopting the same notation as the beginning of this section, let $A$ and $B$ be blocks forming a $[w:k]$-pair with $k<w$ and $w\in\{3,4\}$. Let $\la$ be a partition in $A$. 
\begin{itemize}
\item If $\la$ were exceptional, then $\varphi_{j}(\la)=k+1$ (equivalently $\epsilon_{j}(\la)=1$).
\item If $\la$ were non-exceptional, then $\varphi_{j}(\la)=k$ (equivalently $\epsilon_{j}(\la)=0$).
\end{itemize}
Let $\sigma$ be a partition in $B$. 
\begin{itemize}
\item If $\sigma$ were exceptional, then $\epsilon_{j}(\sigma)=k+1$ (equivalently $\varphi_{j}(\sigma)=1$).
\item If $\sigma$ were non-exceptional, then $\epsilon_{j}(\sigma)=k$ (equivalently $\varphi_{j}(\sigma)=0$).
\end{itemize}
\end{remark}

Therefore, we have the following result specific to blocks of $\s_{n}$ of weights 3 and 4.
\begin{prop}\label{proposition non-exceptional mu}
Suppose that $A$ and $B$ are blocks forming a $[w:k]$-pair as above, with $k<w$ and $w\in\{3,4\}$. Let $\lambda$ and $\mu$ be two distinct partitions in $A$.
If $\mu$ is exceptional, then both pairs of partitions $(\lambda,\mu)$ and $(\dot{F}_{j}\lambda,\dot{F}_{j}\mu)$ are lowerable.
\end{prop}
\begin{proof}
This follows directly from Remark \ref{remark excep same} and Definition \ref{definition lowerable}.
\end{proof}

\subsection{Rouquier blocks}\label{section rock}
A weight $w$ block $B$ of $\s_{n}$ with the $\langle b_{0},\dots,b_{e-1}\rangle$ notation is \textit{Rouquier} if for every $0\le i<j\le e-1$, either $b_{i}-b_{j}\ge w$ or $b_{j}-b_{i}\ge w-1$. The Rouquier blocks are Scopes equivalent to each other and are now well understood. In particular, we know that James's Conjecture holds for Rouquier blocks.
\begin{thrm}~\cite[Corollary 3.15]{rouquier}~\cite{filtrations}\label{James Conjecture Rock}
If $B$ is a Rouquier block of weight $w<\textnormal{char}(\mathbb{F})$, then $\textnormal{adj}_{\lambda\mu}=\delta_{\lambda\mu}$ for all $\lambda$ and $\mu$ in $B$.
\end{thrm}
We say that a partition $\lambda$ \textit{induces semi-simply to a Rouquier block} if $\lambda\sim\nu$ for some $\nu$ lying in a Rouquier block. As a consequence of Corollary \ref{working exceptional lambda} and Theorem \ref{James Conjecture Rock}, we have the following result.
\begin{prop}\label{induce semisimply}
Suppose that $\lambda$ and $\mu$ are partitions lying in a block $B$ of $\s_{n}$ of weight $w<\textnormal{char}(\mathbb{F})$. If $\lambda$ induces semi-simply to a Rouquier block, then $\textnormal{adj}_{\lambda\mu}=\delta_{\lambda\mu}$.
\end{prop}
In sections \ref{section wt 3} and \ref{section wt 4}, we will sometimes state a partition $\nu$ lying in a Rouquier block and invite the reader to verify that $\lambda\sim\nu$ for some partition $\lambda$ for which we want to show that $\textnormal{adj}_{\lambda\mu}=\delta_{\lambda\mu}$.

\subsection{Outline of the proof of James's Conjecture for the blocks of $q$-Schur algebras of weights 3 and 4}\label{outline of proof}
From now on, we assume that $\textnormal{char}(\mathbb{F})\ge5$ and $3\le w\le4$. We prove Theorem \ref{James Conjecture schur 3} and Theorem \ref{James Conjecture schur 4} by induction on $n$, with the base case being the unique weight $w$ block of $\s_{we}$. We will deal with the base case at the beginning of sections \ref{section wt 3} and \ref{section wt 4}.

For the inductive step, we use $[w:k]$-pairs. If $B$ is a weight $w$ block of $\s_{n}$ and $n>we$, then there is at least one block $A$ forming a $[w:k]$-pair with $B$. Suppose that $A_{1},\dots,A_{t}$ are all the blocks with $A_{m}$ forming a $[w:k_{m}]$-pair with $B$, for each $m$. If $\lambda$ is a partition in $B$ which is non-exceptional for some pair $(A_{m},B)$, then $\textnormal{adj}_{\lambda\mu}=\delta_{\lambda\mu}$ for every partition $\mu$ lying in $B$ by Proposition \ref{proposition exceptional lambda} and induction.

Therefore, we may assume that $\lambda$ is exceptional for every pair $(A_{m},B)$. Suppose that an abacus display for the core of $A_{m}$ is obtained from that of $B$ by removing $k_{m}$ removable nodes on runner $i_{m}$, for each $m$. If $\lambda$ were exceptional for $(A_{m},B)$, then there must be at least $k_{m}+1$ normal beads on runner $i_{m}$ in the abacus display for $\lambda$. Hence, $|\lambda(i_{m})|\ge k_{m}+1$ for each $m$ (see Figure \ref{figure exceptional 31}, Figure \ref{figure exceptional 32} and Figure \ref{figure exceptional 41}), so we have $(k_{1}+1)+\dots+(k_{t}+1)\le|\lambda(i_{1})|+\dots+|\lambda(i_{t})|\le w$. When $w=3$, this implies that $t=1$ and $k_{1}\le2$. When $w=4$, this implies that either $t=1$ and $k_{1}\le3$ or $t=2$ and $k_{1}=k_{2}=1$. The blocks $B$ satisfying these conditions are dealt with in the rest of the paper.

By Remark \ref{remark row removal}, we may always assume that $\mu\gg\lambda$; that is $\mu>_{P}\lambda$, $\mu_{1}>\lambda_{1}$ and $l(\mu)<l(\lambda)$. Additionally, for each $e$-residue $j$ such that $\epsilon_{j}(\mu)>0$, we may assume that $\epsilon_{j}(\la)<\epsilon_{j}(\mu)$ by Remark \ref{remark lowerable}. Finally, we may assume that $\mu$ is $e$-singular by Proposition \ref{e-regular mu} and Theorem \ref{James Conjecture hecke 4}.
\section{Proof of James's Conjecture for weight 3 blocks of $\s_{n}$}\label{section wt 3}
In this section, we will first prove Theorem \ref{James Conjecture schur 3} and then use it to show that the decomposition numbers for weight 3 blocks of $\s_{n}$ are bounded above by one. Whenever $\lambda$ and $\mu$ are partitions with weight less than 3, $\textnormal{adj}_{\lambda\mu}=\delta_{\lambda\mu}$ by Theorem \ref{James Conjecture schur 2}. We will use this fact repeatedly without further comment.
\subsection{The principal block of \texorpdfstring{$\s_{3e}$}{TEXT}}\label{sec 3e}
Let $B$ be the principal block of $\s_{3e}$; that is the weight $3$ block which we display on an abacus with the $\langle 3^{e}\rangle$ notation.

\begin{prop} \label{weight 3 empty core}
Suppose that $\lambda$ and $\mu$ are partitions lying in $B$. Then, $\textnormal{adj}_{\lambda\mu}=\delta_{\lambda\mu}$.
\end{prop}
\begin{proof}
Let $\la$ be an arbitrary partition lying in $B$. Observe that each runner of $\la$ has at most one normal bead. By Remark \ref{remark lowerable}, we may assume that $\mu$ has no normal beads on any runner. By Remark \ref{remark row removal}, we may also assume that $\mu\gg\lambda$. Table \ref{table 3e} (in the appendix) lists all the possible pairs of partitions $(\lambda, \mu)$. We invite the reader to check that every partition in column $\lambda$ of Table \ref{table 3e} induces semi-simply to a Rouquier block.
\begin{itemize}
\item $\langle 0_{1^{3}} \rangle \sim \langle 0_{1^{3}} \mid 3,5,\dots, 2e+1 \rangle,$
\item $\langle 1_{1^{3}} \rangle \sim \langle 0_{1^{2}},1 \mid 3,5,\dots, 2e+1 \rangle,$
\item $\langle 0_{1^{2}},1 \rangle \sim \langle 0_{2,1} \mid 3,5,\dots, 2e+1 \rangle,$
\item $\langle  0, 1_{1^{2}} \rangle \sim \langle 0_{2},1 \mid 3,5,\dots, 2e+1 \rangle.$
\end{itemize}
Proposition \ref{induce semisimply} completes the proof of Proposition \ref{weight 3 empty core}.
\end{proof}

\begin{comment}
\begin{table}[h]
\centering
\caption{}
\label{table 3e}
\begin{tabular}{|p{3cm}|p{2.5cm}|} \hline
$\mu$ & $\lambda$   \\ \hline
$\langle    0_{1^{3}} \rangle $ & NIL \\ \hline
$\langle    0_{2},e-1      \rangle$, $e=2$ & $\langle    0_{1^{3}} \rangle $ \\ 
 & $\langle    1_{1^{3}} \rangle $  \\
 & $\langle 0_{1^{2}},1 \rangle $  \\
 & $\langle 0, 1_{1^{2}} \rangle $\\ \hline
$\langle 0_{1^{2}},1       \rangle$ & $\langle    0_{1^{3}} \rangle $  \\ \hline
$\langle 0,1,2 \rangle$, $e\ge3$ & $\langle    0_{1^{3}} \rangle $  \\
 & $\langle    1_{1^{3}} \rangle $  \\
 & $\langle 0_{1^{2}},1 \rangle $  \\
 & $\langle 0, 1_{1^{2}} \rangle $  \\ \hline
\end{tabular}
\end{table}
\end{comment}

\subsection{Blocks forming exactly one $[3:1]$-pair}
In this section, we prove the following proposition:
\begin{prop}\label{weight 3,1}
Suppose that $A$ and $B$ are weight $3$ blocks of $\s_{n-1}$ and $\s_{n}$ respectively, forming a $[3:1]$-pair. Moreover, suppose that there is no block other than $A$ forming a $[3:k]$-pair with $B$ for any $k$. Additionally, suppose that the adjustment matrix for every weight $3$ block of $\s_{m}$ is the identity matrix whenever $m<n$. Then, the adjustment matrix for $B$ is the identity matrix.
\end{prop}
The conditions on $A$ and $B$ mean that we may $B$ on an abacus with the $\langle 3^{a},4^{b-a},3^{e-b}\rangle$ notation, where $0<a<b\le e$. Suppose that $\lambda$ and $\mu$ are two distinct partitions lying in $B$, so we want to prove that $\textnormal{adj}_{\lambda\mu}=0$. By Proposition \ref{proposition exceptional lambda} and Proposition \ref{proposition non-exceptional mu}, we may assume that $\lambda$ is exceptional and that $\mu$ is non-exceptional. We list all the exceptional partitions in $B$ below:
\begin{itemize}
\item $\langle a_{2,1}\rangle$, 
\item $\langle a_{1^{2}},i \rangle$, $a<i<e$,
\item $\langle i,a_{1^{2}} \rangle$, $0\le i \le a-2$,
\item $\langle a_{1^{3}} \rangle$.
\end{itemize}
We invite the reader to check that $\langle a_{1^{3}} \rangle$ induces semi-simply to a Rouquier block, so we may assume that $\lambda\neq\langle a_{1^{3}} \rangle$ by Proposition \ref{induce semisimply}.
\begin{equation*}
\langle a_{1^{3}} \rangle \sim
\begin{cases}
\langle 0,a,a+e-b \mid 3,5,\dots,2e+1\rangle & \text{if          } e-b>0,\\
\langle 0,a_{1^{2}} \mid 3,5,\dots,2e+1\rangle & \text{if          } e-b=0.
\end{cases}
\end{equation*}

\[
\begin{array}{c@{\qquad}c}
\langle 0, a_{1^{2}}\rangle &\langle a_{2,1}\rangle\\
\abacus(nbhbnbhbbhb,bnhnbbhbnhn,nnhnbnhnnhn,nnhnnnhnnhn,nnhnnnhnnhn)&
\abacus(bhbnbhbbhb,nhnbbhbnhn,nhnnnhnnhn,nhnbnhnnhn,nhnnnhnnhn)
\end{array}
\]

Let $\nu:=\langle0,a_{1^{2}}\rangle$. Note that $l(\nu)\ge l(\la)$ for all the remaining possible $\la$, so we may assume that $l(\mu)<l(\nu)$ by Remark \ref{remark row removal}. Since the first unoccupied position in the abacus display for $\la$ occurs at position $0$, the first unoccupied position in an abacus display for $\mu$ must occur strictly after position $0$. Consequently, all beads in runner $0$ of an abacus display for $\mu$ have zero weight, therefore $\mu$ must be $e$-regular and $\adj_{\la\mu}=\delta_{\la\mu}$ by Proposition \ref{e-regular mu} and Theorem \ref{James Conjecture hecke 4}. This completes the proof of Proposition \ref{weight 3,1}.

\subsection{Blocks forming exactly one $[3:2]$-pair}
\begin{prop} \label{weight 3,2}
Suppose that $A$ and $B$ are weight $3$ blocks of $\s_{n-2}$ and $\s_{n}$ respectively, forming a $[3:2]$-pair. Moreover, suppose that there is no block other than $A$ forming a $[3:k]$-pair with $B$ for any $k$. Additionally, suppose that the adjustment matrix for every weight $3$ block of $\s_{m}$ is the identity matrix whenever $m<n$. Then, the adjustment matrix for $B$ is the identity matrix.
\end{prop}
The conditions on $B$ mean that we may represent $B$ on an abacus with the $\langle 3^{a},5^{b-a},4^{c-b},3^{e-c}\rangle$ notation, where $0<a<b\le c\le e$. If $\lambda$ and $\mu$ are two distinct partitions in $B$, then we have $\textnormal{adj}_{\lambda\mu}=0$ by Proposition \ref{proposition exceptional lambda} unless $\lambda$ is the unique exceptional partition for $(A,B)$, namely $\lambda=\langle a_{1^{3}}\rangle$. 
\[
\begin{array}{c}
\la \\
\abacus(bhbnbhbbhbbhb,nhnbbhbbhbnhn,nhnbbhbnhnnhn,nhnbnhnnhnnhn,nhnnnhnnhnnhn)
\end{array}
\]

Observe that the first unoccupied position in the abacus display for $\la$ occurs at position $a$. By Remark \ref{remark row removal}, we may assume that $l(\mu)<l(\la)$, so the first unoccupied position in an abacus display for $\mu$ must occur after position $a$. Consequently, all beads in runner $0$ of an abacus display for $\mu$ have zero weight, therefore $\mu$ must be $e$-regular and $\adj_{\la\mu}=\delta_{\la\mu}$ by Proposition \ref{e-regular mu} and Theorem \ref{James Conjecture hecke 4}. 

As discussed in section \ref{outline of proof}, the combination of Proposition \ref{weight 3 empty core}, Proposition \ref{weight 3,1} and Proposition \ref{weight 3,2} completes the proof of Theorem \ref{James Conjecture schur 3}.

In his weight 3 paper, Fayers proved an upper bound for the decomposition numbers of $\mathcal{H}_{n}$. 
\begin{thrm}~\cite[Theorem 1.1]{wt 3}\label{d<2 hecke weight 3}
Suppose that $\textnormal{char}(\mathbb{F})\ge5$ and that $B$ is a block of $\h_{n}$ of weight 3. Let $\lambda$ and $\mu$ be partitions in $B$, with $\mu$ being $e$-regular. Then, $$[S^{\lambda}:D^{\mu}]\le1.$$
\end{thrm}
An easy consequence of Theorem \ref{James Conjecture schur 3} is that we may extend this upper bound to the case of the $q$-Schur algebras.
\begin{cor}\label{d<2 schur weight 3}
Suppose that $\textnormal{char}(\mathbb{F})\ge5$ and that $B$ is a block of $\s_{n}$ of weight 3. Let $\lambda$ and $\mu$ be partitions in $B$. Then, $$[W^{\lambda}:L^{\mu}]\le1.$$
\end{cor}
\begin{proof}
If $\mu$ is $e$-regular, then $[W^{\lambda}:L^{\mu}]=[S^{\lambda}:D^{\mu}]$ by Theorem \ref{schur functor d}, so we are done by Theorem \ref{d<2 hecke weight 3}. If $\mu$ is $e$-singular, we display $\lambda$ and $\mu$ on an abacus with $e$ runners and $r$ beads, for some $r$ large enough. Then, we define two abacus displays with $e+1$ runners each by adding a runner with every space unoccupied to the right of all the existing runners in the abacus displays for $\lambda$ and $\mu$ (see example \ref{example empty runner}). Let $\lambda^{+}$ and $\mu^{+}$ be the partitions corresponding to these two new abacus displays. Theorem \ref{Runner Removal Theorem Empty} applies and so we have $d^{e}_{\lambda\mu}(v)=d^{e+1}_{\lambda^{+}\mu^{+}}(v)$. Moreover, $\mu^{+}$ is $(e+1)$-regular, so $[W^{\lambda}_{\mathbb{C}}:L^{\mu}_{\mathbb{C}}]=[W^{\lambda^{+}}_{\mathbb{C}}:L^{\mu^{+}}_{\mathbb{C}}]=[S^{\lambda^{+}}_{\mathbb{C}}:D^{\mu^{+}}_{\mathbb{C}}]\le1$ by Theorem \ref{schur functor d}, Theorem \ref{v-decomposition numbers} and Theorem \ref{d<2 hecke weight 3} (note that $\la^{+}$ and $\mu^{+}$ have weight 3). By Theorem \ref{James Conjecture schur 3}, $[W^{\lambda}_{\mathbb{C}}:L^{\mu}_{\mathbb{C}}]=[W^{\lambda}_{\mathbb{F}}:L^{\mu}_{\mathbb{F}}]$.
\end{proof}
\section{Proof of James's Conjecture for weight 4 blocks of $\s_{n}$}\label{section wt 4}
We shall prove Theorem \ref{James Conjecture schur 4} in this section. Whenever $\lambda$ and $\mu$ are partitions with weight less than 4, $\textnormal{adj}_{\lambda\mu}=\delta_{\lambda\mu}$ by Theorem \ref{James Conjecture schur 2} and Theorem \ref{James Conjecture schur 3}. We will use this fact repeatedly without further comment.
\subsection{The principal block of \texorpdfstring{$\s_{4e}$}{TEXT}}\label{sec 4e}
Let $B$ be the principal block of $\s_{4e}$; that is the weight $4$ block which we display on an abacus with the $\langle 4^{e}\rangle$ notation.
\begin{prop}\label{empty core weight 4}
Suppose that $\lambda$ and $\mu$ are partitions lying in $B$. Then, $\textnormal{adj}_{\lambda\mu}=\delta_{\lambda\mu}$.
\end{prop}
\begin{proof}
Observe that when $\lambda = \langle i_{2^{2}} \rangle$ and $0<i<e$, $\lambda$ has two normal beads on runner $i$ and no normal beads on every other runner. By Remark \ref{remark lowerable}, we may assume that $\mu$ has at most one normal bead on runner $i$ and no normal beads on every other runner. By Remark \ref{remark row removal}, we may also assume that $\mu\gg\lambda$. We list all the possiblities for $\mu$ below:
\begin{itemize}
\item $\langle i_{4} \rangle$,
\item $\langle i_{3}, i+1 \rangle$, $i+1<e$,
\item $\langle 0, i_{3} \rangle$,
\item $\langle i_{2}, (i+1)_{2} \rangle$, $i+1<e$.
\end{itemize}
Notice that every partition $\mu$ in the list above are $e$-regular, therefore $\textnormal{adj}_{\lambda\mu}=\delta_{\lambda\mu}$ by Proposition \ref{e-regular mu} and Theorem \ref{James Conjecture hecke 4}.

Hence, we may assume that $\lambda \neq \langle i_{2^{2}} \rangle$ for $0<i<e$. In this case, observe that in the abacus displays for the remaining possibilities for $\lambda$, no runner has more than one normal bead. By Remark \ref{remark lowerable}, we may assume that $\mu$ has no normal beads on any runner. By Remark \ref{remark row removal}, we may also assume that $\mu\gg\lambda$. Table \ref{table 4e 1} and Table \ref{table 4e 2} (in the appendix) list all the possible pairs of partitions $(\lambda, \mu)$. Note that for each $\mu$, we list the partitions $\lambda$ in descending lexicographic order.

We invite the reader to check that the following partitions induce semisimply to a Rouquier block, thus $\textnormal{adj}_{\lambda\mu}=\delta_{\lambda\mu}$ if $\lambda$ is one of those partitions by Proposition \ref{induce semisimply}.
\begin{itemize}
\item $\langle 0_{2},1_{1^{2}} \rangle \sim \langle 0_{3}, 1 \mid 4,7 \rangle$ when $e=2$,
\item $\langle 0_{2,1},1 \rangle \sim \langle 0_{3,1} \mid 4,7 \rangle$ when $e=2$,
\item $\langle 0,1,2_{1^{2}} \rangle \sim \langle 0_{3},2 \mid 4,7,\dots, 3e+1 \rangle$,	
\item $\langle 0,1_{1^{2}},2 \rangle \sim \langle 0_{3},1 \mid 4,7,\dots, 3e+1 \rangle$,
\item $\langle 0_{1^{2}},1,2 \rangle \sim \langle 0_{3,1} \mid 4,7,\dots, 3e+1 \rangle$,
\item $\langle 1_{1^{2}}, 2_{1^{2}} \rangle \sim \langle 0_{2}, 1_{2} \mid 4,7,\dots, 3e+1 \rangle$,
\item $\langle 1_{1^{3}}, 2 \rangle \sim \langle 0_{1^{2}},1_{2} \mid 4,7,\dots, 3e+1 \rangle$,
\item $\langle 0,2_{1^{3}} \rangle \sim \langle 0_{2,1},2 \mid 4,7,\dots, 3e+1 \rangle$,
\item $\langle 2_{1^{4}} \rangle \sim \langle 0_{1^{3}},2 \mid 4,7,\dots, 3e+1 \rangle$,
\item $\langle 0_{1^{2}}, 1_{1^{2}} \rangle \sim \langle 0_{2^{2}} \mid 4,7,\dots, 3e+1 \rangle$,
\item $\langle 0,1_{1^{3}} \rangle \sim \langle 0_{2,1},1 \mid 4,7,\dots, 3e+1 \rangle$,
\item $\langle 0_{1^{3}},1 \rangle \sim \langle 0_{2,1^{2}} \mid 4,7,\dots, 3e+1 \rangle$,
\item $\langle 1_{1^{4}} \rangle \sim \langle 0_{1^{3}},1 \mid 4,7,\dots, 3e+1 \rangle$,
\item $\langle 0_{1^{4}} \rangle \sim \langle 0_{1^{4}} \mid 4,7,\dots, 3e+1 \rangle$.
\end{itemize}

Let $\lambda^{1}$, $\lambda^{2}$, $\lambda^{3}$, $\mu^{1}$, $\mu^{2}$ and $\mu^{3}$ (see Table \ref{table 4e 1} and Table \ref{table 4e 2}) be the following partitions:
\begin{itemize}
\item $\lambda^{1}:=\langle    0_{1^{3}} ,2     \rangle$, $e=3$,
\item $\mu^{1}:=\langle   0_{2},1,2        \rangle$, $e=3$,
\item $\lambda^{2}:=\langle    0_{1^{3}} ,2     \rangle$, $e\ge4$,
\item $\mu^{2}:=\langle     0,1,2,3      \rangle$, $e\ge4$,
\item $\lambda^{3}:=\langle       0_{2^{2}}    \rangle$, $e=2$,
\item $\mu^{3}:=\langle   0_{2},1_{2}  \rangle$, $e=2$.
\end{itemize}
If we managed to show that $\textnormal{adj}_{\lambda^{1}\mu^{1}}=\textnormal{adj}_{\lambda^{2}\mu^{2}}=\textnormal{adj}_{\lambda^{3}\mu^{3}}=0$, the proof of Proposition \ref{empty core weight 4} would follow from Proposition \ref{decomposition number 0 or v}. 

We now relax the definition of $\la^{3}$ and $\mu^{3}$ slightly, so that $\la^{3}:=\langle 0_{2^{2}} \mid 4^{2}, 5^{e-2} \rangle$ and $\mu^{3} :=\langle 0_{2},1_{2}\mid 4^{2}, 5^{e-2} \rangle$, where $e\ge2$. We also define $\lambda^{0}:=\langle 0_{2,1^{2}} \mid 4^{2},5^{e-2} \rangle$, where $e\ge2$. Recall from Definition \ref{new notation} that $\dot{F}_{1}{\searrow}_{e-2}\dot{F}_{0}{\searrow}_{e-2}\nu=\dot{F}_{4}\dots\dot{F}_{0}\dot{F}_{1}\dot{F}_{3}\dots\dot{F}_{e-1}\dot{F}_{0}\nu$ is well-defined for partitions $\nu$ in the block $B$. When $i \in \{1,2\}$, we invite the reader to check (in an abacus display with the $\langle 4^{e}\rangle$ notation, the beads on runner $0$ have $e$-residue $0$) that $\dot{F}_{1}{\searrow}_{e-2}\dot{F}_{0}{\searrow}_{e-2}\la^{i}$ and $\dot{F}_{1}{\searrow}_{e-2}\dot{F}_{0}{\searrow}_{e-2}\mu^{i}$ are semisimply induced and moreover, $\dot{F}_{1}{\searrow}_{e-2}\dot{F}_{0}{\searrow}_{e-2}\la^{i}=\la^{0}$ and $\dot{F}_{1}{\searrow}_{e-2}\dot{F}_{0}{\searrow}_{e-2}\mu^{i}=\mu^{3}$. Therefore, we are left to show that $\textnormal{adj}_{\lambda^{0}\mu^{3}}=0$ and $\textnormal{adj}_{\lambda^{3}\mu^{3}}=0$ by Corollary \ref{working exceptional lambda}.

\begin{prop}\label{prop 4e difficult}
$\textnormal{adj}_{\lambda^{3}\mu^{3}}=0$.
\end{prop}
\begin{proof}
Let $f(\la^{3}):=\dot{F}_{2}{\searrow}_{e-2}\la^{3}$ and $f(\mu^{3}):=\dot{F}_{2}{\searrow}_{e-2}\mu^{3}$ (in an abacus display with the $\langle 4^{2},5^{e-2}\rangle$ notation, the beads on runner $0$ have $e$-residue $2$. When $e=2$, $f$ is the identity map). We observe that $f(\la^{3})$ and $f(\mu^{3})$ are both semisimply induced and moreover, $f(\la^{3})=\langle (e-1)_{2^{2}} \mid 6^{e-2},4,3 \rangle$ and $f(\mu^{3})=\langle (e-2)_{2}, (e-1)_{2} \mid 6^{e-2},4,3 \rangle$. Hence, $\adj_{\la^{3}\mu^{3}}=\adj_{f(\la^{3})f(\mu^{3})}$ by Corollary \ref{working exceptional lambda}. We also note that $d^{e}_{f(\la^{3})f(\mu^{3})}(v)=d^{2}_{\la^{3}\mu^{3}}(v)=v^{3}+v$ by Theorem \ref{Runner Removal Theorem Full}. Hence, $[W_{\mathbb{C}}^{f(\lambda^{3})}:L_{\mathbb{C}}^{f(\mu^{3})}]=2$ and it suffices to prove that $[W_{\mathbb{F}}^{f(\lambda^{3})}:L_{\mathbb{F}}^{f(\mu^{3})}]=2$ by Lemma \ref{lemma easy fact}.

\[
\begin{array}{c@{\qquad}c@{\qquad}c@{\qquad}c}
\la^{3}&f(\la^{3})&\mu^{3}&f(\mu^{3})\\
\abacus(bbbhb,bbbhb,nbbhb,nbbhb,bnbhb,bnnhn,nnnhn)&
\abacus(bhbbb,bhbbn,bhbbn,bhbbb,bhbnb,bhbnn,nhnnn)&
\abacus(bbbhb,bbbhb,bbbhb,nnbhb,nnbhb,bbnhn,nnnhn)&
\abacus(bhbbb,bhbbb,bhbbn,bhbnn,bhbnb,bhbbn,nhnnn)
\end{array}
\]

Let $B^{i}$ be the weight $4$ block with the $\langle 6^{e-i-2},3,6^{i},4 \rangle$ notation for $0\le i \le e-2$. We define $\la^{y}$, $\la^{x}$ and $\mu^{x}$ to be the partitions lying in $B^{0}$ by their abacus displays below. We may check using the modular branching rules (Theorem \ref{branching rule} and Theorem \ref{modular branching rule}) that 
$$W_{\mathbb{F}}^{f(\lambda^{3})}{\uparrow}^{B^{0}}\sim W_{\mathbb{F}}^{\la^{y}}+W_{\mathbb{F}}^{\la^{x}},$$
$$L_{\mathbb{F}}^{f(\mu^{3})}{\uparrow}^{B^{0}}\cong L_{\mathbb{F}}^{\mu^{x}}.$$

\[
\begin{array}{c@{\qquad}c@{\qquad}c}
\la^{y}&\la^{x}&\mu^{x}\\
\abacus(bhbbb,bhbbn,bhbnb,bhbbb,bhbnb,bhbnn,nhnnn)&
\abacus(bhbbb,bhbnb,bhbbn,bhbbb,bhbnb,bhbnn,nhnnn)&
\abacus(bhbbb,bhbbb,bhbnb,bhbnn,bhbnb,bhbbn,nhnnn)
\end{array}
\]

We define $\la^{y,x}$, $\la^{x,x}$ and $\mu^{x,x}$ to be the partitions lying in $B^{e-2}$ by their abacus displays below. We may check using the modular branching rules that
$$\delf^{\la^{y}}\uar_{B^{0}}^{B^{1}}\uar_{B^{1}}^{B^{2}}\uar\cdots\uar_{B^{e-3}}^{B^{e-2}}\sim \delf^{\la^{y,x}},$$
$$\delf^{\la^{x}}\uar_{B^{0}}^{B^{1}}\uar_{B^{1}}^{B^{2}}\uar\cdots\uar_{B^{e-3}}^{B^{e-2}}\sim \delf^{\la^{x,x}},$$
$$\LF^{\mu^{x}}\uar_{B^{0}}^{B^{1}}\uar_{B^{1}}^{B^{2}}\uar\cdots\uar_{B^{e-3}}^{B^{e-2}}\cong \LF^{\mu^{x,x}}.$$

\[
\begin{array}{c@{\qquad}c@{\qquad}c}
\la^{y,x}&\la^{x,x}&\mu^{x,x}\\
\abacus(bbhbb,bbhbn,nbhbb,bbhbb,nbhbb,nbhbn,nnhnn)&
\abacus(bbhbb,nbhbb,bbhbn,bbhbb,nbhbb,nbhbn,nnhnn)&
\abacus(bbhbb,bbhbb,nbhbb,nbhbn,nbhbb,bbhbn,nnhnn)
\end{array}
\]

Let $C$ be the weight $4$ block with the $\langle 5,6^{e-2},2 \rangle$ notation. We define $\la^{y,x,x}$, $\la^{x,x,x}$, $\la^{x,x,y}$, $\la^{x,x,z}$ and $\mu^{x,x,x}$ to be the partitions lying in $C$ by their abacus displays below. We may check using the modular branching rules that
$$\delf^{\la^{y,x}}\uar_{B^{e-2}}^{C}\sim (\delf^{\la^{y,x,x}})^{2},$$
$$\delf^{\la^{x,x}}\uar_{B^{e-2}}^{C}\sim (\delf^{\la^{x,x,x}})^{2}+(\delf^{\la^{x,x,y}})^{2}+(\delf^{\la^{x,x,z}})^{2},$$
$$\LF^{\mu^{x,x}}\uar_{B^{e-2}}^{C}\cong \LF^{\mu^{x,x,x}}\oplus \LF^{\mu^{x,x,x}}.$$
Hence, we have the upper bound
\begin{align*}
2[\delf^{f(\la^{3})}:\LF^{f(\mu^{3})}]&=[\delf^{f(\la^{3})}:\LF^{f(\mu^{3})}][\LF^{f(\mu^{3})}\uar^{B^{0}}\uar_{B^{0}}^{B^{1}}\cdots\uar_{B^{e-3}}^{B^{e-2}}\uar_{B^{e-2}}^{C}:\LF^{\mu^{x,x,x}}]\\
&\le \sum_{\nu}[\delf^{f(\la^{3})}:\LF^{\nu}][\LF^{\nu}\uar^{B^{0}}\uar_{B^{0}}^{B^{1}}\cdots\uar_{B^{e-3}}^{B^{e-2}}\uar_{B^{e-2}}^{C}:\LF^{\mu^{x,x,x}}]\\
& = [\delf^{f(\la^{3})}\uar^{B^{0}}\uar_{B^{0}}^{B^{1}}\cdots\uar_{B^{e-3}}^{B^{e-2}}\uar_{B^{e-2}}^{C}:\LF^{\mu^{x,x,x}}]\\
& = 2([\delf^{\la^{y,x,x}}:\LF^{\mu^{x,x,x}}]+[\delf^{\la^{x,x,x}}:\LF^{\mu^{x,x,x}}]+[\delf^{\la^{x,x,y}}:\LF^{\mu^{x,x,x}}]+[\delf^{\la^{x,x,z}}:\LF^{\mu^{x,x,x}}]).
\end{align*}

\[
\begin{array}{c@{\qquad}c@{\qquad}c@{\qquad}c@{\qquad}c}
\la^{y,x,x}&\la^{x,x,x}&\la^{x,x,y}&\la^{x,x,z}&\mu^{x,x,x}\\
\abacus(bbhbb,bbhbn,nbhbb,bbhbn,bbhbn,bbhbn,nnhnn)&
\abacus(bbhbb,nbhbb,bbhbn,bbhbn,bbhbn,bbhbn,nnhnn)&
\abacus(bbhbn,bbhbb,bbhbn,bbhbb,nbhbn,bbhbn,nnhnn)&
\abacus(bbhbn,bbhbb,bbhbn,bbhbn,bbhbb,nbhbn,nnhnn)&
\abacus(bbhbb,bbhbn,bbhbn,bbhbn,nbhbb,bbhbn,nnhnn)
\end{array}
\]

Using the LLT algorithm and Theorem \ref{Runner Removal Theorem Full}, we have $d^{e}_{\lambda^{y,x,x}\mu^{x,x,x}}(v)=v^{2}$, $d^{e}_{\lambda^{x,x,x}\mu^{x,x,x}}(v)=v^{3}$, $d^{e}_{\lambda^{x,x,y}\mu^{x,x,x}}(v)=0$ and $d^{e}_{\lambda^{x,x,z}\mu^{x,x,x}}(v)=0$. Additionally, observe that $\lambda^{y,x,x}_{i}=\lambda^{x,x,x}_{i}=\lambda^{x,x,y}_{i}=\mu^{x,x,x}_{i}$ for $1\le i \le e-1$, so we may combine Corollary \ref{row removal adj} (applied $e-1$ times) and Theorem \ref{James Conjecture schur 3} to conclude that $[W_{\mathbb{F}}^{\lambda^{y,x,x}}:L_{\mathbb{F}}^{\mu^{x,x,x}}]=[W_{\mathbb{C}}^{\lambda^{y,x,x}}:L_{\mathbb{C}}^{\mu^{x,x,x}}]=1$, $[W_{\mathbb{F}}^{\lambda^{x,x,x}}:L_{\mathbb{F}}^{\mu^{x,x,x}}]=[W_{\mathbb{C}}^{\lambda^{x,x,x}}:L_{\mathbb{C}}^{\mu^{x,x,x}}]=1$ and $[W_{\mathbb{F}}^{\lambda^{x,x,y}}:L_{\mathbb{F}}^{\mu^{x,x,x}}]=[W_{\mathbb{C}}^{\lambda^{x,x,y}}:L_{\mathbb{C}}^{\mu^{x,x,x}}]=0$.

By Theorem \ref{theorem adjustment} and Corollary \ref{jantzen order adj}, we have 
$$[W_{\mathbb{F}}^{\lambda^{x,x,z}}:L_{\mathbb{F}}^{\mu^{x,x,x}}]=[W_{\mathbb{C}}^{\lambda^{x,x,z}}:L_{\mathbb{C}}^{\mu^{x,x,x}}]+\textnormal{adj}_{\lambda^{x,x,z}\mu^{x,x,x}}+\sum\limits_{\lambda^{x,x,z}<_{J} \nu <_{J} \mu^{x,x,x}} [W_{\mathbb{C}}^{\lambda^{x,x,z}}:L_{\mathbb{C}}^{\nu}]\textnormal{adj}_{\nu\mu^{x,x,x}}.$$

From the abacus display of $\mu^{x,x,x}$, we observe that $\row^{e-1}(\mu^{x,x,x})$ has weight 3. Therefore, the terms in the sum above are non-zero only if $\lambda^{x,x,z}<_{J} \nu <_{J} \mu^{x,x,x}$ and $\nu_{e-1}<\mu_{e-1}$ by Corollary \ref{row removal adj}, Corollary \ref{column removal adj} and Theorem \ref{James Conjecture schur 3}. It is easy to check that the set $\{\nu : \lambda^{x,x,z}<_{J} \nu <_{J} \mu^{x,x,x} \textnormal{ and } \nu_{e-1}<\mu_{e-1}\} $ is empty. Moreover, $\lambda^{x,x,z}\sim\langle 0_{3,1} \mid 4,7,\dots, 4e+1\rangle$ induces semisimply to a Rouquier block, so $ \textnormal{adj}_{\lambda^{x,x,z}\mu^{x,x,x}}=0$. Additionally, $[W_{\mathbb{C}}^{\lambda^{x,x,z}}:L_{\mathbb{C}}^{\mu^{x,x,x}}]=d^{e}_{\lambda^{x,x,z}\mu^{x,x,x}}(1)=0$, therefore $[W_{\mathbb{F}}^{\lambda^{x,x,z}}:L_{\mathbb{F}}^{\mu^{x,x,x}}]=0$ and $[W_{\mathbb{F}}^{f(\lambda^{3})}:L_{\mathbb{F}}^{f(\mu^{3})}] = 2$.
\end{proof}

To prove Proposition \ref{empty core weight 4}, all that remains is to show that $\textnormal{adj}_{\lambda^{0}\mu^{3}}=0$. Since $d^{e}_{\lambda^{0}\mu^{3}}(v)=0$ (calculated using the LLT algorithm and runner removal theorems \ref{Runner Removal Theorem Empty} and \ref{Runner Removal Theorem Full}), we just need to check that $\textnormal{adj}_{\nu\mu^{3}}=0$ for every partition $\nu$ satisfying $\lambda^{0}<_{P}\nu<_{P}\mu^{3}$ in order to apply Proposition \ref{decomposition number 0 or v}.

We list all the partitions $\nu$ such that $\lambda^{0}<_{P} \nu <_{P} \mu^{3}$:
\begin{itemize}
\item $\nu=\langle       0_{2}, 1_{1^{2}}   \mid 4^{2}, 5^{e-2}  \rangle\sim \langle 0_{3},1 \mid 4,7,\dots,3e+1 \rangle $, so $\textnormal{adj}_{\nu\mu^{3}}=0$ by Proposition \ref{induce semisimply}.
\item $\nu=\langle       0_{2,1}, 1   \mid 4^{2}, 5^{e-2}  \rangle\sim \langle 0_{3,1} \mid 4,7,\dots,3e+1 \rangle$, so $\textnormal{adj}_{\nu\mu^{3}}=0$ by Proposition \ref{induce semisimply}.
\item $\nu=\lambda^{3}$, so $\textnormal{adj}_{\lambda^{3}\mu^{3}}=0$ by Proposition \ref{prop 4e difficult}.
\end{itemize}

This concludes the proof of Proposition \ref{empty core weight 4}.
\end{proof}

\subsection{Blocks forming exactly two $[4:1]$-pairs}
In the next two sections, we prove the following proposition.
\begin{prop}\label{2 [4:1]-pairs}
Let $B$ be a weight $4$ block of $\s_{n}$. Suppose that there are exactly two blocks $A_{1}$ and $A_{2}$ forming a $[4:1]$-pair with $B$, and that there are no other blocks $C$ forming a $[4:k]$-pair with $B$ for any $k$. Additionally, suppose that the adjustment matrix for every weight $4$ block of $\s_{m}$ is the identity matrix whenever $m<n$. Then, the adjustment matrix for $B$ is the identity matrix.
\end{prop}
The conditions above give two distinct types of block $B$. 
\subsubsection{Blocks with the $\langle4^{a},5^{b-a},4^{c-b},5^{d-c},4^{e-d}\rangle$ notation}
We first consider the case where $B$ has the $\langle4^{a},5^{b-a},4^{c-b},5^{d-c},4^{e-d}\rangle$ notation, for some $0<a<b<c<d\le e$. Let $\lambda$ and $\mu$ be two distinct partitions in $B$ for which we want to prove that $\textnormal{adj}_{\lambda\mu}=0$. By Proposition \ref{proposition exceptional lambda}, if $\lambda$ were non-exceptional for any one of $(A_{i},B)$, then $\textnormal{adj}_{\lambda\mu}=\delta_{\lambda\mu}$. Therefore, we may assume that $\lambda$ is exceptional for both pairs $(A_{1},B)$ and $(A_{2},B)$; that is $\lambda=\langle a_{1^{2}},c_{1^{2}}\rangle$.
\[
\begin{array}{c}
\la \\
\abacus(bhbnbhbbhbnbhbbhb,nhnbbhbnhnbbhbnhn,nhnbnhnnhnbnhnnhn,nhnnnhnnhnnnhnnhn)
\end{array}
\]
Observe that the first unoccupied position in the abacus display for $\la$ above occurs at position $a$. By Remark \ref{remark row removal}, we may assume that $l(\mu)<l(\la)$, so the first unoccupied position in an abacus display for $\mu$ must occur after position $a$. Consequently, all beads in runner $0$ of an abacus display for $\mu$ have zero weight, therefore $\mu$ must be $e$-regular and $\adj_{\la\mu}=\delta_{\la\mu}$ by Proposition \ref{e-regular mu} and Theorem \ref{James Conjecture hecke 4}. 

\subsubsection{Blocks with the $\langle4^{a},5^{b-a},6^{c-b},5^{d-c},4^{e-d}\rangle$ notation}
We complete the proof of Proposition \ref{2 [4:1]-pairs} by considering blocks $B$ with the $\langle4^{a},5^{b-a},6^{c-b},5^{d-c},4^{e-d}\rangle$ notation, for some $0<a<b<c\le d\le e$. Let $\lambda$ and $\mu$ be two distinct partitions in $B$ for which we want to prove that $\textnormal{adj}_{\lambda\mu}=0$. By Proposition \ref{proposition exceptional lambda}, we may assume that $\lambda$ is exceptional for both pairs $(A_{1},B)$ and $(A_{2},B)$; that is $\lambda=\langle a_{1^{2}},b_{1^{2}}\rangle$ with $b-a \ge 2$.
\[
\begin{array}{c}
\la \\
\abacus(bhbnbhbbbhbbhbbhb,nhnbbhbnbhbbhbnhn,nhnbnhnbbhbnhnnhn,nhnnnhnbnhnnhnnhn,nhnnnhnnnhnnhnnhn)
\end{array}
\]
Observe that the first unoccupied position in the abacus display for $\la$ above occurs at position $a$. By Remark \ref{remark row removal}, we may assume that $l(\mu)<l(\la)$, so the first unoccupied position in an abacus display for $\mu$ must occur after position $a$. Consequently, all beads in runner $0$ of an abacus display for $\mu$ have zero weight, therefore $\mu$ must be $e$-regular and $\adj_{\la\mu}=\delta_{\la\mu}$ by Proposition \ref{e-regular mu} and Theorem \ref{James Conjecture hecke 4}. 

This completes the proof of Proposition \ref{2 [4:1]-pairs}.

\subsection{Blocks forming exactly one $[4:3]$-pair}
\begin{prop}\label{[4:3]-pairs}
Suppose that $A$ and $B$ are weight $4$ blocks of $\s_{n-1}$ and $\s_{n}$ respectively, forming a $[4:3]$-pair. Moreover, suppose that there is no block other than $A$ forming a $[4:k]$-pair with $B$ for any $k$. Additionally, suppose that the adjustment matrix for every weight $4$ block of $\s_{m}$ is the identity matrix whenever $m<n$. Then, the adjustment matrix for $B$ is the identity matrix.

% If James's Conjecture holds for all weight $4$ blocks of $\s_{m}$ with weight at most 4 whenever $m<n$, then James's Conjecture holds for $B$.

\end{prop}
\begin{proof}
The conditions on $B$ mean that we may represent $B$ on an abacus with the $\langle4^{a},7^{b-a},6^{c-b},5^{d-c},4^{e-d}\rangle$ notation, where $0<a<b\le c\le d\le e$. If $\lambda$ and $\mu$ are two distinct partitions in $B$, then we have $\textnormal{adj}_{\lambda\mu}=0$ by Proposition \ref{proposition exceptional lambda} unless $\lambda$ is the unique exceptional partition for $(A,B)$, namely $\lambda=\langle a_{1^{4}}\rangle$. 
\[
\begin{array}{c}
\la \\
\abacus(bhbnbhbbhbbhbbhb,nhnbbhbbhbbhbnhn,nhnbbhbbhbnhnnhn,nhnbbhbnhnnhnnhn,nhnbnhnnhnnhnnhn,nhnnnhnnhnnhnnhn)
\end{array}
\]

Observe that the first unoccupied position in the abacus display for $\la$ above occurs at position $a$. By Remark \ref{remark row removal}, we may assume that $l(\mu)<l(\la)$, so the first unoccupied position in an abacus display for $\mu$ must occur after position $a$. Consequently, all beads in runner $0$ of an abacus display for $\mu$ have zero weight, therefore $\mu$ must be $e$-regular and $\adj_{\la\mu}=\delta_{\la\mu}$ by Proposition \ref{e-regular mu} and Theorem \ref{James Conjecture hecke 4}. 
\end{proof}

\subsection{Blocks forming exactly one $[4:2]$-pair}
\begin{prop}\label{[4:2]-pairs}
Suppose that $A$ and $B$ are weight $4$ blocks of $\s_{n-1}$ and $\s_{n}$ respectively, forming a $[4:2]$-pair. Moreover, suppose that there is no block other than $A$ forming a $[4:k]$-pair with $B$ for any $k$. Additionally, suppose that the adjustment matrix for every weight $4$ block of $\s_{m}$ is the identity matrix whenever $m<n$. Then, the adjustment matrix for $B$ is the identity matrix.
\end{prop}
The conditions on $B$ mean that we may represent $B$ on an abacus with the $\langle 4^{a},6^{b-a},5^{c-b},4^{e-c}\rangle$ notation, where $0<a<b\le c\le e$. Suppose that $\lambda$ and $\mu$ are two distinct partitions lying in $B$ and we want to prove that $\textnormal{adj}_{\lambda\mu}=0$. If $\lambda$ were non-exceptional for $(A,B)$, then $\textnormal{adj}_{\lambda\mu}=\delta_{\lambda\mu}$ by Proposition \ref{proposition exceptional lambda}. Therefore we may assume that $\lambda$ is exceptional; that is $\la$ must be one of the following partitions:
\begin{itemize}
\item $\langle a_{2,1^{2}}\rangle$,
\item $\langle a_{1^{3}},i\rangle, i\notin\{a-1,a\}$,
\item $\langle a_{1^{4}} \rangle$.
\end{itemize}
We invite the reader to verify that $\langle a_{1^{4}}\rangle$ induces semi-simply to a Rouquier block:
\begin{equation*}
\langle a_{1^{4}}\rangle \sim
\begin{cases}
\langle0,a,a+e-c,a+e-b\mid4,7,\dots,3e+1\rangle  & \text{if     }e-c>0,c-b>0,\\
\langle0,a_{1^{2}},a+e-c\mid4,7,\dots,3e+1\rangle & \text{if     }e-c=0,c-b>0,\\
\langle0,a,(a+e-c)_{1^{2}}\mid4,7,\dots,3e+1\rangle & \text{if     }e-c>0,c-b=0,\\
\langle0,a_{1^{3}}\mid4,7,\dots,3e+1\rangle & \text{if     }e-c=0,c-b=0.
\end{cases}
\end{equation*}
Therefore, we may assume that $\la\neq \langle a_{1^{4}}\rangle$ by Proposition \ref{induce semisimply}. We are left to consider the cases $\la=\langle a_{2,1^{2}}\rangle$ or $\la=\langle a_{1^{3}},i\rangle, i\notin\{a-1,a\}$. 

\[
\begin{array}{c@{\qquad}c}
\langle 0, a_{1^{3}}\rangle &\langle a_{2,1^{2}}\rangle\\
\abacus(nbhbnbhbbhbbhb,bnhnbbhbbhbnhn,nnhnbbhbnhnnhn,nnhnbnhnnhnnhn,nnhnnnhnnhnnhn,nnhnnnhnnhnnhn)&
\abacus(bhbnbhbbhbbhb,nhnbbhbbhbnhn,nhnbbhbnhnnhn,nhnnnhnnhnnhn,nhnbnhnnhnnhn,nhnnnhnnhnnhn)
\end{array}
\]

Let $\nu:=\langle0,a_{1^{3}}\rangle$. Note that $l(\nu)\ge l(\la)$ for all the remaining possible $\la$, so we may assume that $l(\mu)<l(\nu)$ by Remark \ref{remark row removal}. Since the first unoccupied position in the abacus display for $\la$ occurs at position $0$, the first unoccupied position in an abacus display for $\mu$ must occur strictly after position $0$. Consequently, all beads in runner $0$ of an abacus display for $\mu$ have zero weight, therefore $\mu$ must be $e$-regular and $\adj_{\la\mu}=\delta_{\la\mu}$ by Proposition \ref{e-regular mu} and Theorem \ref{James Conjecture hecke 4}. This completes the proof of Proposition \ref{[4:2]-pairs}.

\subsection{Blocks forming exactly one $[4:1]$-pair}\label{section 41}
In this section, we prove the following proposition:
\begin{prop}\label{[4:1]-pairs}
Suppose that $A$ and $B$ are weight $4$ blocks of $\s_{n-1}$ and $\s_{n}$ respectively, forming a $[4:1]$-pair. Moreover, suppose that there is no block other than $A$ forming a $[4:k]$-pair with $B$ for any $k$. Additionally, suppose that the adjustment matrix for every weight $4$ block of $\s_{m}$ is the identity matrix whenever $m<n$. Then, the adjustment matrix for $B$ is the identity matrix.
\end{prop}
The conditions on $B$ mean that we may represent $B$ on an abacus with the $\langle 4^{a},5^{b-a},4^{e-b}\rangle$ notation, where $0<a<b\le e$. Suppose that $\lambda$ and $\mu$ are two distinct partitions lying in $B$ and we want to prove that $\textnormal{adj}_{\lambda\mu}=0$. By Proposition \ref{proposition exceptional lambda} and Proposition \ref{proposition non-exceptional mu} we may assume that $\lambda$ is exceptional, so that $\lambda$ is one of the following partitions:
\begin{itemize}
\item $\langle a_{3,1}\rangle$,
\item $\langle a_{2^{2}}\rangle$,
\item $\langle a_{2,1^{2}}\rangle$,
\item $\langle a_{2,1},i\rangle$, $i\notin\{a-1,a\}$, $e\ge3$,
\item $\langle a_{1^{2}},i_{2}\rangle$, $i\notin\{a-1,a\}$, $e\ge3$,
\item $\langle a_{1^{2}},i,j\rangle$, $i,j\notin\{a-1,a\}$ and $i<j$, $e\ge4$,
\item $\langle a_{1^{3}},i\rangle$, $i\neq a$,
\item $\langle a_{1^{2}},i_{1^{2}}\rangle$, $i\notin\{a-1,a\}$, $e\ge3$,
\item $\langle a_{1^{4}} \rangle$.
\end{itemize}
We invite the reader to check that $\langle a_{1^{4}} \rangle$ induces semi-simply to a Rouquier block:
\begin{equation*}
\langle a_{1^{4}} \rangle \sim
\begin{cases}
\langle 0_{1^{2}},a,a+e-b\mid4,7,\dots,3e+1\rangle & \text{if          } e-b>0,\\
\langle 0_{1^{2}},a_{1^{2}}\mid4,7,\dots,3e+1\rangle & \text{if          } e-b=0.
\end{cases}
\end{equation*}
Therefore, we may assume that $\lambda\neq\langle a_{1^{4}} \rangle$ by Proposition \ref{induce semisimply}. In the remaining possibilities for $\la$, observe that $\lambda$ has exactly two normal beads on runner $a$ and at most one normal bead on each of the other runners. By Remark \ref{remark lowerable}, we may assume that $\mu$ has at most one normal bead on runner $a$ (non-exceptional) and no normal beads on every other runner. Note that $\nu:=\langle 0_{1^{2}},a_{1^{2}}\rangle \le_{P} \la$ for all the remaining possibilities for $\lambda$. By Remark \ref{remark row removal} we may further assume that $\mu\gg\nu$ since $\mu\gg\lambda$ and $\lambda \ge_{P}\nu$ imply that $\mu\gg\nu$. By Corollary \ref{e-regular mu}, we may also assume that $\mu$ is $e$-singular. We list every possibility for $\mu$ satisfying all the conditions stated above:
\begin{itemize}
\item[(A1)] $\langle a_{2},b_{1^{2}} \rangle$, $e-b>0$,
\item[(A2)] $\langle a,b_{2,1}\rangle$, $b-a=1,e-b>0$,
\item[(A3)] $\langle a,a+1,b_{1^{2}}\rangle$, $b-a\ge2,e-b>0$,
\item[(B1)] $\langle 0_{3},a\rangle$, $a=1,b-a=1,e-b=0$,
\item[(B2)] $\langle 0_{2},a,a+1 \rangle$, $a=1,e-b=0$, $b-a\ge 2$,
\item[(C1)] $\langle 0,a_{3}\rangle$,
\item[(C2)] $\langle 0,a_{2},a+1\rangle$, $b-a\ge2$,
\item[(C3)] $\langle 0,a_{2,1}\rangle$, $a=1$,
\item[(C4)] $\langle 0,1,a_{2}\rangle$, $a\ge2$,
\item[(C5)] $\langle 0,a,a+1,a+2\rangle$, $b-a\ge3$,
\item[(C6)] $\langle 0,a_{1^{2}},a+1\rangle$, $a=1,b-a\ge2$,
\item[(C7)] $\langle 0,1,a,a+1\rangle$, $a\ge2,b-a\ge2$.
\end{itemize}
We now consider 2 separate cases.
\subsubsection{The case $e-b>0$}
For any partition $\nu$ lying in $B$, we define $\mathfrak{a}(\nu):=\dot{F}_{a-b}{\nearrow}^{a}\nu$ (in an abacus display with the $\langle 4^{a},5^{b-a},4^{e-b} \rangle$ notation, the beads on runner $0$ have $e$-residue $a-b$). We may check that $\mathfrak{a}(\lambda)$ and $\mathfrak{a}(\mu)$ are both semisimply induced and that $(\mathfrak{a}(\la),\mathfrak{a}(\mu))$ is lowerable except when one of the following happens:
\begin{itemize}
\item $\lambda=\langle a_{2^{2}}\rangle$ or $\lambda=\langle i_{2},a_{1^{2}}\rangle$ for some $0\le i\le a-2$,
\item $e-b=1$ and $\mu$ is in case A2 or A3.
\end{itemize}
When $\lambda\in\{\langle a_{2^{2}}\rangle,\langle i_{2},a_{1^{2}}\rangle\}$, none of our ten cases of $\mu$ (cases A1-A3, C1-C7) satisfy $\mu\gg\lambda$, therefore $\textnormal{adj}_{\lambda\mu}=\delta_{\lambda\mu}$ by Remark \ref{remark row removal}. In view of Corollary \ref{corollary lowerable} and Corollary \ref{working exceptional lambda}, we are left to consider $\mu$ in cases A2 and A3 where $e-b=1$. By considering $\mu\gg\lambda$ again, we are left with the following possibilities for $\lambda$:
\begin{itemize}
\item $\langle a_{1^{3}},i\rangle$, $0\le i\le a-1$ or $i=b$, 
\item $\langle i_{1^{2}},a_{1^{2}}\rangle$, $0\le i\le a-2$.
\end{itemize}
For any partition $\nu$ lying in $B$, we define $\mathfrak{b}(\nu):=\dot{F}_{a}{\searrow}_{b-a}\nu$ (in an abacus display with the $\langle 4^{a},5^{b-a},4^{e-b} \rangle$ notation, the beads on runner $b$ have $e$-residue $a$). We observe that for the remaining pairs $(\lambda,\mu)$, $\mathfrak{b}(\lambda)$ and $\mathfrak{b}(\mu)$ are both semisimply induced and that $(\mathfrak{b}(\la),\mathfrak{b}(\mu))$ is lowerable, so the proof of Proposition \ref{[4:1]-pairs} when $e-b>0$ is complete by Corollary \ref{corollary lowerable} and Corollary \ref{working exceptional lambda}.

\subsubsection{The case $e-b=0$}\label{sec e=b}
The case $e=b$ is much more difficult to deal with. From now on, we assume that $e=b$. We will first refine the list of possible $\mu$. We note that $\langle 0,a_{1^{3}}\rangle\sim \langle 0_{2},a_{1^{2}}\mid4,7,\dots,3e+1\rangle$, so we may assume that $\lambda\neq\langle 0,a_{1^{3}}\rangle$ by Proposition \ref{induce semisimply}. In view of this, we have $\tau:=\langle a_{1^{3}},a+1 \rangle \le_{P} \la$ for the remaining possibilities for $\la$ when $a=1$ and $e-a\ge2$. When $a=1$ and $e-a=1$, we have $\gamma:=\langle a_{2,1^{2}} \rangle\le_{P} \la$ for the remaining possibilities for $\la$.

By Remark \ref{remark row removal}, we may assume that $\mu\gg\tau$ when $a=1$ and $e-a\ge2$, since $\mu\gg\la$ and $\la \ge_{P}\tau$ imply that $\mu\gg\tau$; when $a=1$ and $e-a=1$, we may assume that $\mu\gg\gamma$. These two additional restrictions produce the following refined list of possibilities for $\mu$:
\begin{itemize}
\item[(C1)] $\langle 0,a_{3}\rangle$,
\item[(C2)] $\langle 0,a_{2},a+1\rangle$, $e-a\ge2$,
\item[(C3)] $\langle 0,a_{2,1}\rangle$, $a=1$, $e-a\ge2$,
\item[(C4)] $\langle 0,1,a_{2}\rangle$, $a\ge2$,
\item[(C5)] $\langle 0,a,a+1,a+2\rangle$, $e-a\ge3$,
\item[(C7)] $\langle 0,1,a,a+1\rangle$, $a\ge2,e-a\ge2$.
\end{itemize}

For any partition $\sigma$ in $B$, define $\mathfrak{a}(\sigma):=\dot{F}_{a}{\nearrow}^{a}\sigma$ (in an abacus display with the $\langle 4^{a},5^{e-a} \rangle$ notation, the beads on runner $0$ have $e$-residue $a$). When $\mu$ is in cases C1, C3 and C4 with $e-a\ge2$, $\mathfrak{a}(\lambda)$ and $\mathfrak{a}(\mu)$ are both semisimply induced and $(\mathfrak{a}(\la),\mathfrak{a}(\mu))$ is lowerable. When $e-a\ge3$ and $\mu$ is in case C2, $\mathfrak{a}(\lambda)$ and $\mathfrak{a}(\mu)$ are both semisimply induced and $(\mathfrak{a}(\la),\mathfrak{a}(\mu))$ is lowerable. In view of Corollary \ref{corollary lowerable} and Corollary \ref{working exceptional lambda}, we are left with the following possibilities for $\mu$:
\begin{itemize}
\item[(C1)] $\langle 0,a_{3}\rangle$, $e-a=1$,
\item[(C2)] $\langle 0,a_{2},a+1\rangle$, $e-a=2$,
\item[(C4)] $\langle 0,1,a_{2}\rangle$, $a\ge2$, $e-a=1$,
\item[(C5)] $\langle 0,a,a+1,a+2\rangle$, $e-a\ge3$,
\item[(C7)] $\langle 0,1,a,a+1\rangle$, $a\ge2,e-a\ge2$.
\end{itemize}

\paragraph{\textbf{Cases C1, C2 and C5}}\label{subsub 1}
Suppose that $\mu$ is in either case C1, C2 or C5. When $e-a=1$, $\langle a_{2,1^{2}}\rangle \sim \langle0,a_{2,1} \mid 4,7,\dots,3e+1\rangle$, so we may assume that $\lambda\neq\langle a_{2,1^{2}}\rangle$ by Proposition \ref{induce semisimply}. By Remark \ref{remark row removal}, we may also assume that $\lambda \ll \mu$. We find that the only remaining possibilities for $\la$ are:
\begin{itemize}
\item $\langle i_{1^{2}},a_{1^{2}}\rangle$ for some $0\le i\le a-2$,
\item $\langle i, a_{1^{3}} \rangle$ for some $1 \le i \le a-1$.
\end{itemize}

Our strategy to deal with the remaining pairs $(\lambda,\mu)$ is to first induce them up semisimply via the same sequence of inductions (for example using\\
$f:=\dot{F}_{5a+2}{\searrow}_{e-a-1}\dots\dot{F}_{2a+1}{\searrow}_{e-a-1}\dot{F}_{2a}{\searrow}_{e-a-1}\dot{F}_{2a-1}{\searrow}_{e-a}\dots\dot{F}_{a+1}{\searrow}_{e-a}\dot{F}_{a}{\searrow}_{e-a}$ which is well-defined for partitions in block $B$) to the block with the $\langle 4^{a},6,9^{e-a-1}\rangle$ notation, followed by an application of the Jantzen-Schaper formula. We denote their induced counterparts as $(\hat{\lambda},\hat{\mu})$; that is $(\lambda,\mu)\sim(\hat{\lambda},\hat{\mu})$. Using Theorem \ref{jantzen schaper}, we would show that $[W^{\hat{\lambda}}:L^{\hat{\mu}}]$ is independent of $\textnormal{char}(\F)$, so that $\textnormal{adj}_{\hat{\lambda}\hat{\mu}}=\delta_{\hat{\lambda}\hat{\mu}}$ and hence $\textnormal{adj}_{\lambda\mu}=\delta_{\lambda\mu}$ by Corollary \ref{working exceptional lambda} and Lemma \ref{lemma easy fact}. We may check that
\begin{itemize}
\item $\langle i_{1^{2}},a_{1^{2}}\rangle\sim \langle i_{2,1},a \mid 4^{a},6,9^{e-a-1}\rangle$ when $i \le a-2$,
\item $\langle i, a_{1^{3}} \rangle \sim \langle 0,i_{2},a \mid 4^{a},6,9^{e-a-1}\rangle$ when $1\le i \le a-1$,
\item $\mu \sim \langle 0,a_{3}\mid4^{a},6,9^{e-a-1}\rangle$.
\end{itemize}

Note that $\langle 0_{2,1},a \mid 4^{a},6,9^{e-a-1}\rangle\le_{J}\hat{\la}$ for all remaining possible $\la$. Table \ref{4:1 1}, Table \ref{4:1 2} and Table \ref{4:1 3} (in the appendix) illustrates how we use Theorem \ref{jantzen schaper}. The entries of the tables are the Jantzen-Schaper coefficients (Definition \ref{defn js}) $J(\nu,\sigma)$, for partitions $\langle 0_{2,1},a \mid 4^{a},6,9^{e-a-1}\rangle\le_{J}\nu<_{J}\sigma\le_{J}\hat{\mu}=\langle 0,a_{3} \mid 4^{a},6,9^{e-a-1}\rangle$. Note that $J_{\mathbb{F}}(\nu,\sigma)=J_{\mathbb{C}}(\nu,\sigma)$ since char$(\mathbb{F})>w=4$. We also omit the columns indexed by $\nu$ if $[W^{\nu}:L^{\hat{\mu}}]=0$ as these do not contribute to our calculations. If we are able to justify the last column of Table \ref{4:1 1}, Table \ref{4:1 2} and Table \ref{4:1 3}, it would imply that $[W^{\nu}:L^{\hat{\mu}}]$ is independent of $\F$, hence $\adj_{\nu\hat{\mu}}=\delta_{\nu\hat{\mu}}$ for all $\langle 0_{2,1},a \mid 4^{a},6,9^{e-a-1}\rangle\le_{J}\nu\le_{J}\hat{\mu}$; this shows that $\textnormal{adj}_{\lambda\mu}=\delta_{\lambda\mu}$ for every remaining possible $\la$ by Corollary \ref{working exceptional lambda} and Lemma \ref{lemma easy fact}.

We now proceed to justify the last column of Table \ref{4:1 1}, Table \ref{4:1 2} and Table \ref{4:1 3}. When $B(\nu,\hat{\mu})\le1$, we have $[W^{\nu}:L^{\hat{\mu}}]=B(\nu,\hat{\mu})$ by Corollary \ref{cor jantzen}, so $\adj_{\nu\hat{\mu}}=\delta_{\nu\hat{\mu}}$ by Lemma \ref{lemma easy fact}. When $B(\nu,\hat{\mu})>1$, we have to do more work.

Let $\nu^{0}:=\langle (a-1)_{2,1},a\mid4^{a},6,9^{e-a-1}\rangle$. Observe that $B(\nu^{0},\hat{\mu})=2$ and that this is the only partition $\nu$ in Table \ref{4:1 1}, Table \ref{4:1 2} and Table \ref{4:1 3} with $B(\nu,\hat{\mu})>1$ (shaded the tables). Following Remark \ref{remark after}, we calculate $\sigma_{e}(\nu^{0})$ and $\sigma_{e}(\hat{\mu})$, and find that they are both $(-1)^{a}$, hence $[W^{\nu^{0}}_{\mathbb{C}}:L^{\hat{\mu}}_{\mathbb{C}}]=1$. A priori, we only know that $[W^{\nu^{0}}_{\mathbb{F}}:L^{\hat{\mu}}_{\mathbb{F}}]\le2$. 
\begin{lem}
We have $[W^{\nu^{0}}_{\mathbb{F}}:L^{\hat{\mu}}_{\mathbb{F}}]=1$.
\end{lem}
\begin{proof}

We may check that $\nu^{0}$ induces semi-simply to a Rouquier block:
\begin{equation*}
\nu^{0} \sim
\begin{cases}
\langle 0_{2,1},a\mid4,7,\dots,3e+1\rangle & \text{if          } a=1,\\
\langle 0,(a-1)_{2},a\mid4,7,\dots,3e+1\rangle & \text{if          } a\ge2.
\end{cases}
\end{equation*}
Hence, $\adj_{\nu^{0}\hat{\mu}}=0$ by Proposition \ref{induce semisimply}. Moreover, we deduce from the rows above $\nu^{0}$ in Table \ref{4:1 1}, Table \ref{4:1 2} and Table \ref{4:1 3} that $[W^{\nu}:L^{\hat{\mu}}]$ is independent of $\F$ whenever $\nu^{0}<_{J}\nu<_{J}\hat{\mu}$, therefore $\adj_{\nu\hat{\mu}}=0$ by Lemma \ref{lemma easy fact}.
Finally, $$[W^{\nu^{0}}_{\mathbb{F}}:L^{\hat{\mu}}_{\mathbb{F}}]=\sum\limits_{\nu^{0}\le_{J}\nu\le_{J}\hat{\mu}}[W^{\nu^{0}}_{\C}:L^{\nu}_{\C}]\adj_{\nu\hat{\mu}}=[W^{\nu^{0}}_{\mathbb{C}}:L^{\hat{\mu}}_{\mathbb{C}}]=1.$$
\end{proof}

\paragraph{\textbf{Cases C4 and C7}}\label{subsub 2}
Suppose that $\mu$ is in either case C4 or C7. When $e-a=1$, $\langle a_{2,1^{2}}\rangle \sim \langle0,a_{2,1} \mid 4,7,\dots,3e+1\rangle$, so we may assume that $\lambda\neq\langle a_{2,1^{2}}\rangle$ by Proposition \ref{induce semisimply}. By Remark \ref{remark row removal}, we may also assume that $\lambda \ll \mu$. We find that the only remaining possibilities for $\la$ are:
\begin{itemize}
\item $\langle 1, a_{1^{3}} \rangle$,
\item $\langle 1_{1^{2}},a_{1^{2}}\rangle$,
\item $\langle 0_{1^{2}},a_{1^{2}}\rangle$.
\end{itemize}

We use the Jantzen Schaper formula and Corollary \ref{working exceptional lambda} in the same fashion as in section \ref{subsub 1} to deal with the remaining pairs of $(\la,\mu)$. We may check that (for example using $f:=\dot{F}_{5a+2}{\searrow}_{e-a-1}\dots\dot{F}_{2a+1}{\searrow}_{e-a-1}\dot{F}_{2a}{\searrow}_{e-a-1}\dot{F}_{2a-1}{\searrow}_{e-a}\dots\dot{F}_{a+1}{\searrow}_{e-a}\dot{F}_{a}{\searrow}_{e-a}$)
\begin{itemize}
\item $\langle 1, a_{1^{3}} \rangle \sim \langle 0,1_{2},a \mid 4^{a},6,9^{e-a-1}\rangle$,
\item $\langle 1_{1^{2}},a_{1^{2}}\rangle\sim \langle 1_{2,1},a \mid 4^{a},6,9^{e-a-1}\rangle$,
\item $\langle 0_{1^{2}},a_{1^{2}}\rangle\sim \langle 0_{2,1},a \mid 4^{a},6,9^{e-a-1}\rangle$,
\item $\mu \sim \langle 0,1,a_{2}\mid4^{a},6,9^{e-a-1}\rangle$.
\end{itemize}

Note that $\langle 0_{2,1},a \mid 4^{a},6,9^{e-a-1}\rangle<_{J}\langle 1_{2,1},a \mid 4^{a},6,9^{e-a-1}\rangle<_{J}\langle 0,1_{2},a \mid 4^{a},6,9^{e-a-1}\rangle$. The entries of Table \ref{4:1 6} (in the appendix) are the Jantzen-Schaper coefficients (Definition \ref{defn js}) $J(\nu,\sigma)$, for partitions $\langle 0_{2,1},a \mid 4^{a},6,9^{e-a-1}\rangle\le_{J}\nu<_{J}\sigma\le_{J}\hat{\mu}:=\langle 0,1,a_{2} \mid 4^{a},6,9^{e-a-1}\rangle$. Note that $J_{\mathbb{F}}(\nu,\sigma)=J_{\mathbb{C}}(\nu,\sigma)$ since char$(\mathbb{F})>w=4$. We also omit the columns indexed by $\nu$ if $[W^{\nu}:L^{\hat{\mu}}]=0$ as these do not contribute to our calculations. If we are able to justify the last column of Table \ref{4:1 6}, it would imply that $[W^{\nu}:L^{\hat{\mu}}]$ is independent of $\F$, hence $\adj_{\nu\hat{\mu}}=\delta_{\nu\hat{\mu}}$ for all $\langle 0_{2,1},a \mid 4^{a},6,9^{e-a-1}\rangle\le_{J}\nu\le_{J}\hat{\mu}$ by Lemma \ref{lemma easy fact}; this shows that $\textnormal{adj}_{\lambda\mu}=\delta_{\lambda\mu}$ for every remaining possible $\la$ by Corollary \ref{working exceptional lambda}.

We now proceed to justify the last column of Table \ref{4:1 6}. When $B(\nu,\hat{\mu})\le1$, we have $[W^{\nu}:L^{\hat{\mu}}]=B(\nu,\hat{\mu})$ by Corollary \ref{cor jantzen}, so $\adj_{\nu\hat{\mu}}=\delta_{\nu\hat{\mu}}$ by Lemma \ref{lemma easy fact}. When $B(\nu,\hat{\mu})>1$, we have to do more work.

Let $\nu^{0}:=\langle 1_{2,1},a\mid4^{a},6,9^{e-a-1}\rangle$. Observe that $B(\nu^{0},\hat{\mu})=2$ (shaded in Table \ref{4:1 6}) and that this is the only partition $\nu$ in Table \ref{4:1 6} with $B(\nu,\hat{\mu})>1$. Following Remark \ref{remark after}, we calculate $\sigma_{e}(\nu^{0})$ and $\sigma_{e}(\hat{\mu})$, and find that they are both $+1$, hence $[W^{\nu^{0}}_{\mathbb{C}}:L^{\hat{\mu}}_{\mathbb{C}}]=1$. A priori, we only know that $[W^{\nu^{0}}_{\mathbb{F}}:L^{\hat{\mu}}_{\mathbb{F}}]\le2$. 
\begin{lem}
We have $[W^{\nu^{0}}_{\mathbb{F}}:L^{\hat{\mu}}_{\mathbb{F}}]=1$.
\end{lem}
\begin{proof}
We may check that $\nu^{0}\sim \tilde{\nu}:=\langle0,1,2_{2}\mid 7,3,4^{a-1},10^{e-a-1} \rangle$ and $\hat{\mu}\sim \tilde{\mu}:=\langle 0_{2},1_{2} \mid 7,3,4^{a-1},10^{e-a-1} \rangle$ (see Figure \ref{fig long} in the appendix). Observe that $l(\tilde{\nu})=l(\tilde{\mu})$ and that $\col(\tilde{\nu})$ and $\col(\tilde{\mu})$ have weight 3. Hence $\adj_{\nu^{0}\hat{\mu}}=0$ by Theorem \ref{James Conjecture schur 3} and Corollary \ref{column removal adj}. Moreover,we deduce from the rows above $\nu^{0}$ in Table \ref{4:1 6} that $[W^{\nu}:L^{\hat{\mu}}]$ is independent of $\F$ whenever $\nu^{0}<_{J}\nu<_{J}\hat{\mu}$, therefore $\adj_{\nu\hat{\mu}}=0$ by Lemma \ref{lemma easy fact}.
Finally, $$[W^{\nu^{0}}_{\mathbb{F}}:L^{\hat{\mu}}_{\mathbb{F}}]=\sum\limits_{\nu^{0}\le_{J}\nu\le_{J}\hat{\mu}}[W^{\nu^{0}}_{\C}:L^{\nu}_{\C}]\adj_{\nu\hat{\mu}}=[W^{\nu^{0}}_{\mathbb{C}}:L^{\hat{\mu}}_{\mathbb{C}}]=1.$$
\end{proof}

\begin{comment}
\[
\begin{array}{ccccc}
\nu^{0}&&&&\\
\abacus(bbbhbbbhb,bbbhbbbhb,bnbhbbbhb,bbbhbbbhb,nnnhnbbhb,nbnhnnbhb,nnnhnbbhb,nnnhnnbhb,nnnhnnbhb,nnnhnnnhn,nnnhnnnhn)&
\sim&
\abacus(Ebbbhbbhbb,Ebbbhbbhbb,Ebnbhbbhbb,Ebbbhbbhbb,Cnnnhnbhbb,Cnbnhnbhbn,Ennnhnbhbb,Cnnnhnbhbn,Ennnhnbhbn,Ennnhnbhbn,Ennnhnnhnn)&
\sim&
\abacus(bbbhbbhbb,bbbhbbhbb,bnbhbbhbb,bbbhbbhbn,bnnhnbhbn,bbnhnbhbn,nnnhnbhbn,bnnhnbhbn,nnnhnbhbn,nnnhnbhbn,nnnhnnhnn)
\end{array}
\]
\[
\begin{array}{cccc}
&&&\tilde{\nu}\\
\sim&
\abacus(bbbEbhbbhb,bbbEbhbbhb,bnb-bhbbhb,bbn+bhbbhb,bnnEnhnbhb,bbnCnhnbhb,nnnEnhnbhb,bnnEnhnbhb,nnnEnhnbhb,nnnEnhnbhb,nnnEnhnnhn)&
\sim&
\abacus(bbbbhbbhb,bbbbhbbhb,bnbbhbbhb,bbnbhbbhb,bnnnhnbhb,bnbnhnbhb,nnnnhnbhb,bnnnhnbhb,nnnnhnbhb,nnnnhnbhb,nnnnhnnhn)
\end{array}
\]

\[
\begin{array}{ccccc}
\hat{\mu}&&&&\\
\abacus(bbbhbbbhb,bbbhbbbhb,bbbhbbbhb,nnbhbbbhb,bbnhnbbhb,nnnhnnbhb,nnnhnnbhb,nnnhnbbhb,nnnhnnbhb,nnnhnnnhn,nnnhnnnhn)&
\sim&
\abacus(Ebbbhbbhbb,Ebbbhbbhbb,Ebbbhbbhbb,Cnnbhbbhbb,Ebbnhnbhbb,Cnnnhnbhbn,Ennnhnbhbn,Ennnhnbhbb,Cnnnhnbhbn,Ennnhnbhbn,Ennnhnnhnn)&
\sim&
\abacus(bbbhbbhbb,bbbhbbhbb,bbbhbbhbn,bnbhbbhbb,bbnhnbhbn,bnnhnbhbn,nnnhnbhbn,nnnhnbhbn,bnnhnbhbn,nnnhnbhbn,nnnhnnhnn)
\end{array}
\]

\[
\begin{array}{cccc}
&&&\tilde{\mu}\\
\sim&
\abacus(bbbEbhbbhb,bbbEbhbbhb,bbnCbhbbhb,bnb-bhbbhb,bbn+nhnbhb,bnnEnhnbhb,nnnEnhnbhb,nnnEnhnbhb,bnnEnhnbhb,nnnEnhnbhb,nnnEnhnnhn)&
\sim&
\abacus(bbbbhbbhb,bbbbhbbhb,bnbbhbbhb,bnbbhbbhb,bbnnhnbhb,bnnnhnbhb,nnnnhnbhb,nnnnhnbhb,bnnnhnbhb,nnnnhnbhb,nnnnhnnhn)
\end{array}
\]
\end{comment}

This completes the proof of Proposition \ref{[4:1]-pairs}.

As discussed in section \ref{outline of proof}, the combination of Proposition \ref{empty core weight 4}, Proposition \ref{[4:1]-pairs}, Proposition \ref{[4:2]-pairs}, Proposition \ref{[4:3]-pairs} and Proposition \ref{2 [4:1]-pairs} completes the proof of Theorem \ref{James Conjecture schur 4}.

\\

\textit{Acknowledgements}. This paper was written under the supervision of Kai Meng Tan at the National University of Singapore. The author would like to thank Prof Tan for his many helpful comments and guidance. The author is also grateful for the financial support given by the Academic Research Fund R-146-000-317-114 of NUS.

\appendix
\section{Figures and tables}
\begin{figure}[h]
  \caption{Exceptional partitions for $[3:1]$-pairs.}
 \label{figure exceptional 31}
  \includegraphics[scale=0.2]{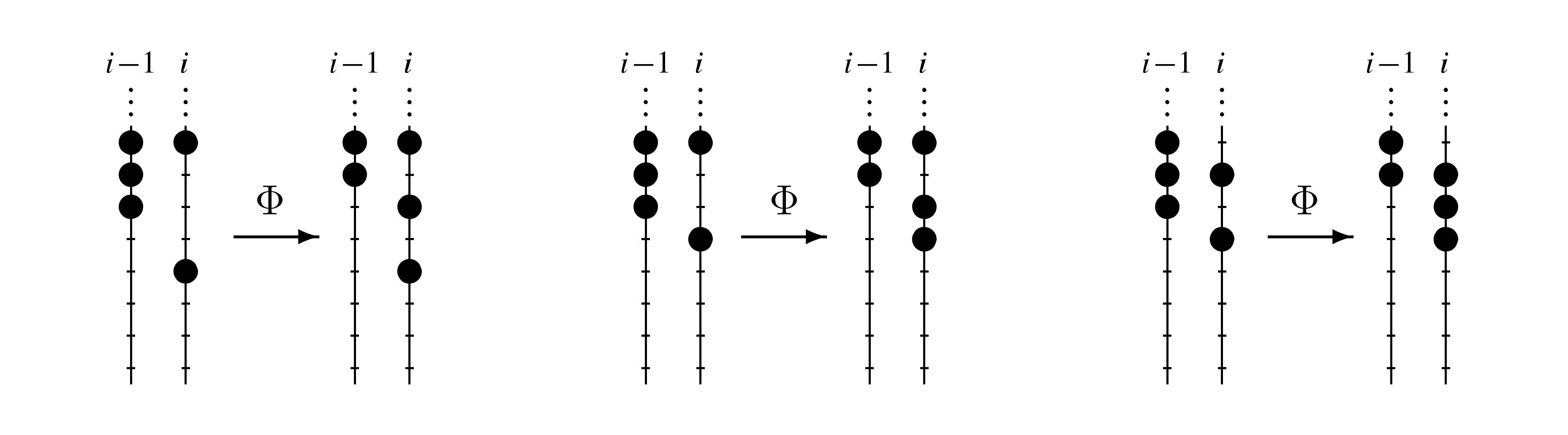}
\end{figure}

\begin{figure}[h]
  \caption{Exceptional partitions for $[3:2]$-pairs.}
 \label{figure exceptional 32}
 \includegraphics{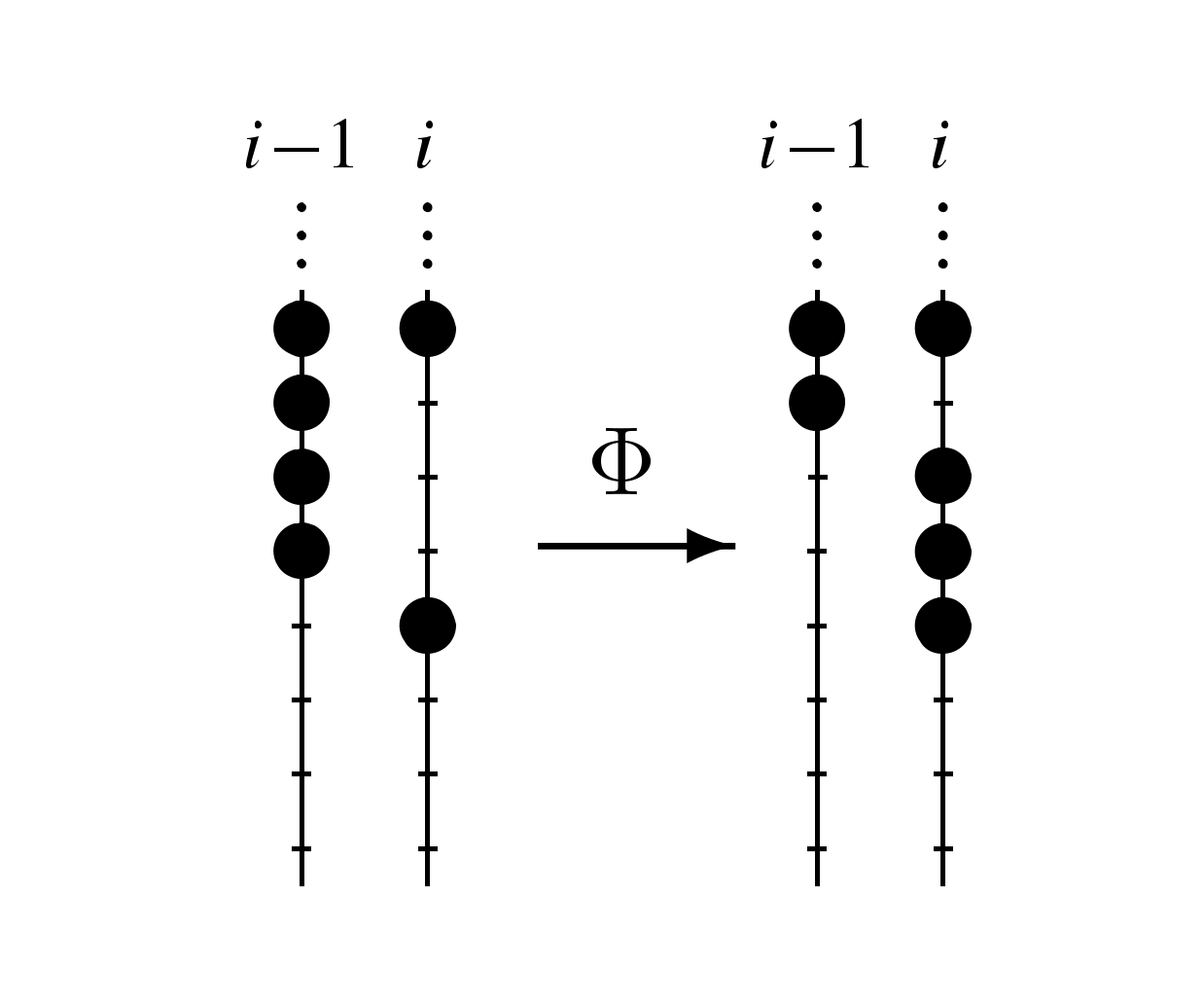}
\end{figure}

\begin{figure}[h]
  \caption{Exceptional partitions for $[4:k]$-pairs, $0<k<4$.}
 \label{figure exceptional 41}
  \includegraphics{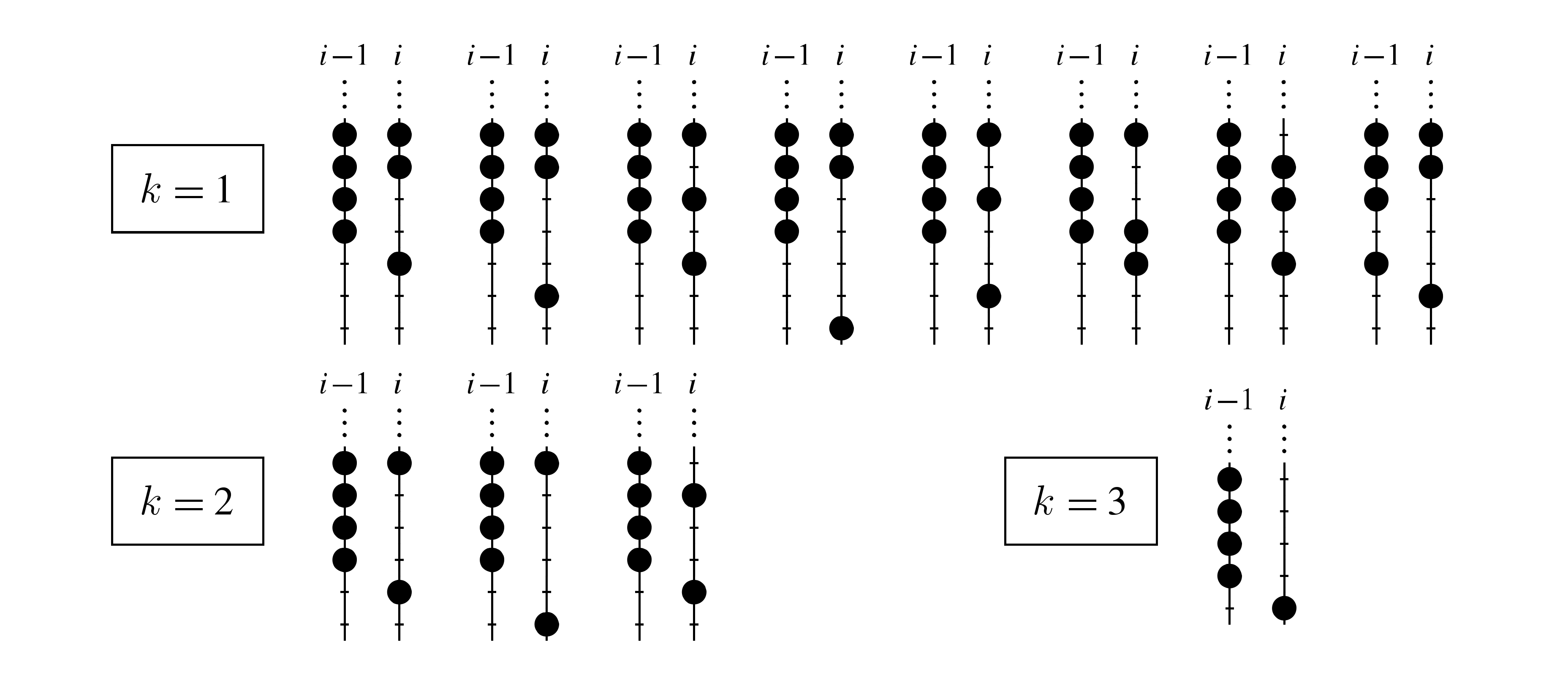}
\end{figure}

\begin{figure}[h]
\caption{$\nu^{0}=\langle 1_{2,1},a \mid 4^a,6,9^{e-a-1} \rangle$, $\hat{\mu}=\langle 0,1,a_2 \mid 4^a,6,9^{e-a-1} \rangle$}
\label{fig long}
\[
\begin{array}{ccccc}
\nu^{0}&&&&\\
\abacus(bbbhbbbhb,bbbhbbbhb,bnbhbbbhb,bbbhbbbhb,nnnhnbbhb,nbnhnnbhb,nnnhnbbhb,nnnhnnbhb,nnnhnnbhb,nnnhnnnhn,nnnhnnnhn)&
\sim&
\abacus(Ebbbhbbhbb,Ebbbhbbhbb,Ebnbhbbhbb,Ebbbhbbhbb,Cnnnhnbhbb,Cnbnhnbhbn,Ennnhnbhbb,Cnnnhnbhbn,Ennnhnbhbn,Ennnhnbhbn,Ennnhnnhnn)&
\sim&
\abacus(bbbhbbhbb,bbbhbbhbb,bnbhbbhbb,bbbhbbhbn,bnnhnbhbn,bbnhnbhbn,nnnhnbhbn,bnnhnbhbn,nnnhnbhbn,nnnhnbhbn,nnnhnnhnn)
\end{array}
\]
\[
\begin{array}{cccc}
&&&\tilde{\nu}\\
\sim&
\abacus(bbbEbhbbhb,bbbEbhbbhb,bnb-bhbbhb,bbn+bhbbhb,bnnEnhnbhb,bbnCnhnbhb,nnnEnhnbhb,bnnEnhnbhb,nnnEnhnbhb,nnnEnhnbhb,nnnEnhnnhn)&
\sim&
\abacus(bbbbhbbhb,bbbbhbbhb,bnbbhbbhb,bbnbhbbhb,bnnnhnbhb,bnbnhnbhb,nnnnhnbhb,bnnnhnbhb,nnnnhnbhb,nnnnhnbhb,nnnnhnnhn)
\end{array}
\]

\[
\begin{array}{ccccc}
\hat{\mu}&&&&\\
\abacus(bbbhbbbhb,bbbhbbbhb,bbbhbbbhb,nnbhbbbhb,bbnhnbbhb,nnnhnnbhb,nnnhnnbhb,nnnhnbbhb,nnnhnnbhb,nnnhnnnhn,nnnhnnnhn)&
\sim&
\abacus(Ebbbhbbhbb,Ebbbhbbhbb,Ebbbhbbhbb,Cnnbhbbhbb,Ebbnhnbhbb,Cnnnhnbhbn,Ennnhnbhbn,Ennnhnbhbb,Cnnnhnbhbn,Ennnhnbhbn,Ennnhnnhnn)&
\sim&
\abacus(bbbhbbhbb,bbbhbbhbb,bbbhbbhbn,bnbhbbhbb,bbnhnbhbn,bnnhnbhbn,nnnhnbhbn,nnnhnbhbn,bnnhnbhbn,nnnhnbhbn,nnnhnnhnn)
\end{array}
\]

\[
\begin{array}{cccc}
&&&\tilde{\mu}\\
\sim&
\abacus(bbbEbhbbhb,bbbEbhbbhb,bbnCbhbbhb,bnb-bhbbhb,bbn+nhnbhb,bnnEnhnbhb,nnnEnhnbhb,nnnEnhnbhb,bnnEnhnbhb,nnnEnhnbhb,nnnEnhnnhn)&
\sim&
\abacus(bbbbhbbhb,bbbbhbbhb,bnbbhbbhb,bnbbhbbhb,bbnnhnbhb,bnnnhnbhb,nnnnhnbhb,nnnnhnbhb,bnnnhnbhb,nnnnhnbhb,nnnnhnnhn)
\end{array}
\]
\end{figure}

\begin{table}[h]
\centering
\caption{Some cases $(\la,\mu)$ in section \ref{sec 3e}}
\label{table 3e}
\begin{tabular}{|p{3cm}|p{2.5cm}|} \hline
$\mu$ & $\lambda$   \\ \hline
$\langle    0_{1^{3}} \rangle $ & NIL \\ \hline
$\langle    0_{2},e-1      \rangle$, $e=2$ & $\langle    0_{1^{3}} \rangle $ \\ 
 & $\langle    1_{1^{3}} \rangle $  \\
 & $\langle 0_{1^{2}},1 \rangle $  \\
 & $\langle 0, 1_{1^{2}} \rangle $\\ \hline
$\langle 0_{1^{2}},1       \rangle$ & $\langle    0_{1^{3}} \rangle $  \\ \hline
$\langle 0,1,2 \rangle$, $e\ge3$ & $\langle    0_{1^{3}} \rangle $  \\
 & $\langle    1_{1^{3}} \rangle $  \\
 & $\langle 0_{1^{2}},1 \rangle $  \\
 & $\langle 0, 1_{1^{2}} \rangle $  \\ \hline
\end{tabular}
\end{table} 

\begin{table}[h]
\caption{Some cases $(\la,\mu)$ in section \ref{sec 4e}. For each $\mu$, the possiblities for $\la$ are listed in descending lexicographic order}
\label{table 4e 1}
\begin{tabular}{|p{5.5cm}|p{5cm}|p{2.5cm}|} \hline
$\mu$ & $\lambda$  & $d^{e}_{\lambda\mu}(v)$ \\ \hline
$\langle    0_{1^{4}} \rangle $ & NIL & \\ \hline
$\langle 0_{1^{3}},1 \rangle$ &  $\langle       0_{1^{4}}   \rangle$    &   \\      \hline
$\langle      0_{1^{2}}  , 1_{1^{2}}    \rangle$    &    $\langle       0_{1^{4}}    \rangle$       &         \\    \hline
$\langle     0_{1^{2}},1,2      \rangle$, $e\ge3$    &    $\langle   0,1_{1^{3}}       \rangle$        &           \\
 &  $\langle    0_{1^{3}}, 1      \rangle$         &           \\
 &  $\langle     1_{1^{4}}    \rangle$          &           \\
 &  $\langle  0_{1^{4}}       \rangle$         &           \\ \hline
$\langle 0_{2,1},1 \rangle$, $e=2$ &     $\langle      0  , 1_{1^{3}}   \rangle$   &  \\
& $\langle       0_{1^{3}}, 1    \rangle$  &      \\
& $\langle       1_{1^{4}}   \rangle$    &      \\
& $\langle       0_{1^{4}}   \rangle$    &      \\ \hline
$\mu^{1}:=\langle   0_{2},1,2        \rangle$, $e=3$    &     $\langle   0,1,2_{1^{2}}     \rangle$      &         \\
 &  $\langle   0,1_{1^{2}},2       \rangle$         &       \\
 &  $\langle  0_{1^{2}},1,2        \rangle$         &           \\
 &  $\langle   1_{1^{2}},2_{1^{2}}      \rangle$         &         \\
 &  $\langle    0_{1^{2}},2_{1^{2}}      \rangle$         &   $v$        \\
 &  $\langle   1,2_{1^{3}}       \rangle$         &   $v$       \\
 &  $\langle     1_{1^{3}},2     \rangle$         &           \\
 &  $\langle     0,2_{1^{3}}     \rangle$         &           \\
 &  $\lambda^{1}:=\langle    0_{1^{3}} ,2     \rangle$         &    $v^{3}$ \\
 &  $\langle    2_{1^{4}}     \rangle$         &         \\
 &  $\langle      0_{1^{2}},1_{1^{2}}    \rangle$         &           \\
 &  $\langle        0,1_{1^{3}}  \rangle$         &           \\
 &  $\langle    0_{1^{3}} ,1        \rangle$         &           \\
 &  $\langle   1_{1^{4}}      \rangle$         &           \\
 &  $\langle    0_{1^{4}}     \rangle$         &           \\ \hline
$\mu^{2}:=\langle     0,1,2,3      \rangle$, $e\ge4$    &   $\langle   0,1,2_{1^{2}}       \rangle$         &          \\
&  $\langle   0,1_{1^{2}},2       \rangle$         &          \\
 &  $\langle  0_{1^{2}},1,2        \rangle$         &           \\
 &  $\langle   1_{1^{2}},2_{1^{2}}      \rangle$         &          \\
 &  $\langle    0_{1^{2}},2_{1^{2}}      \rangle$         &     $d^{4}_{\hat{\lambda}\hat{\mu}}(v)=v$      \\
 &  $\langle   1,2_{1^{3}}       \rangle$         &     $d^{4}_{\hat{\lambda}\hat{\mu}}(v)=v$       \\
 &  $\langle     1_{1^{3}},2     \rangle$         &           \\
 &  $\langle     0,2_{1^{3}}     \rangle$         &            \\
 &  $\lambda^{2}:=\langle    0_{1^{3}} ,2     \rangle$         &      $d^{4}_{\hat{\lambda}\hat{\mu}}(v)=v^{3}$  \\
 &  $\langle    2_{1^{4}}     \rangle$         &          \\
 &  $\langle      0_{1^{2}},1_{1^{2}}    \rangle$         &           \\
 &  $\langle        0,1_{1^{3}}  \rangle$         &           \\
 &  $\langle    0_{1^{3}} ,1        \rangle$         &           \\
 &  $\langle   1_{1^{4}}      \rangle$         &           \\
 &  $\langle    0_{1^{4}}     \rangle$         &           \\ \hline
\end{tabular}
\end{table}

\begin{table}[h]
\caption{Some cases $(\la,\mu)$ in section \ref{sec 4e}. For each $\mu$, the possibilities for $\la$ are listed in descending lexicographic order}
\label{table 4e 2}
\begin{tabular}{|p{5.5cm}|p{5cm}|p{2.5cm}|} \hline
$\mu$ & $\lambda$  & $d^{e}_{\lambda\mu}(v)$ \\ \hline
$\mu^{3}:=\langle   0_{2},1_{2}  \rangle$, $e=2$  & $\langle       0_{2}, 1_{1^{2}}    \rangle$  &      \\
& $\langle       0_{2,1}, 1    \rangle$  &      \\
 &   $\lambda^{3}:=\langle       0_{2^{2}}    \rangle$        &  $v^{3}+v$       \\
& $\langle       0_{2,1^{2}}    \rangle$  &   $0$     \\
& $\langle       0_{1^{2}}, 1_{1^{2}}    \rangle$  &        \\
& $\langle       0, 1_{1^{3}}    \rangle$  &      \\
& $\langle       0_{1^{3}}, 1    \rangle$  &     \\
& $\langle       1_{1^{4}}    \rangle$  &     \\
& $\langle       0_{1^{4}}    \rangle$   &    \\ \hline
\end{tabular}
\end{table}

\begin{table}[htbp]
\footnotesize
  \caption{We have $a>3$, $\nu$ and $\hat{\mu}=\langle0,a_{3}\rangle$ lie in the block with the $\langle4^{a},6,9^{e-a-1}\rangle$ notation. See section \ref{subsub 1} for an explanation of how we obtained the last column.}
\centering
\addtolength{\leftskip} {-4cm}
\addtolength{\rightskip}{-2cm}
\setlength\tabcolsep{4pt}

    \begin{tabular}{c|c|c|c|c|c|cV{2.7}cV{2.7}c}
          &$\langle0,a_{3}\rangle$ & $\langle0_{2},a_{2}\rangle$ & $\langle(a-1)_{1^{2}},a_{2}\rangle$ & $\langle0,(a-1)_{2},a\rangle$ & $\langle(a-1)_{2,1},a\rangle$ & $\langle(a-2)_{2,1},a\rangle$ & $B(\nu,\hat{\mu})$ & $[W^{\nu}:L^{\hat{\mu}}]$ \\ \hline
    $\langle0,a_{3}\rangle$ &       &       &       &       &       &       &       & 1 \\ \hline
    $\langle0,a_{2,1}\rangle$ & 0 &       &       &       &       &       & 0     & 0 \\ \hline
    $\langle0_{2},a_{2}\rangle$ & 1     &       &       &       &       &       & 1     & 1 \\ \hline
    $\langle0_{4}\rangle$ & $-1$    & 1     &       &       &       &       & 0     & 0 \\ \hline
    $\langle0_{3},a\rangle$ & 1     & $-1$    &       &       &       &       & 0     & 0 \\ \hline
    $\langle0_{2},a_{1^{2}}\rangle$ &       & 0 &       &       &       &       & 0     & 0 \\ \hline
    $\langle0,a_{1^{3}}\rangle$ & 0 &       &       &       &       &       & 0     & 0 \\ \hline
    $\langle a_{1^{4}}\rangle$ &       &       &       &       &       &       & 0     & 0 \\ \hline
    $\langle0,a-1,a_{2}\rangle$ & \cellcolor[rgb]{ 1,  1,  1} $-1$ & \cellcolor[rgb]{ 1,  1,  1} 1 & \cellcolor[rgb]{ 1,  1,  1} & \cellcolor[rgb]{ 1,  1,  1} & \cellcolor[rgb]{ 1,  1,  1} & \cellcolor[rgb]{ 1,  1,  1} & \cellcolor[rgb]{ 1,  1,  1}0 & \cellcolor[rgb]{ 1,  1,  1}0 \\
     \vdots   & \cellcolor[rgb]{ 1,  1,  1}\vdots & \cellcolor[rgb]{ 1,  1,  1}\vdots & \cellcolor[rgb]{ 1,  1,  1} & \cellcolor[rgb]{ 1,  1,  1} & \cellcolor[rgb]{ 1,  1,  1} & \cellcolor[rgb]{ 1,  1,  1} & \cellcolor[rgb]{ 1,  1,  1}0 & \cellcolor[rgb]{ 1,  1,  1}0 \\ 
    $\langle0,1,a_{2}\rangle$ & \cellcolor[rgb]{ 1,  1,  1} $(-1)^{a-1}$ & \cellcolor[rgb]{ 1,  1,  1}$(-1)^{a}$ & \cellcolor[rgb]{ 1,  1,  1} & \cellcolor[rgb]{ 1,  1,  1} & \cellcolor[rgb]{ 1,  1,  1} & \cellcolor[rgb]{ 1,  1,  1} & \cellcolor[rgb]{ 1,  1,  1}0 & \cellcolor[rgb]{ 1,  1,  1}0 \\ \hline
     $\langle (a-1)_{1^{2}},a_{2}\rangle$ & \cellcolor[rgb]{ 1,  1,  1} & \cellcolor[rgb]{ 1,  1,  1}1 & \cellcolor[rgb]{ 1,  1,  1} & \cellcolor[rgb]{ 1,  1,  1} & \cellcolor[rgb]{ 1,  1,  1} & \cellcolor[rgb]{ 1,  1,  1} & \cellcolor[rgb]{ 1,  1,  1}1 & \cellcolor[rgb]{ 1,  1,  1}1 \\ 
   $\langle (a-2)_{1^{2}},a_{2}\rangle$ & \cellcolor[rgb]{ 1,  1,  1} & \cellcolor[rgb]{ 1,  1,  1}  $-1$   & \cellcolor[rgb]{ 1,  1,  1}1 & \cellcolor[rgb]{ 1,  1,  1} & \cellcolor[rgb]{ 1,  1,  1} & \cellcolor[rgb]{ 1,  1,  1} & \cellcolor[rgb]{ 1,  1,  1}0 & \cellcolor[rgb]{ 1,  1,  1}0 \\
     \vdots     & \cellcolor[rgb]{ 1,  1,  1} & \cellcolor[rgb]{ 1,  1,  1}\vdots &\cellcolor[rgb]{ 1,  1,  1}\vdots & \cellcolor[rgb]{ 1,  1,  1} & \cellcolor[rgb]{ 1,  1,  1} & \cellcolor[rgb]{ 1,  1,  1} & \cellcolor[rgb]{ 1,  1,  1}0 & \cellcolor[rgb]{ 1,  1,  1}0 \\
     $\langle1_{1^{2}},a_{2}\rangle$ & \cellcolor[rgb]{ 1,  1,  1} & \cellcolor[rgb]{ 1,  1,  1}$(-1)^{a}$ & \cellcolor[rgb]{ 1,  1,  1}$(-1)^{a-1}$ & \cellcolor[rgb]{ 1,  1,  1} & \cellcolor[rgb]{ 1,  1,  1} & \cellcolor[rgb]{ 1,  1,  1} & \cellcolor[rgb]{ 1,  1,  1}0 & \cellcolor[rgb]{ 1,  1,  1}0 \\ \hline
    $\langle0_{1^{2}},a_{2}\rangle$ & $(-1)^{a-1}$ & 0     & $(-1)^{a}$ &       &       &       & 0     & 0 \\ \hline
    $\langle0,a-1,a_{1^{2}}\rangle$ & \cellcolor[rgb]{ 1,  1,  1} & \cellcolor[rgb]{ 1,  1,  1} & \cellcolor[rgb]{ 1,  1,  1} & \cellcolor[rgb]{ 1,  1,  1} & \cellcolor[rgb]{ 1,  1,  1} & \cellcolor[rgb]{ 1,  1,  1} & \cellcolor[rgb]{ 1,  1,  1}0 & \cellcolor[rgb]{ 1,  1,  1}0 \\ 
    \vdots   & \cellcolor[rgb]{ 1,  1,  1} & \cellcolor[rgb]{ 1,  1,  1} & \cellcolor[rgb]{ 1,  1,  1} & \cellcolor[rgb]{ 1,  1,  1} & \cellcolor[rgb]{ 1,  1,  1} & \cellcolor[rgb]{ 1,  1,  1} & \cellcolor[rgb]{ 1,  1,  1}0 & \cellcolor[rgb]{ 1,  1,  1}0 \\ 
     $\langle0,1,a_{1^{2}}\rangle$ & \cellcolor[rgb]{ 1,  1,  1} & \cellcolor[rgb]{ 1,  1,  1} & \cellcolor[rgb]{ 1,  1,  1} & \cellcolor[rgb]{ 1,  1,  1} & \cellcolor[rgb]{ 1,  1,  1} & \cellcolor[rgb]{ 1,  1,  1} & \cellcolor[rgb]{ 1,  1,  1}0 & \cellcolor[rgb]{ 1,  1,  1}0 \\ \hline
     $\langle (a-1)_{1^{2}},a_{1^{2}}\rangle$ & \cellcolor[rgb]{ 1,  1,  1} & \cellcolor[rgb]{ 1,  1,  1} & \cellcolor[rgb]{ 1,  1,  1} & \cellcolor[rgb]{ 1,  1,  1} & \cellcolor[rgb]{ 1,  1,  1} & \cellcolor[rgb]{ 1,  1,  1} & \cellcolor[rgb]{ 1,  1,  1}0 & \cellcolor[rgb]{ 1,  1,  1}0 \\ 
     \vdots     & \cellcolor[rgb]{ 1,  1,  1} & \cellcolor[rgb]{ 1,  1,  1} & \cellcolor[rgb]{ 1,  1,  1} & \cellcolor[rgb]{ 1,  1,  1} & \cellcolor[rgb]{ 1,  1,  1} & \cellcolor[rgb]{ 1,  1,  1} & \cellcolor[rgb]{ 1,  1,  1}0 & \cellcolor[rgb]{ 1,  1,  1}0 \\ 
   $\langle1_{1^{2}},a_{1^{2}}\rangle$ & \cellcolor[rgb]{ 1,  1,  1} & \cellcolor[rgb]{ 1,  1,  1} & \cellcolor[rgb]{ 1,  1,  1} & \cellcolor[rgb]{ 1,  1,  1} & \cellcolor[rgb]{ 1,  1,  1} & \cellcolor[rgb]{ 1,  1,  1} & \cellcolor[rgb]{ 1,  1,  1}0 & \cellcolor[rgb]{ 1,  1,  1}0 \\ \hline
    $\langle0_{1^{2}},a_{1^{2}}\rangle$ &       &       &       &       &       &       & 0     & 0 \\ \hline
   $\langle0,(a-1)_{2},a\rangle$ & \cellcolor[rgb]{ 1,  1,  1}1 & \cellcolor[rgb]{ 1,  1,  1} & \cellcolor[rgb]{ 1,  1,  1} & \cellcolor[rgb]{ 1,  1,  1} & \cellcolor[rgb]{ 1,  1,  1} & \cellcolor[rgb]{ 1,  1,  1} & \cellcolor[rgb]{ 1,  1,  1}1 & \cellcolor[rgb]{ 1,  1,  1}1 \\ 
     $\langle0,(a-2)_{2},a\rangle$ & \cellcolor[rgb]{ 1,  1,  1} $-1$ & \cellcolor[rgb]{ 1,  1,  1} & \cellcolor[rgb]{ 1,  1,  1} & \cellcolor[rgb]{ 1,  1,  1}1 & \cellcolor[rgb]{ 1,  1,  1} & \cellcolor[rgb]{ 1,  1,  1} & \cellcolor[rgb]{ 1,  1,  1}0 & \cellcolor[rgb]{ 1,  1,  1}0 \\ 
     \vdots     & \cellcolor[rgb]{ 1,  1,  1}\vdots & \cellcolor[rgb]{ 1,  1,  1} & \cellcolor[rgb]{ 1,  1,  1} & \cellcolor[rgb]{ 1,  1,  1}\vdots & \cellcolor[rgb]{ 1,  1,  1} & \cellcolor[rgb]{ 1,  1,  1} & \cellcolor[rgb]{ 1,  1,  1}0 & \cellcolor[rgb]{ 1,  1,  1}0 \\
     $\langle0,1_{2},a\rangle$ & \cellcolor[rgb]{ 1,  1,  1}$(-1)^{a}$ & \cellcolor[rgb]{ 1,  1,  1} & \cellcolor[rgb]{ 1,  1,  1} & \cellcolor[rgb]{ 1,  1,  1}$ (-1)^{a-1}$ & \cellcolor[rgb]{ 1,  1,  1} & \cellcolor[rgb]{ 1,  1,  1} & \cellcolor[rgb]{ 1,  1,  1}0 & \cellcolor[rgb]{ 1,  1,  1}0 \\ \hline
     \cellcolor[gray]{0.8} $\langle (a-1)_{2,1},a\rangle$ & \cellcolor[rgb]{ 1,  1,  1} & \cellcolor[rgb]{ 1,  1,  1} & \cellcolor[rgb]{ 1,  1,  1}1 & \cellcolor[rgb]{ 1,  1,  1}1 & \cellcolor[rgb]{ 1,  1,  1} & \cellcolor[rgb]{ 1,  1,  1} & \cellcolor[gray]{0.8}2 & \cellcolor[gray]{0.8}1 \\ 
    $\langle (a-2)_{2,1},a\rangle$ & \cellcolor[rgb]{ 1,  1,  1} & \cellcolor[rgb]{ 1,  1,  1} & \cellcolor[rgb]{ 1,  1,  1} & \cellcolor[rgb]{ 1,  1,  1} & \cellcolor[rgb]{ 1,  1,  1}1 & \cellcolor[rgb]{ 1,  1,  1} & \cellcolor[rgb]{ 1,  1,  1}1 & \cellcolor[rgb]{ 1,  1,  1}1 \\ 
  $\langle (a-3)_{2,1},a\rangle$ & \cellcolor[rgb]{ 1,  1,  1} & \cellcolor[rgb]{ 1,  1,  1} & \cellcolor[rgb]{ 1,  1,  1} & \cellcolor[rgb]{ 1,  1,  1} & \cellcolor[rgb]{ 1,  1,  1} $-1$ & \cellcolor[rgb]{ 1,  1,  1} 1 & \cellcolor[rgb]{ 1,  1,  1}0 & \cellcolor[rgb]{ 1,  1,  1}0 \\ 
  \vdots     & \cellcolor[rgb]{ 1,  1,  1} & \cellcolor[rgb]{ 1,  1,  1} & \cellcolor[rgb]{ 1,  1,  1} & \cellcolor[rgb]{ 1,  1,  1} & \cellcolor[rgb]{ 1,  1,  1}\vdots& \cellcolor[rgb]{ 1,  1,  1}\vdots & \cellcolor[rgb]{ 1,  1,  1}0 & \cellcolor[rgb]{ 1,  1,  1}0 \\ 
 $\langle1_{2,1},a\rangle$ & \cellcolor[rgb]{ 1,  1,  1} & \cellcolor[rgb]{ 1,  1,  1} & \cellcolor[rgb]{ 1,  1,  1}0 & \cellcolor[rgb]{ 1,  1,  1} & \cellcolor[rgb]{ 1,  1,  1}$(-1)^{a-1}$ & \cellcolor[rgb]{ 1,  1,  1}$(-1)^{a}$ & \cellcolor[rgb]{ 1,  1,  1}0 & \cellcolor[rgb]{ 1,  1,  1}0 \\ 
   $\langle0_{2,1},a\rangle$ & \cellcolor[rgb]{ 1,  1,  1} & \cellcolor[rgb]{ 1,  1,  1} & \cellcolor[rgb]{ 1,  1,  1} & \cellcolor[rgb]{ 1,  1,  1} & \cellcolor[rgb]{ 1,  1,  1}$(-1)^{a}$ & \cellcolor[rgb]{ 1,  1,  1}$(-1)^{a-1}$ & \cellcolor[rgb]{ 1,  1,  1}0 & \cellcolor[rgb]{ 1,  1,  1}0 \\ \hline
    \end{tabular}
  \label{4:1 1}
\end{table}
\pagebreak

% Table generated by Excel2LaTeX from sheet 'Sheet1'
\begin{table}[htbp]
\footnotesize
  \centering

\setlength\tabcolsep{3pt}
  \caption{We have $a=3$, $\nu$ and $\hat{\mu}=\langle0,3_{3}\rangle$ lie in the block with the $\langle4^{a},6,9^{e-a-1}\rangle$ notation. See section \ref{subsub 1} for an explanation of how we obtained the last column.}
    \begin{tabular}{c|c|c|c|c|c|cV{2.7}cV{2.7}c}
          & $\langle0,3_{3}\rangle$ & $\langle0_{2},3_{2}\rangle$ &$\langle2_{1^{2}},3_{2}\rangle$ & $\langle0,2_{2},3\rangle$ & $\langle2_{2,1},3\rangle$ & $\langle1_{2,1},a\rangle$ & $B(\nu,\hat{\mu})$ & $[W^{\nu}:L^{\hat{\mu}}]$ \\ \hline
   $\langle 0,a_{3} \rangle$ &       &       &       &       &       &       &       & 1 \\ \hline
    $\langle 0,a_{2,1}  \rangle$& 0 &       &       &       &       &       & 0     & 0 \\ \hline
   $\langle  0_{2},3_{2}  \rangle$& 1     &       &       &       &       &       & 1     & 1 \\ \hline
    $\langle 0_{4} \rangle$& $-1$     & 1     &       &       &       &       & 0     & 0 \\ \hline
   $\langle  0_{3},a \rangle$& 1     & $-1$    &       &       &       &       & 0     & 0 \\ \hline
    $\langle 0_{2},a_{1^{2}}  \rangle$&       & 0 &       &       &       &       & 0     & 0 \\ \hline
   $\langle  0,a_{1^{3}}  \rangle$& 0 &       &       &       &       &       & 0     & 0 \\ \hline
   $\langle  a_{1^{4}}  \rangle$&       &       &       &       &       &       & 0     & 0 \\ \hline
  $\langle   0,2,3_{2}  \rangle$& $-1$    & 1     &       &       &       &       & 0     & 0 \\ \hline
  $\langle   0,1,3_{2}  \rangle$& 1     & $-1$    &       &       &       &       & 0     & 0 \\ \hline
  $\langle   2_{1^{2}},a_{2}  \rangle$&       & 1     &       &       &       &       & 1     & 1 \\ \hline
  $\langle   1_{1^{2}},a_{2}  \rangle$&       &$-1$    & 1     &       &       &       & 0     & 0 \\ \hline
  $\langle   0_{1^{2}},a_{2}  \rangle$& 1     & 0     &$-1$    &       &       &       & 0     & 0 \\ \hline
  $\langle   0,2,3_{1^{2}} \rangle$ &       &       &       &       &       &       & 0     & 0 \\ \hline
  $\langle   0,1,3_{1^{2}} \rangle$ &       &       &       &       &       &       & 0     & 0 \\ \hline
   $\langle  2_{1^{2}},3_{1^{2}} \rangle$ &       &       &       &       &       &       & 0     & 0 \\ \hline
   $\langle  1_{1^{2}},3_{1^{2}} \rangle$ &       &       &       &       &       &       & 0     & 0 \\ \hline
  $\langle   0_{1^{2}},3_{1^{2}}  \rangle$&       &       &       &       &       &       & 0     & 0 \\ \hline
   $\langle  0,2_{2},3  \rangle$& 1     &       &       &       &       &       & 1     & 1 \\ \hline
  $\langle   0,1_{2},3 \rangle$ & $-1$   &       &       & 1     &       &       & 0     & 0 \\ \hline
 \cellcolor[gray]{0.8}  $\langle  2_{2,1},3  \rangle$&       &       & 1     & 1     &       &       & \cellcolor[gray]{0.8}2 & \cellcolor[gray]{0.8}1 \\ \hline
  $\langle   1_{2,1},a \rangle$ &       &       &       &       & 1     &       & 1     & 1 \\ \hline
   $\langle  0_{2,1},a \rangle$ &       &       &       &       &$-1$    & 1     & 0     & 0 \\ \hline
    \end{tabular}
  \label{4:1 2}
\end{table}
\pagebreak

% Table generated by Excel2LaTeX from sheet 'Sheet1'
\begin{table}[htbp]
\footnotesize
  \centering
  \caption{We have $a=2$, $\nu$ and $\hat{\mu}=\langle  0,2_{3}  \rangle$ lie in the block with the $\langle4^{a},6,9^{e-a-1}\rangle$ notation. See section \ref{subsub 1} for an explanation of how we obtained the last column.}
    \begin{tabular}{c|c|c|c|c|cV{2.7}cV{2.7}c}
          & $\langle 0,2_{3} \rangle$ & $\langle 0_{2},2_{2} \rangle$ &$\langle 1_{1^{2}},2_{2} \rangle$ & $\langle 0,1_{2},2 \rangle$ & $\langle1_{2,1},2\rangle$ & $B(\nu,\hat{\mu})$& $[W^{\nu}:L^{\hat{\mu}}]$ \\ \hline
   $\langle  0,2_{3}  \rangle$&       &       &       &       &       &       & 1 \\ \hline
   $\langle  0,2_{2,1} \rangle$ & 0 &       &       &       &       & 0     & 0 \\ \hline
    $\langle 0_{2},2_{2}  \rangle$& 1     &       &       &       &       & 1     & 1 \\ \hline
    $\langle 0_{4}  \rangle$& $-1$    & 1     &       &       &       & 0     & 0 \\ \hline
  $\langle   0_{3},2  \rangle$& 1     & $-1$    &       &       &       & 0     & 0 \\ \hline
  $\langle   0_{2},2_{1^{2}} \rangle$ &       & 0 &       &       &       & 0     & 0 \\ \hline
  $\langle   0,2_{1^{3}}  \rangle$& 0 &       &       &       &       & 0     & 0 \\ \hline
  $\langle   2_{1^{4}}  \rangle$&       &       &       &       &       & 0     & 0 \\ \hline
   $\langle  0,1,2_{2}  \rangle$& $-1$    & 1     &       &       &       & 0     & 0 \\ \hline
  $\langle   1_{1^{2}},2_{2}  \rangle$&       & 1     &       &       &       & 1     & 1 \\ \hline
  $\langle   0_{1^{2}},2_{2} \rangle$ & $-1$ & 0     & $1$ &       &       & 0     & 0 \\ \hline
    $\langle 0,1,2_{1^{2}} \rangle$ &       &       &       &       &       & 0     & 0 \\ \hline
   $\langle  1_{1^{2}},2_{1^{2}} \rangle$ &       &       &       &       &       & 0     & 0 \\ \hline
   $\langle  0_{1^{2}},2_{1^{2}} \rangle$ &       &       &       &       &       & 0     & 0 \\ \hline
   $\langle  0,1_{2},2 \rangle$ & 1     &       &       &       &       & 1     & 1 \\ \hline
\cellcolor[gray]{0.8}  $\langle   1_{2,1},2  \rangle$&       &       & 1     & 1     &       & \cellcolor[gray]{0.8}2 &\cellcolor[gray]{0.8}1 \\ \hline
  $\langle   0_{2,1},2  \rangle$&       &       &       &       & 1     & 1     & 1 \\ \hline
    \end{tabular}
  \label{4:1 3}
\end{table}

% Table generated by Excel2LaTeX from sheet 'Sheet1'
\begin{table}[htbp]
\footnotesize
  \centering
  \caption{We have $a\ge2$, $e-a\ge1$, $\nu$ and $\hat{\mu}=\langle0,1,a_{2}\rangle$ lie in the block with the $\langle4^{a},6,9^{e-a-1}\rangle$ notation. See section \ref{subsub 2} for an explanation of how we obtained the last column.}
    \begin{tabular}{c|c|c|c|cV{2.7}cV{2.7}c}
          & $\langle 0,1,a_{2} \rangle$ & $\langle 0,1_{2},a \rangle$ & $\langle 1_{1^{2}},a_{2} \rangle$ & $\langle 1_{2,1},a \rangle$ & $B(\nu,\hat{\mu})$& $[W^{\nu}:L^{\hat{\mu}}]$ \\ \hline
    $\langle 0,1,a_{2} \rangle$ &       &       &       &       &      & 1 \\ \hline
    $\langle 0,1,a_{1^{2}}  \rangle$& 0     &       &       &       & 0     & 0 \\ \hline
   $\langle  0,1_{2},a  \rangle$& 1     &       &       &       & 1     & 1 \\ \hline
   $\langle  1_{1^{2}},a_{2}  \rangle$& 1     &       &       &       & 1     & 1 \\ \hline
   $\langle  1_{1^{2}},a_{1^{2}} \rangle$ &       &       &       &       & 0     & 0 \\ \hline
  \cellcolor[gray]{0.8} $\langle  1_{2,1},a  \rangle$&       & 1     & 1     &       & \cellcolor[gray]{0.8}2 &\cellcolor[gray]{0.8}1 \\ \hline
   $\langle  0_{2},1,a \rangle$ & $-1$    & 1     &       &       & 0     & 0 \\ \hline
   $\langle  0_{1^{2}},a_{2}  \rangle$& $-1$    &       & 1     &       & 0     & 0 \\ \hline
   $\langle  0_{2,1},a  \rangle$&       &       &       & 1     & 1     & 1 \\ \hline
    \end{tabular}
  \label{4:1 6}
\end{table}


\begin{thebibliography}{99}
\bibitem{Ariki}
S. Ariki, On the decomposition numbers of the Hecke algebra of $G(m, 1, n)$,
\emph{J. Math. Kyoto Univ.} \textbf{36 (4)} (1996) 789-808.

\bibitem{brundan}
J. Brundan,
Modular branching rules and the Mullineux map for Hecke algebras of type A,
\emph{Proceedings of the London Mathematical Society}, \textbf{77(3)} (1998), 551–81.

%\bibitem{branching}
%J. Brundan, A. Kleshchev,
%Representation theory of the symmetric groups and their double covers,
%in: \emph{Groups,
%Combinatorics \& Geometry Durham 2001} (pp. 31–53), World Sci. Publishing, River Edge, NJ, 2003.

\bibitem{filtrations}
J. Chuang and K. M. Tan,
Filtrations in Rouquier Blocks of Symmetric Groups and Schur Algebras,
\emph{Proceedings of the London Mathematical Society}, Volume \textbf{86}, Issue \textbf{3} (2003),
685-706.


\bibitem{dipper 1}
R. Dipper and G. James, Representations of Hecke algebras of general linear groups, \emph{Proceedings of the London Mathematical
Society}, (3) \textbf{52} (1986), 20–52.

\bibitem{dipper 2}
R. Dipper and G. James, $q$-tensor space and $q$-Weyl modules, \emph{Trans. Amer. Math. Soc.}, \textbf{327} (1991),
251–82.

\bibitem{donkin}
S. Donkin, 
The $q$-Schur Algebra,
\emph{London Mathematical Society Lecture Note Series} \textbf{253},
Cambridge University Press, 1998.

\bibitem{another runner removal theorem}
M. Fayers,
Another runner removal theorem for $v$-decomposition numbers of Iwahori-Hecke
algebras and $q$-Schur algebras.
\emph{Journal of Algebra} \textbf{310} (2007), 396–404.

\bibitem{wt 3}
M. Fayers,
Decomposition numbers for weight three blocks of symmetric groups and Iwahori–Hecke algebras,
\emph{Trans. Amer. Math. Soc.} \textbf{360(3)} (2008), 1341-1376.

\bibitem{wt 4}
M. Fayers,
James’s Conjecture holds for weight four blocks of Iwahori–Hecke algebras,
\emph{Journal of Algebra} \textbf{317} (2007), 593–633.

\bibitem{graham}
J. J. Graham and G. I. Lehrer,
Cellular Algebras,
\emph{Invent. Math.} \textbf{123}
(1994), 1-34.

\bibitem{glnq}
G. James,
The decomposition matrices of $GL_{n}(q)$ for $n\le10$,
\emph{Proc. London Math. Soc.} \textbf{60(3)} (1990), 225–265.

\bibitem{rouquier}
G. James, S. Lyle and A. Mathas,
Rouquier blocks, 
\emph{Math. Z.} \textbf{252} (2006), 511–531.

\bibitem{Jantzen Schaper}
G. James and A. Mathas,
A $q$-analogue of the Jantzen-Schaper theorem,
\emph{Proc. London Math. Soc.} \textbf{74(3)} (1997), 241-274.

\bibitem{runner removal}
G. James and A. Mathas, 
Equating decomposition numbers for different primes, 
\emph{J. Algebra} \textbf{258} (2002), 599–614.

\bibitem{llt}
A. Lascoux, B. Leclerc and J.-Y. Thibon,
Hecke algebras at roots of unity and crystal bases of quantum affine algebras,
\emph{Comm. Math. Phys.} \textbf{181}
(1996), 205–263.

\begin{comment}
\bibitem{LT}
Leclerc and J.-Y. Thibon,
Canonical bases of $q$-deformed Fock spaces,
\emph{International Mathematics Research Notices} \textbf{9}
(1996), 447-456.
\end{comment}

\bibitem{wt 5}
A.Y.R. Low,
Adjustment matrices for the principal block of the Iwahori-Hecke algebra $\mathcal{H}_{5e}$, \emph{Journal of Algebra}, \textbf{563}
(2020), 274-291.

\bibitem{mathas book}
A. Mathas,
Iwahori–Hecke Algebras and Schur Algebras of the Symmetric Group,
\emph{University Lecture Series} \textbf{15},
American Mathematical Society, 1999.

\bibitem{esign}
 A. O. Morris and J. B. Olsson, 
On $p$-quotients for spin characters.
\emph{Journal of Algebra} \textbf{119},
(1988): 51–82.

\bibitem{Richards}
M. Richards,
Some decomposition numbers for Hecke algebras of general linear groups,
\emph{Math. Proc. Cambridge Philos. Soc.} \textbf{119} (1996), 383-402.

\bibitem{Scopes}
J. C. Scopes, 
Cartan matrices and Morita equivalence for blocks of the symmetric groups, 
\emph{Journal of Algebra} \textbf{142} (1991), 441–455.

\bibitem{Ac duality}
S. Schroll and K. M. Tan,
Weight 2 blocks of general linear groups and modular Alvis–Curtis duality,
\emph{International Mathematics Research Notices}, Vol. \textbf{2007}
(2007), Article ID \textbf{rnm130}.

\bibitem{parities}
K. M. Tan,
Parities of $v$-decomposition numbers and an application to symmetric group algebras.
(2006), preprint

\bibitem{beyond}
K.M. Tan,
Beyond Rouquier partitions,
\emph{Journal of Algebra} \textbf{321}
(2009), 248-263.

\bibitem{varagnolo}
M. Varagnolo and E. Vasserot,
On the decomposition numbers of the quantized Schur algebra,
\emph{Duke Mathematical Journal} \textbf{100}
(1999), 267-297.

\bibitem{Williamson}
G. Williamson,
Schubert calculus and torsion explosion,
\emph{J. Amer. Math. Soc.} \textbf{30} (2017).


\end{thebibliography}
\end{document}